\documentclass[moor,nonblindrev]{template}%
\usepackage{multirow}
\usepackage{algorithm,algorithmicx,algpseudocode}
\usepackage{amsmath}
\usepackage{amssymb}
\usepackage{natbib,xspace}
\usepackage[]{subcaption}
\usepackage[colorlinks=true,breaklinks=true,bookmarks=true,urlcolor=blue,
citecolor=blue,linkcolor=blue,bookmarksopen=false,draft=false]{hyperref}%
\usepackage{amsfonts}%
\usepackage{graphicx}
\providecommand{\U}[1]{\protect\rule{.1in}{.1in}}
\NatBibNumeric
\def\bibfont{\small}

\bibpunct[, ]{[}{]}{,}{n}{}{,}
\def\EMAIL#1{\href{mailto:#1}{#1}}

\TheoremsNumberedThrough
\EquationsNumberedThrough
\MANUSCRIPTNO{MOR-2018-281}
\begin{document}


\RUNAUTHOR{Blanchet, Murthy, and Zhang}

\RUNTITLE{Optimal Transport Based Distributionally
Robust Optimization}

\TITLE{Optimal Transport Based Distributionally Robust Optimization: Structural Properties and Iterative Schemes}

\ARTICLEAUTHORS{\AUTHOR{Jose Blanchet}
\AFF{Management Science and Engineering, Stanford University, Stanford, California 94305;  \EMAIL{jose.blanchet@stanford.edu}}
\AUTHOR{Karthyek Murthy}
\AFF{Engineering Systems and Design, Singapore University of
	Technology \& Design, Singapore 487372;  \EMAIL{karthyek\_murthy@sutd.edu.sg}}
\AUTHOR{Fan Zhang}
\AFF{Management Science and Engineering, Stanford University, Stanford, California 94305;  \EMAIL{fzh@stanford.edu}}
}

\ABSTRACT{We consider optimal transport based distributionally robust
optimization (DRO) problems with locally strongly convex transport
cost functions and affine decision rules. Under conventional
convexity assumptions on the underlying loss function, we obtain
structural results about the value function, the optimal policy, and
the worst-case optimal transport adversarial model. These results
expose a rich structure embedded in the DRO problem (e.g. strong
convexity even if the non-DRO problem was not strongly convex, a
suitable scaling of the Lagrangian for the DRO constraint,
etc. which are crucial for the design of efficient algorithms). As a
consequence of these results, one can develop efficient optimization
procedures which have the same sample and iteration complexity as a
natural non-DRO benchmark algorithm such as stochastic gradient
descent.

}

\MSCCLASS{Primary: 90C15; 	Secondary: 65K05, 90C47.}
\ORMSCLASS{Programming: Stochastic; Programming: Nonlinear Algorithms; Mathematics: Convexity.}




\maketitle
\section{Introduction.}

In this paper we study the distributionally robust optimization (DRO) version
of stochastic optimization models with linear decision rules of the form
\begin{equation}
\inf_{\beta\in B}E_{P^{\ast}}[\ell(\beta^{T}X)],\label{Stoch_Opt_Model}%
\end{equation}
where $E_{P^{\ast}}[ \cdot]$ represents the expectation operator associated to
the probability model $P^{\ast}$, which describes the random element
$X\in\mathbb{R}^{d}$. The decision (or optimization) variable $\beta$ is
assumed to take values on a convex set $B\subseteq\mathbb{R}^{d}$, and the
loss function $\ell:\mathbb{R\rightarrow R}$ is assumed to satisfy certain
convexity and regularity assumptions discussed in the sequel. The formulation
also includes affine decision rules by simply redefining $X$ by $(X,\mathbf{1}%
) $.

Stochastic optimization problems such as (\ref{Stoch_Opt_Model}) include
standard formulations in important Operations Research (OR) and Machine
Learning (ML) applications, including newsvendor models, portfolio
optimization via utility maximization, and a large portion of the most
conventional generalized linear models in the setting of statistical learning problems.

The corresponding DRO\ version of (\ref{Stoch_Opt_Model}) takes the form
\begin{equation}
\inf_{\beta\in B}\sup_{P\in\mathcal{U}_{\delta}\left(  P_{0}\right)  }%
E_{P}[\ell(\beta^{T}X)],\label{DRO_B0}%
\end{equation}
where $\mathcal{U}_{\delta}\left(  P_{0}\right)  $ is a so-called
distributional uncertainty region \textquotedblleft centered\textquotedblright%
\ around some benchmark model, $P_{0}$, which may be data-driven (for example,
an empirical distribution) and $\delta>0$ parameterizes the size of the
distributional uncertainty. Precisely, we assume that $P_{0}$ is an arbitrary
distribution with suitably bounded moments.

The DRO counterpart of (\ref{Stoch_Opt_Model}) is motivated by the fact that
the underlying model $P^{\ast}$ generally is unknown, while the benchmark
model, $P_{0}$, is typically chosen to be a tractable model which in principle
should retain as much model fidelity as possible (i.e. $P_{0}$ should at least
capture the most relevant features present in $P^{\ast}$). However, simply
replacing $P^{\ast}$ by $P_{0}$ in the formulation (\ref{Stoch_Opt_Model}) may
result in the selection of a decision, $\beta_{0}$, which significantly
under-performs in actual practice, relative to the optimal decision for
the actual problem (based on $P^{\ast}$).

The DRO formulation (\ref{DRO_B0}) introduces an adversary (represented by the
inner $\sup$) which explores the implications of any decision $\beta$ as the
benchmark model $P_{0}$ varies within $\mathcal{U}_{\delta}\left(
P_{0}\right)  $. The adversary should be seen as a powerful modeling tool
whose goal is to explore the impact of potential decisions in the phase of
distributional uncertainty. The DRO\ formulation then prescribes a choice
which minimizes the worst case expected cost induced by the models in the
distributional uncertainty region.

An important ingredient in the DRO\ formulation is the description of
the distributional uncertainty region $\mathcal{U}_{\delta}(P_{0})$. In recent years, there has been significant
interest in distributional uncertainty regions satisfying
\[
\mathcal{U}_{\delta}(P_{0})=\{P:\mathcal{W}(P_{0},P)\leq\delta\},
\]
where $\mathcal{W}(P_{0},P)$ is a Wasserstein distance (see, for example,
\cite{blanchet2016robust,blanchet_quantifying_2016,wolfram2018,gao2016distributionally,gao2018robust,MohajerinEsfahani2017,shafieezadeh-abadeh_distributionally_2015,sinha2018certifiable,volpi2018generalizing,7938642,ZHAO2018262}
and references therein). 

The Wasserstein distance is a particular case of optimal transport
discrepancies, which we will review momentarily. A general optimal transport
discrepancy computes the cheapest cost of transporting the mass of $P_{0}$ to
the mass of $P$ so that a unit of mass transported from position $x$ to
position $y$ is measured according to a transportation cost function,
$c\left(  \cdot\right)  $. The definition of $\mathcal{W}(P_{0},P)$ requires
that $c\left(  \cdot\right)  $ be a norm or a distance, but this is not
necessary and endowing modelers with increased flexibility in choosing
$c\left(  \cdot\right)  $ is an important part of our motivation.

The use of the Wasserstein distance is closely related to norm-regularization
and DRO formulations have been shown to recover approximately and exactly a
wide range of machine learning estimators; see, for example,
\cite{blanchet2016robust,gao2017wasserstein,shafieezadeh-abadeh_distributionally_2015,ShafieezadehAbadeh:231966}%
. These and some other applications of the DRO formulation (\ref{DRO_B0})
based on Wasserstein distance lead to a reduction from (\ref{DRO_B0}) back to
a problem of the form (\ref{Stoch_Opt_Model}), in which the objective loss
function is modified by adding a regularization penalty expressed in terms of
the norm of $\beta$ and a regularization penalty parameter as an explicit
function of $\delta$. 

We stress that in many of these settings, particularly the cases in which
$\ell\left(  \cdot\right)  $ is Lipschitz and convex, the worst-case
distribution is degenerate (i.e. it is realized by moving infinitesimally
small mass towards infinity or moving no mass at all). 

We will enable efficient algorithms which can be applied to more flexible cost
functions $c\left(  \cdot\right)  $ and losses $\ell\left(  \cdot\right)  $ in
order to induce adversarial distributions which can be both informed by side
information and endowed with meaningful interpretations. 

For other special cases which are amenable to either analytical solutions or
software implementations, $P_{0}$ is either assumed to have a special
structure (e.g. Gaussian distribution, as in \cite{nguyen2018distributionally}) or ultimately requires robust
optimization formulations which require $P_{0}$ to have finite support; see
\cite{wolfram2018,chen2018adaptive,luo2017decomposition,MohajerinEsfahani2017,shafieezadeh-abadeh_distributionally_2015,ShafieezadehAbadeh:231966,weijun,ZHAO2018262}

As we shall see, our analysis will enable the application of stochastic
gradient descent algorithms to approximate the solution to (\ref{DRO_B0}) and
which are applicable to cases in which $P_{0}$ has unbounded support (under
suitable moment constraints). Moreover, by enabling the use of stochastic
gradient descent algorithm we open the door to further research on accelerated
stochastic gradient methods. In this paper, we shall focus on providing
stochastic gradient descent implementation to demonstrate the direct
application of our structural results.

We mention \cite{sinha2018certifiable}, in which relaxed Wasserstein DRO
formulations are explored in the context of certifying robustness in deep
neural networks. The stochastic gradient descent-type employed in
\cite{sinha2018certifiable} is similar to the ones that we discuss in Section
\ref{Sec-SA-Algo}. Nevertheless, these algorithms are designed for a fixed value of the
dual parameter (which we call $\lambda$), chosen to be large. Our analysis
suggests that rescaling $\lambda$, so that ($\lambda=O\left(\delta
^{-1/2}\right)  $) may enhance performance, even in the case of the more
general type of losses considered in \cite{sinha2018certifiable}. The impact
of this type of rescaling in terms of performance guarantees for computational
algorithms has not been studied in the literature and we believe that our
analysis could prove useful in future studies. Additional discussion on the
rescaling is given at the end of Section \ref{subsect_rates_of_convergence}.

The challenge in our study lies in the inner maximization (\ref{DRO_B0}),
which is not easy to perform and its properties, parametrically as a function
of both $\beta$ and $\delta$, are non-trivial to analyze. So, much of our
effort will go into understanding these properties. But before we describe our
results, we first describe a flexible class of models for distributional
uncertainty sets, $\mathcal{U}_{\delta}\left(  P_{0}\right)  $.

\textbf{A description of the distributional uncertainty region
$\mathcal{U}_{\delta}(P_{0})$.} We focus on DRO formulations
based on extensions of the Wasserstein distance, called optimal transport
discrepancies. Formally, an optimal transport discrepancy between
distributions $P$ and $P_{0}$ with respect to the (lower semicontinuous) cost
function $c:\mathbb{R}^{d}\times\mathbb{R}^{d}\rightarrow\lbrack0,\infty]$ is
defined as follows.

First, let $\mathcal{P}(\mathbb{R}^{d}\times\mathbb{R}^{d}) $ be the set of
Borel probability measures on $\mathbb{R}^{d}\times\mathbb{R}^{d}$. So, for
any $X\in\mathbb{R}^{d}$ and $X^{\prime}\in\mathbb{R}^{d}$ random elements
living on the same probability space there exists $\pi\in\mathcal{P}(
\mathbb{R}^{d}\times\mathbb{R}^{d}) $ which governs the joint distribution of
$( X,X^{\prime}) $.

If we use $\pi_{_{X}}$ to denote the marginal distribution of $X$ under $\pi$
and $\pi_{_{X^{\prime}}}$ to denote the marginal distribution of $X^{\prime}$
under $\pi$, then the optimal transport cost between $P$ and $P_{0}$ can be
written as,
\begin{align}
D_{c}( P_{0},P) =\inf\big\{E_{\pi}\left[ c\left(  X,X^{\prime}\right)
\right]  :\ \pi\in\mathcal{P}( \mathbb{R}^{d}\times\mathbb{R}^{d})
,\ \pi_{_{X}}=P_{0},\,\pi_{_{X^{\prime}}}=P\big\}.\label{OT-Defn}%
\end{align}
The Wasserstein distance is recovered if $c\left(  x,x^{\prime}\right)
=\left\Vert x-x^{\prime}\right\Vert $ under any given norm. If $c\left(
x,x^{\prime}\right)  $ is not a distance, then $D_{c}\left(  P_{0},P\right)  $
is not necessarily a distance.

Ultimately, we are interested in the computational tractability of the DRO
problem (\ref{DRO_B0}) assuming
\begin{equation}
\mathcal{U}_{\delta}(P_{0})=\left\{  P:D_{c}(P_{0},P)\leq\delta\right\}
,\label{Choice_U}%
\end{equation}
for a flexible class of functions $c$. We concentrate on what we call local
Mahalanobis or state-dependent Mahalanobis cost functions of the form,
\begin{equation}
c(x,x^{\prime})=(x-x^{\prime})^{T}A(x)(x-x^{\prime}),\label{Choice_C}%
\end{equation}
where $A(x)$ is a positive definite matrix for each $x$. For this choice of $c(\cdot),$ the distributions in $\mathcal{U}_\delta(P_0)$ are unrestricted in support.
We explain in the
Conclusions section how our results can be applied to other cost functions.

The family of cost functions that we consider is motivated by the perspective
that the adversary introduced in the DRO formulation (\ref{DRO_B0})
(represented by the inner sup) is a modeling tool which explores the impact of
potential decisions. 

Let us consider, for example a situation in which we are interested in choosing an optimal portfolio strategy. In this setting, historical returns can naturally be used to fit a statistical
model. However, there is also current market information which is not of
statistical nature but of economic nature in the form of, for instance,
implied volatilities (i.e., the volatility that is implied by the current
supply and demand reflected by the prices of derivative securities). The
implied volatility differs from the historical volatility and it is more
sensible to capturing current market perceptions. An enhanced DRO\ formulation
which uses a cost function such as (\ref{Choice_C}) could incorporate market
information as follows. Returns with higher implied volatility, maybe even
depending on the current stock values, could be assigned lower cost of
transportation; while returns with lower implied volatility may be given
higher transportation costs. The intuition is that high implied volatilities correspond to potentially higher future fluctuations (as perceived by the
market), so the adversary should be given higher ability relative (and thus
lower costs) to explore the potential implications of such future
out-of-sample fluctuations on portfolio choices.

In general, just as we discussed in the previous paragraph, it is not
difficult to imagine more situations in which the optimizer may be more
concerned about the impact of distributional uncertainty on certain regions of
the outcome space relative to other regions. Such situations may arise as a
consequence of different amounts of information available in different regions
of the outcome space, or perhaps due to data contamination or measurement
errors, which may be more prone to occur for certain values of $x$.

In this paper we do not focus on the problem of fitting the cost function, but
we do consider the portfolio optimization discussed earlier and the use of
implied volatilities in an empirical study in Section \ref{Sec-Eg}. We point
out, however, that related questions have been explored, at least empirically,
in classification settings, using manifold learning procedures
(\cite{blanchet2017data,noh2010generative,wang2012parametric}). Our motivation
is that flexible formulations based on cost functions such as (\ref{Choice_C})
are useful if one wishes to fully exploit the role of the artificial adversary
in (\ref{DRO_B0}) as a modeling tool.

Now, leaving aside the modeling advantages of choosing a cost function such as
(\ref{Choice_C}) and coming back to the computational challenges, even if one
selects $A(x)$ to be the identity (thus recovering a more traditional
Wasserstein DRO formulation) solving (\ref{DRO_B0}) is not entirely easy
because the inner optimization problem in \eqref{DRO_B0} is non-trivial to study. An exact convex optimization reformulation has been demonstrated only for losses taking a specific form.
For example,  \cite{hanasusanto2018conic} provides a conic reformulation for the data-driven DRO problem with piecewise linear convex losses. With the number of conic constraints being proportional to the data set size, it is however computationally less suited for handling large datasets.  

To exploit these DRO formulations one must develop scalable algorithms with guaranteed good performance for solving
(\ref{DRO_B0}). By good performance, we mean that we can easily develop
algorithms for solving (\ref{DRO_B0}) with complexity which is comparable to
that of natural benchmark algorithms for solving (\ref{Stoch_Opt_Model}).
Enabling these good-performing algorithms is precisely one of the goals of
this paper. To this end, several properties such as duality
representations, convexity and the structure of worst-case adversaries are
studied. These results have far ranging implications, as we discuss next.

\textbf{A more in-depth discussion of our technical contributions.}

First, using a standard duality result, we write the inner maximization in
(\ref{DRO_B0}) as,
\begin{equation}
\sup_{P:D_{c}(P_{0},P)\leq\delta}E_{P}\left[  \ell(\beta^{T}X)\right]
=\inf_{\lambda\geq0}E_{P_{0}}[\ell_{rob}\xspace(\beta,\lambda
;X)],\label{Rob-Obj-Fn}%
\end{equation}
for a dual objective function $\ell_{rob}\xspace(\cdot)$ and a dual variable
$\lambda\geq0$. 

Then, we show that after a rescaling in $\lambda$ that the objective function,
$E_{P_{0}}[\ell_{rob}\xspace(\beta,\lambda;X)]$, is locally strongly convex in
$\left(  \beta,\lambda\right)  $ uniformly over a compact set containing the
optimizer; the strong convexity parameter of at least $\kappa_{1}\delta^{1/2}$
(for some $\kappa_{1}>0$ which we identify), under suitable convexity and
growth assumptions on $\ell\left(  \cdot\right)  $; see Theorems
\ref{Thm-strong-duality} - \ref{Thm-Str-Convexity}.

It turns out that the function $\ell_{rob}\xspace(\cdot)$ can be computed by
solving a one dimensional search problem on a compact interval. This can be
solved quite efficiently under (exponentially fast rate of convergence) under
the setting of Theorems \ref{Thm-strong-duality} - \ref{Thm-Str-Convexity}.

We then study a natural stochastic gradient descent algorithm for solving
(\ref{DRO_B0}) which, due to the strong convexity properties derived for
$\ell_{rob}\left(  \cdot\right)  $ achieves an iteration complexity of order
$O_{p}(\varepsilon^{-1}L)$ to reach $O\left(  \varepsilon\right)  $ error,
where $L$ is the cost of solving the one dimensional search problem. We also
discuss in the Appendix how to execute this line search procedure efficiently,
provided that suitable smoothness assumptions are imposed on $\ell\left(
\cdot\right)  $ (leading to an extra factor of order $L=O\left(  \log\left(
1/\varepsilon\right)  \right)  $ in total cost. In this sense, we obtain a
provably efficient iterative procedure to solve (\ref{DRO_B0}).

It is important to note that the non-DRO version of the problem, namely
(\ref{Stoch_Opt_Model}), corresponding to the case $\delta=0$ may not be
strongly convex even if $\ell\left(  \cdot\right)  $ is strongly convex, see Remark \ref{remark-strongly-convex} following Theorem \ref{Thm-Str-Convexity}. So,
in principle, (\ref{Stoch_Opt_Model}) may require $O\left(  1/\varepsilon
^{2}\right)  $ stochastic gradient descent iterations to reach $O\left(
\varepsilon\right)  $ error of the optimal value. Indeed, if $\ell\left(
\cdot\right)  $ is convex, the problem is always convex in $\beta$ (for
$\delta\geq0$), because the supremum of convex functions is convex. 

Of course, $\delta>0$ may be seen as a form of \textquotedblleft
regularization\textquotedblright\ in some cases, as discussed earlier, and
this is a feature that could explain, at least intuitively, the convexity
properties of the objective function. But the goal of formulation
(\ref{DRO_B0}) is not to regularize for the sake of making the problem better
possed from an optimization standpoint. Rather, the point of formulation
(\ref{DRO_B0}) is enabling the flexibility in choosing effective
DRO\ formulations (via (\ref{Choice_C})) in order to improve out-of-sample
properties. This flexibility could come at a price in terms of computational
tractability. The point of keeping the case $\delta=0$ in mind as a benchmark
is that such a price is not incurred and, therefore, our results enable
modelers to use formulations such as (\ref{DRO_B0}) to improve out-of-sample
performance based on side information, as in the portfolio optimization
example mentioned earlier.

Another useful consequence of our results involves the application of standard
Sample Average Approximation statistical analysis results to Optimal Transport
based DRO. This enables the direct application of results in conjunction with,
for example, \cite{shapiro2009lectures}, to produce confidence regions for the
solution of the DRO formulation.

Another interesting contribution of our analysis consists in studying the local
structure of the worst-case optimal transport plan, including uniqueness and
comparative statics results, see Proposition \ref{Prop-Comp-Stat} and Theorem
\ref{Thm-WC-Dist}. 

The structure of the optimal transport plan, we believe, could prove helpful
in the development of statistical results to certify robustness and in
providing insights for robustification in non-convex objective functions. Some
of the statistical implications are studied in (see \cite{blanchet2019confidence}).

\textbf{Organization of the paper.} We now describe how to navigate the
results in the paper. Throughout the rest of the paper we introduce
assumptions as we need them. Often these assumptions and the corresponding
results that are obtained involved constants, which are surveyed in a table
presented in Appendix \ref{Sec-table}.

Section \ref{Sec-Reform-Convexity}, sets the stage for our analysis by first
obtaining the duality result (\ref{Rob-Obj-Fn}). The duality result in
(\ref{Rob-Obj-Fn}) is given only under the assumption that $\ell\left(
\cdot\right)  $ is upper semi-continuous and $c\left(  \cdot\right)  $ is as
in (\ref{Choice_C}), assuming $A\left(  x\right)  $ is uniformly well
conditioned in $x$.

In Section \ref{subsec_basic_convexity}, under the assumption that
$\ell\left(  \cdot\right)  $ is convex, with at most quadratic growth and
fourth order moments of $P_{0}$, we establish convexity and finiteness in the
right-hand of (\ref{Rob-Obj-Fn}). 

In Section \ref{subsec:str-convexity-compactB}, we add the assumptions that
$\ell\left(  \cdot\right)  $ is twice differentiable, with a natural
non-degeneracy condition on $P_{0}$, and that the feasible set, $B$, is convex
and compact. We characterize a useful region (compact and with convenient
analytical properties), called $\mathbb{V}$, which contains the dual optimizer
$\lambda_{\ast}\left(  \beta\right)  $, parametrically as a function of each
decision $\beta$. Then, we show smoothness and strong convexity in $\beta$ of
the right-hand side of (\ref{Rob-Obj-Fn}) on $\mathbb{V}$.

Also in Section \ref{subsec:str-convexity-compactB}, now under a local
strong convexity condition on $\ell\left(  \cdot\right)  $, and a
strengthening of the non-degeneracy condition on $P_{0}$ mentioned earlier, we
extend the smoothness and strong convexity of the right hand side of
(\ref{Rob-Obj-Fn}) both in $\beta$ and the dual variable $\lambda>0$, provided
that $\delta$ is chosen suitably small, throughout $\mathbb{V}$.

The assumption that $B$ is compact is imposed to simplify the strong convexity
analysis and comparative statics (i.e. the structure of the worst case
distribution and comparative statics). We show in Section
\ref{subsec:noncompact-B} that the compactness of $B$ can be relaxed at the
expense of additional technical burden.

The structure of the worst case is studied in Section \ref{Sec-WorstCase-Dist}%
, in Theorem \ref{Thm-WC-Dist}. The result includes the amount of displacement
(parametrically in $\delta$) of the optimal transport plan and the existence
of a Monge map (i.e. a direct `matching' between outcomes of $P_{0}$ and those
of the worst case distribution). We also discuss situations in which the
optimal transport plan may not exist (even if an optimal solution to
(\ref{Rob-Obj-Fn}) exists), among other results.

Comparative statics results, including the uniqueness of the worst case distribution as a Monge map, as well as monotonicity in the amount of the
displacement as a function of $\delta$ for every single outcome of $P_{0}$ are
also discussed in Section \ref{Sec-WorstCase-Dist}. Also, the geometry of the
worst case transportation parametrically in $\delta$ is shown to follow
straight lines.

In Section \ref{Sec-SA-Algo},  we examine the wide range of algorithmic implications
which follow from the results in earlier sections. Section
\ref{Sec-Oracle-Info-Extr} studies how to evaluate subgradients of the
function $\ell_{rob}\xspace(\cdot)$ inside the expectation in
(\ref{Rob-Obj-Fn}). This is discussed under mild assumptions which do not require the loss $\ell(\cdot)$ to be differentiable. So, the result
can be applied to developing stochastic subgradient descent algorithms for
non-differentiable losses if derivatives and expectations can be swapped.

This swapping is explored in Section \ref{Sec-First-SGD}. We evaluate gradients
for the expectation in the right hand side of (\ref{Rob-Obj-Fn}) under the
assumptions imposed in Section \ref{subsec_basic_convexity}, and a formal
stochastic gradient descent scheme is given in Section
\ref{sec-iterative-scheme}, together with corresponding iteration complexity
analysis discussed in Section \ref{subsect_rates_of_convergence}.

In Section \ref{Sec-Mod-SGD} we discussed potential enhancements of the basic
stochastic gradient descent strategy introduced in Section
\ref{sec-iterative-scheme}. These include a two-scale stochastic approximation
scheme for dealing with the evaluation of the gradients of $\ell
_{rob}\xspace(\cdot)$ and the case in which $\delta$ may not be small enough
to apply the smoothness results from Section
\ref{subsec:str-convexity-compactB} and we need to deal with
non-differentiable losses as well.

We provide several specific examples in Section \ref{Sec-Eg}. These are
designed to derive the expressions of the structural results that we present,
explore the structure of the worst case probability model and its behavior
parametrically in $\delta$. The various constants summarized required in the
assumptions for application of our structural results are summarized in
Appendix \ref{Sec-table}. With the complexity of the SGD approach not scaling with the data size, the numerical study in Section \ref{sec-sgd-socp} demonstrates the distinct computational advantage enjoyed by the proposed SGD scheme over second order cone formulations derived from piecewise linear approximation to the loss $\ell(\cdot)$.

In Section \ref{subsec_portfolio}, we provide a discussion related to the
portfolio optimization discussed earlier in the Introduction. The set of
matrices, $A\left(  x\right)  $, is calibrated based on an implied volatility
index and $P_{0}$ is constructed based on several years of historical data for
the S\&P500 index.

The proofs of our main structural results are given in Section
\ref{Sec-Proofs}. Additional discussion involving technical lemmas and
propositions, which are auxiliary to our main structural results are given in
the appendix, in Section \ref{Sec-App-Proofs}. The discussion on the
complexity of the line search, which underlies the gradient evaluation of
$\ell_{rob}\xspace(\cdot)$ and is given in Section \ref{Sec-Line-Search}.

\textbf{Notations.} In the sequel, the symbol $\mathcal{P}(S)$ is used to
denote the set of all probability measures defined on a complete separable
metric space $S.$ A collection of random variables $\{X_{n}:n\geq1\}$ is said
to satisfy the relationship $X_{n}=O_{p}(1)$ if it is tight; in other words,
for any $\varepsilon>0,$ there exists a constant $C_{\varepsilon}$ such that
$\sup_{n}P(|X_{n}|>C_{\varepsilon})<\varepsilon.$ Following this notation, we
write $X_{n}=O_{p}(g(n))$ to denote that the family $\{X_{n}/g(n):n\geq1\}$ is
tight. The notation $X\sim P$ is to write that the law of $X$ is $P$. For any measurable function $f:S \rightarrow \mathbb{R}$, we denote the essential supremum of $f$ under measure $P\in\mathcal{P}(S)$ as $P-\textrm{ess-sup}_{x} f(x) 
:= \inf\{a\in\mathbb{R}: P(f^{-1}(a,\infty))
= 0\}$.
For any real-symmetric matrix $A$, we write $A\succeq0$ to denote that $A$ is a
positive semidefinite matrix. The set of $d$-dimensional positive definite
matrices with real entries is denoted by $\mathbb{S}_{d}^{++}$. The
$d$-dimensional identity matrix is denoted by $\mathbb{I}_{d}.$ The norm
$\Vert\cdot\Vert$ is written to denote the $\ell_{2}-$euclidean norm unless
specified otherwise. For any real vector $x$ and $r>0,$ $\mathcal{N}_{r}(x)$
denotes the neighborhood $\mathcal{N}_{r}(x):=\{y:\Vert y-x\Vert<r\}.$ We say
that a collection of random variables $\{X_{c}:c\in\mathcal{C}\}$ is $L_{2}%
-$bounded (or bounded in $L_{2}-$norm) if $\sup_{c\in\mathcal{C}}E\Vert
X_{c}\Vert^{2}<\infty$. For any function $f:\mathbb{R}^{d}\rightarrow
\mathbb{R}$, the notation $\nabla f$ and $\nabla^{2}f$ are written to donate,
respectively, the gradient and Hessian of $f.$ In instances where it is
helpful to clarify the variable with which partial derivatives are taken, we
resort to writing, for example, $\nabla_{x}f(x,y)$, $\nabla_{x}^{2}f(x,y)$, or
equivalently, $\partial f/\partial x$, $\partial^{2}f/\partial x^{2}$ to
denote that the partial derivative is taken with respect to the variable $x$.
We write $\partial_{+}f$, $\partial_{-}f$ to denote the right and left derivatives.

\section{Dual reformulation and convexity properties.}

\label{Sec-Reform-Convexity} In this section we first re-express the robust
(worst-case) objective as in (\ref{Rob-Obj-Fn}). Such reformulation, entirely
in terms of the baseline probability distribution $P_{0}$, is useful in
deriving the convexity and other structural properties to be examined in
Sections \ref{Sec-Convexity} - \ref{Sec-Comp-Statics}. In turn, the
reformulation (\ref{Rob-Obj-Fn}) is helpful in developing stochastic gradient
based iterative descent schemes described in Section \ref{Sec-SA-Algo}.

\subsection{Dual reformulation.}

\label{Sec-Dual-Reform} It follows from the definition of the optimal
transport costs $D_{c}(P_{0},P)$ (see \eqref{OT-Defn}) that the worst-case
objective in \eqref{Rob-Obj-Fn} equals
\[
\sup\left\{  \int\ell(\beta^{T}x^{\prime})d\pi(x,x^{\prime}):\ \pi
\in\mathcal{P}(\mathbb{R}^{d}\times\mathbb{R}^{d}),\ \pi(\ \cdot
\times\mathbb{R}^{d})=P_{0}(\cdot),\ \int c(x,x^{\prime})d\pi(x,x^{\prime
})\leq\delta\right\}  ,
\]
which is an infinite-dimensional linear program that maximizes $E_{\pi}%
[\ell(\beta^{T}X^{\prime})]$ over all joint distributions $\pi$ of pair
$(X,X^{\prime})\in\mathbb{R}^{d}\times\mathbb{R}^{d}$ satisfying the linear
marginal constraints that the law of $X$ is $P_{0}$
and the cost constraint that $E_{\pi}[c(X,X^{\prime})]\leq\delta$ (see
\cite[Section 2.2]{blanchet_quantifying_2016} for details). A precise description of the
state-dependent Mahalanobis transport costs $c(\cdot,\cdot)$ we consider in this paper
is given in Assumption \ref{Assump-c} below.

\begin{assumption}
\textnormal{\ \sloppy The transport cost function
$c: \mathbb{R}^d \times \mathbb{R}^d \rightarrow \mathbb{R}_+$ is
of the form \[c(x,x^\prime) = (x-x^\prime)^TA(x)(x-x^\prime),\]
where $A:\mathbb{R}^d \rightarrow \mathbb{S}%
_d^{++}$ is such that a) $c(\cdot)$ is lower-semicontinuous, and
b) there exist positive constants
$\rho_{\min}, \,\rho_{\max}$ satisfying
$\sup_{\Vert v \Vert = 1} v^TA(x)v \, \leq \, \rho_{\max}$ and
$\inf_{\Vert v \Vert = 1} v^TA(x)v \, \geq \, \rho_{\min}\,,$
	for $P_0-$almost every $x \in \mathbb{R}^d.$
\label{Assump-c}}
\end{assumption}

As mentioned in the Introduction, a transport cost function satisfying
Assumption \ref{Assump-c} is not necessarily symmetric (hence need not be a
metric). The special case of $A(x)$ being the identity matrix (for all $x$)
corresponds to the $D_{c}^{1/2}(\cdot)$ being the well-known Wasserstein distance (in this case, the constants $\rho_{\max} = \rho_{\min} = 1$).
Theorem \ref{Thm-strong-duality} below builds on a general strong duality
result applicable for this linear program when the chosen transport cost
function $c(x,x^{\prime})$ is not necessarily a metric.

\begin{theorem}
\label{Thm-strong-duality} Suppose that $\ell:\mathbb{R} \rightarrow \mathbb{%
R}$ is upper semicontinuous. Then, under Assumption \ref{Assump-c}, the
worst-case objective,
\begin{align*}
\sup_{P:D_c(P_0,P) \leq \delta} E_P \left[\ell(\beta^TX) \right] =
\inf_{\lambda\geq 0} f_\delta(\beta,\lambda),
\end{align*}
where
$f_\delta(\beta,\lambda) := E_{P_0}[\ell_{rob}
\xspace(\beta,\lambda;X)],$ $%
\ell_{rob} \xspace(\beta, \lambda;x) := \sup_{\gamma\in \mathbb{R}}
F(\gamma,\beta,\lambda;x),$ and 
\begin{align}
F(\gamma,\beta,\lambda;x):= \ell \left(\beta^Tx + \gamma\sqrt{\delta}
\beta^T A(x)^{-1}\beta \right) - \lambda \sqrt{\delta} \big(\gamma^2 \beta^T
A(x)^{-1}\beta - 1\big).  \label{Univar-Opt}
\end{align}
For any $\beta\in B$, there exists a dual optimizer
$\lambda_{\ast}(\beta)%
\geq 0$ such that
$f_\delta(\beta,\lambda_\ast(\beta)) = \inf_{\lambda\geq
0}f_\delta(\beta,\lambda)$. 
\end{theorem}The proof of Theorem \ref{Thm-strong-duality} is provided in
Section \ref{sec:proofs-dualreform-convex}.

\vspace{-0.1in}

\subsection{Convexity and smoothness properties of the dual DRO objective.}
\label{Sec-Convexity} Here we study the convexity and smoothness properties of the dual objective function $f_{\delta}(\beta,\lambda).$ 

\vspace{-0.1in}

\subsubsection{Convexity. \label{subsec_basic_convexity}}

We first identify conditions under which the function $f_{\delta}(\cdot)$ is
proper and convex. \begin{assumption}
\textnormal{\ The loss function $\ell: \mathbb{R} \rightarrow \mathbb{R}$ is
convex and it satisfies the growth condition that $\kappa := \inf\{ s \geq
0: \sup_{u \in \mathbb{R}}(\ell(u) - s u^2) < \infty\}$ is finite. In
addition, the baseline distribution $P_0$ is such that $E_{P_0}\Vert X
\Vert^4 < \infty.$ } \label{Assump-Obj-Conv}
\end{assumption}

\begin{theorem}
The function
$f_\delta:B \times \mathbb{R}_+ \rightarrow \mathbb{R} \cup \{\infty\} $ is
proper and convex when Assumptions \ref{Assump-c} and 2 hold.
\label{thm:convex}
\end{theorem}The proof of Theorem \ref{thm:convex} can be found in Section
\ref{sec:proofs-dualreform-convex}.

\vspace{-0.1in}

\subsubsection{Smoothness and strong convexity.}

\label{subsec:str-convexity-compactB} Next, we establish smoothness, strong convexity of $f_{\delta}(\cdot, \lambda)$ for fixed
$\lambda,$ and joint strong convexity of $f_{\delta}(\cdot),$ when restricted to the domain $\mathbb{V},$ under
increasingly stronger sets of assumptions. While these assumptions are helpful in  understanding smoothness and strong convexity properties, the development of iterative schemes in Appendix \ref{sec-algorithm} does not require these stronger assumptions.

\begin{assumption}
\textnormal{The loss function
$\ell: \mathbb{R} \rightarrow \mathbb{R}$ is twice differentiable
with bounded second derivatives. Specifically, we have a positive
constant $M$ such that $\ell^{\prime\prime}(\cdot) \leq M.$
Moreover, the baseline distribution $P_0$ is such that
$\ell^\prime(\beta^TX)$ is not identically 0, for any
$\beta \in B.$ }
\label{Assump-Obj-2diff}
\end{assumption}

\begin{assumption}
\textnormal{The set $B \subseteq \mathbb{R}^d$ is convex and compact.
Specifically, $\sup_{\beta \in B} \Vert \beta \Vert =: R_\beta < \infty.$} %
\label{Assump-Compact}
\end{assumption}

Recall from Theorem \ref{Thm-strong-duality} that $\text{arg} \min
_{\lambda\geq0} f_{\delta}(\beta,\lambda)$ is not empty for every $\beta\in
B.$

\begin{proposition}
Suppose that Assumptions \ref{Assump-c} - \ref{Assump-Compact}
hold. Then for any $\beta \in B$ and dual optimizer
$\lambda_\ast(\beta) \in \arg\min_{\lambda \geq 0}f_\delta(\beta,\lambda),$
we have $(\beta,\lambda_\ast(\beta)) \in \mathbb{V},$ where
\begin{align}
\mathbb{V} := \{(\beta,\lambda) \in B \times \mathbb{R}_+: \ K_1\Vert \beta \Vert
\leq \lambda \leq K_2\Vert \beta \Vert\},
\label{Defn-V}
\end{align}
for some positive constants $K_1,K_2$ which can be explicitly
determined in terms of parameters
$\delta,M,R_\beta,\rho_{\max},\rho_{\min}.$
\label{Prop-V-cont-OPT}
\end{proposition}
To avoid clutter, we provide explicit characterizations for the constants $K_{1},K_{2}$ in the proof of Proposition \ref{Prop-V-cont-OPT} (see Section \ref{Sec-Prep-results}) and as well in Table \ref{tab:constants-derived} (see Appendix \ref{Sec-table}).

\begin{theorem}
Suppose that Assumptions \ref{Assump-c} - \ref{Assump-Compact} are
satisfied. Then there exist positive constants
$\delta_0,\kappa_{0}$ such that the following hold: Whenever
$\delta < \delta_0,$ the function
$f_\delta:B \times \mathbb{R}_+ \rightarrow \mathbb{R} \cup
\{\infty\}$ satisfies the following properties:
\begin{itemize}
\item[a)] $f_\delta(\cdot)$ is twice differentiable throughout the domain
$\mathbb{V}$ with a uniformly bounded Hessian;
\item[b)] the second derivative of $f_\delta(\cdot)$ satisfies
\[\frac{\partial^2 f_\delta}{\partial \beta^2}(\beta,\lambda) \succeq \sqrt{\delta} \kappa_{0}
\lambda^{-1} \mathbb{I}_d, \qquad \text{ for } (\beta,\lambda) \in
\mathbb{V}.\]
\end{itemize}
\label{result-smoothness}
\end{theorem}

Theorem \ref{result-smoothness} identifies conditions under which the dual DRO objective $f_{\delta}(\cdot)$ has Lipschitz continuous gradients (smoothness) and also points towards strong
convexity in terms of the parameter $\beta$ (for any fixed $\lambda$), Similar
to Proposition \ref{Prop-V-cont-OPT}, we provide explicit characterizations
for the constants $\delta_{0},\kappa_{0}$ in the proof of Theorem
\ref{result-smoothness} in Section \ref{subsec:proofs-thms-smooth-str-convex} (see also Appendix \ref{Sec-table} for tables summarizing useful constants). 
We next focus on characterizing strong convexity jointly in the parameters
$(\beta,\lambda).$

\begin{assumption}
The loss function $\ell: \mathbb{R} \rightarrow \mathbb{R}$ is locally strongly convex.  In addition, for every
$\beta \in B,$ the baseline distribution $P_0$ is such that there
exist $c_1,c_2 \in (0,\infty),$ $p \in (0,1)$ satisfying
$ P_0 \left( \vert \ell^\prime(\beta^TX) \vert > c_1, \ \vert \beta^T
X \vert > c_2\Vert \beta \Vert \right) \geq p.$
\label{Assump-nondegeneracy}
\end{assumption}

\begin{theorem}
Suppose that Assumptions \ref{Assump-c} - \ref{Assump-nondegeneracy}
hold.  Then there exist constants $\delta_1 \in (0,\delta_0)$ and
$\kappa_{1} \in (0,\infty)$ such that whenever $\delta < \delta_1,$
the Hessian of the function
$f_\delta:B \times \mathbb{R}_+ \rightarrow \mathbb{R} \cup
\{\infty\}$ satisfies,
\[\nabla^2 f_\delta(\theta) \succeq \sqrt{\delta}\kappa_{1} \mathbb{I}_{d+1},\]
for $\theta \in \mathbb{V}.$
\label{Thm-Str-Convexity}
\end{theorem}

The proof of Theorem \ref{Thm-Str-Convexity}, along with an explicit
characterization of the constant $\delta_1$, is presented in Section
\ref{subsec:proofs-thms-smooth-str-convex}. Theorem \ref{Thm-Str-Convexity}
above identifies conditions under which $f_{\delta}(\cdot)$ is strongly convex
(jointly over $(\beta,\lambda)$) when restricted to the set $\mathbb{V}$.
Indeed, because of Proposition \ref{Prop-V-cont-OPT}, it is sufficient to
restrict attention to $\mathbb{V}$ to arrive at local strong convexity around
$\arg\min_{\beta,\lambda}f_{\delta}(\beta,\lambda).$ To the best of our
knowledge, Theorem \ref{Thm-Str-Convexity} is the first result that
presents strong convexity of the objective in Wasserstein distance based
DRO in a suitable sense. As is well-known, strong convexity is a property that
determines the iteration complexity of gradient based descent methods. We
utilize this in Section \ref{Sec-SA-Algo} to derive convergence properties of
the proposed iterative schemes.

\begin{remark}
\label{remark-strongly-convex}
It is instructive to recall that $\ell(\cdot)$ being strongly convex does not
mean $E_{P_{0}}[\ell(\beta^{T}X)]$ is necessarily strongly convex. For
example, consider the underdetermined case of least-squares linear regression
where $\ell(u) = (y-u)^{2}$ and the number of samples $n < d.$ If we take
$P_{0}$ to be the empirical distribution corresponding to the $n$ data samples
$(X_{i},Y_{i}),$ the stochastic optimization objective to be minimized,
$E_{P_{0}}[(Y - \beta^{T}X)^{2}] = n^{-1}\sum_{i=1}^{n}(Y_{i} - \beta^{T}%
X_{i})^{2}$ is not strongly convex. Theorem \ref{Thm-Str-Convexity} asserts
that the respective dual DRO objective $f_{\delta}(\beta,\lambda)$ is,
nevertheless, strongly convex in a region containing the minimizer (refer an
example in Section \ref{sec-linear-regression} for a discussion on how a DRO
formulation of the least squares linear regression problem results in the dual
objective of the form $f_{\delta}(\beta,\lambda)$). Thus, due to Theorem
\ref{Thm-Str-Convexity}, for a considerable class of useful loss functions
$\ell(\cdot),$ the DRO dual objective to be minimized, $f_{\delta}%
(\beta,\lambda),$ is strongly convex in a suitable sense, even if the
non-robust counterpart $E_{P_{0}}[\ell(\beta^{T}X)]$ is not.
\end{remark}


\vspace{-0.1in}

\subsubsection*{Comments on Assumptions \ref{Assump-c}
-~\ref{Assump-nondegeneracy}.}

Assumptions \ref{Assump-c} - \ref{Assump-Obj-Conv} above ensure that the DRO
objective \eqref{Rob-Obj-Fn} is convex, proper and that the strong duality
utilized in Theorem \ref{Thm-strong-duality} is indeed applicable. These
non-restrictive assumptions serve the purpose of clearly stating the framework
considered. Indeed, Assumptions \ref{Assump-c} - \ref{Assump-Obj-Conv} are
satisfied by a wide variety of loss functions $\ell(\cdot)$ and a flexible
class of state-dependent Mahalanobis cost functions $c(\cdot)$ which include commonly
used Euclidean metric, Mahalanobis distances as special cases.
As we shall see in the proof of Theorem \ref{Thm-Str-Convexity}, the twice
differentiability imposed in Assumption \ref{Assump-Obj-2diff} is necessary to
characterize the local strong convexity of $f_{\delta}$ by means of the
positive definiteness of Hessian of $f_{\delta}$. The assumption of
boundedness of the set $B,$ though not necessary for strong convexity (see
following Section \ref{subsec:noncompact-B}), is essential for guaranteeing
differentiability of $f_{\delta}(\cdot).$ Moving to Assumption
\ref{Assump-nondegeneracy},
the positive probability requirement in Assumption \ref{Assump-nondegeneracy}
rules out the degeneracy that $P_{0}$ is not concentrated entirely in the
regions where either $\vert\ell^{\prime}(\beta^{T}x) \vert$ or $\vert\beta
^{T}x \vert$ is small.
See Remark \ref{Rem-nonzero-bX-need} (following the proof of Theorem
\ref{Thm-Str-Convexity} in Section \ref{subsec:proofs-thms-smooth-str-convex})
for an explanation of why the positivity of $c_{1}, c_{2}$ is necessary to
identify the coefficient $\kappa_{1}$ which is independent of the ambiguity
radius $\delta.$ We would like to reiterate that the development of iterative schemes in Section \ref{sec-algorithm} does not require Assumptions \ref{Assump-Obj-2diff} - \ref{Assump-nondegeneracy}.


\vspace{-0.1in}

\subsubsection{Strong convexity property for non-compact $B.$}

\label{subsec:noncompact-B}
As we shall see in Theorem \ref{result-noncompact} below, compactness of the
set $B$ (as in Assumption \ref{Assump-Compact}) is not crucial for strong
convexity of the DRO objective around the minimizer. Assumption
\ref{Assump-Compact} is merely a simplifying assumption which allows to
study additional structural properties such as differentiability, smoothness
(see Theorem \ref{result-smoothness}) and comparative statics (see Section
\ref{Sec-Comp-Statics}). A proof of Theorem \ref{result-noncompact} is
presented in Appendix \ref{Sec-App-Proofs}.


\begin{theorem}
Suppose that Assumptions \ref{Assump-c} - \ref{Assump-Obj-2diff} and
Assumption \ref{Assump-nondegeneracy} are satisfied.  In addition,
suppose we have positive constants $k_1,k_2$ such that
$\vert u \vert\ell^{\prime\prime}(u) \leq k_1 + k_2 \vert
\ell^\prime(u) \vert,$ for $u \in \mathbb{R}.$ Then there exists
$\delta_2 > 0$ such that for every $\delta < \delta_2,$ the
following property holds: For any $\beta \in B,$ we have
positive constants $\kappa, r$ such that
\begin{align*}
f_\delta(\alpha  \theta_1 + (1-\alpha)\theta_2)  \leq
\alpha f_\delta(\theta_1) + (1-\alpha) f_\delta(\theta_2) - \frac{1}{2}
\kappa \alpha(1-\alpha) \Vert\theta_1  - \theta_2 \Vert^2,
\end{align*}
for every
$\theta_1,\theta_2 \in  \mathcal{N}_{r}((\beta,\lambda_{\ast}(\beta))).$
\label{result-noncompact}
\end{theorem}

\vspace{-0.1in}

\subsection{Structure of the worst-case distribution.}

\label{Sec-WorstCase-Dist} Fixing $\beta\in B, $ we explain the structure of
worst case distribution(s) that attains the supremum in \eqref{Rob-Obj-Fn} by
utilizing the solution of the respective dual problem $\inf_{\lambda\geq0}
f_{\delta}(\beta,\lambda)$ (see Theorem \ref{Thm-strong-duality}). Recall the
notation that $\lambda_{\ast}(\beta)$ attains the infimum in $\inf
_{\lambda\geq0} f_{\delta}(\beta,\lambda)$ for fixed $\beta\in B.$ For each
$\beta\in B, \lambda\geq0$ and $x\in\mathbb{R}^{d}$, define the set of optimal
solutions to \eqref{Univar-Opt} as
\begin{align}
\label{Eqn-Optimal-Set}\Gamma^{\ast}(\beta,\lambda;x) = \left\{ \gamma
:\ F(\gamma,\beta,\lambda;x) = \sup_{c \in\mathbb{R}}F(c,\beta,\lambda
;x)\right\} .
\end{align}
Finally, for a fixed $\beta\in B,$ define 
$$\lambda_{thr} \xspace(\beta) =
\kappa \sqrt{\delta}  (P_{0}-\textrm{ess-sup}_{x} \beta^{T}
A(x)^{-1}\beta).$$
Similarly, when Assumption \ref{Assump-Obj-2diff}
holds, define $$\lambda_{thr}^{\prime} \xspace(\beta) =
\frac{1}{2}M\sqrt{\delta} (P_{0}-\textrm{ess-sup}_{x} \beta^{T}
A(x)^{-1}\beta).$$ Since
$\kappa\leq M/2,$ we have $\lambda_{thr}^{\prime}\xspace(\beta) \geq
\lambda_{thr} \xspace(\beta)$ for every $\beta\in B.$

\begin{theorem}
Suppose that Assumptions \ref{Assump-c},\ref{Assump-Obj-Conv} hold and $%
\beta \neq \mathbf{0}.$ Take any dual optimizer $\lambda_\ast(\beta) \in
\arg\min_{\lambda \geq 0} f_{\delta}(\beta,\lambda).$ Then
\begin{itemize}
\item[a)] the dual optimizer $\lambda_\ast(\beta)$ is strictly positive unless $%
\ell(\cdot)$ is a constant function. If $\ell(\cdot)$ is indeed a constant
function, then any distribution in $\mathcal{U}_{\delta}(P_0) = \{P:\ D_c(P_0,P) \leq \delta\}$ attains
the supremum in \eqref{Rob-Obj-Fn};
\item[b)] the dual optimizer $\lambda_\ast(\beta) \geq \lambda_{thr} \xspace%
(\beta)$ whenever $\ell(\cdot)$ is not a constant;
\item[c)] if $\lambda_\ast(\beta) > \lambda_{thr} \xspace(\beta),$ the law
of
\begin{align}
X^\ast := X + \sqrt{\delta} G A(X)^{-1}\beta  \label{Transported-Xstar}
\end{align}
attains the supremum in \eqref{Rob-Obj-Fn} and satisfies $E[c(X,X^\ast)] =
\delta;$ here the random variable $G$ can be written as $G := ZG_{-} +
(1-Z)G_{+},$ with $G_{-} = \inf\Gamma(\beta,\lambda_\ast(\beta);X), \, G_+ =
\sup \Gamma(\beta,\lambda_\ast(\beta);X),$ $P_0-$almost surely, and $Z$ is
an independent Bernoulli random variable satisfying $P(Z = 1) = (\overline{c}
- 1)/(\overline{c} -\underline{c}),$ where $\overline{c} :=
E_{P_0}[G_{+}^2\beta^TA(X)^{-1}\beta]$ and $\underline{c} :=
E_{P_0}[G_{-}^2\beta^TA(X)^{-1}\beta];$
\item[d)] if $\lambda_\ast(\beta) = \lambda_{thr} \xspace(\beta),$ then a
worst-case distribution attaining the supremum in \eqref{Rob-Obj-Fn} may not
exist;
\item[e)] under additional Assumption \ref{Assump-Obj-2diff}, if
$\lambda_\ast(\beta) > \lambda_{thr}^\prime \xspace(\beta),$ the
set $%
\Gamma^\ast(\beta,\lambda_\ast(\beta);x)$ is a singleton for every
$x \in \mathbb{R}^d.$ Then for the random variable $G$ being the
unique element in $%
\Gamma^\ast(\beta,\lambda_\ast(\beta);X),$ $P_0-$almost surely, we
have that the law of $X^\ast := X + \sqrt{\delta}GA(X)^{-1}\beta$ is
the only distribution that attains the supremum in
\eqref{Rob-Obj-Fn}. In addition, $%
E[c(X,X^\ast)] = \delta.$
\end{itemize}
\label{Thm-WC-Dist}
\end{theorem}

The proof of Theorem \ref{Thm-WC-Dist} is presented in Section
\ref{Sec-Pf-Thm-WC}.


\begin{remark}
\textnormal{\ Consider the case $\beta = \mathbf{0}.$ Then $\lambda = 0$
attains the minimum in $\min_{\lambda \geq 0}f_{\delta}(\mathbf{0},\lambda),$ $%
\sup_{D_c(P_0,P) \leq \delta} E_{P_0}[\ell(\beta^TX)] = \ell(0),$ and any
distribution in $\{P:\ D_c(P_0,P) \leq \delta\}$ attains the supremum.} %
\label{Rem-WC-Dist-Zero-Beta}
\end{remark}

\vspace{-0.1in}

\subsection{Comparative statics analysis.}

\label{Sec-Comp-Statics} In this section we explain how the worst-case
distribution structure explained in Section \ref{Sec-WorstCase-Dist} changes
for every realization of $X$ when the radius of ambiguity $\delta$ is changed.
Such a sample-wise description is facilitated by examining the derivative of
the random variable $G$ described in Part e) of Theorem \ref{Thm-WC-Dist},
$P_{0}-$almost surely.



\begin{theorem}
Suppose that the assumptions in Theorem \ref{result-smoothness} are
satisfied. For any $\delta \in (0, \delta_1)$ and fixed
$\beta \in B \setminus \{%
\mathbf{0}\},$ there exists a unique worst-case distribution
$P_{\delta}^\ast$ which attains the supremum in
$\sup_{P:D_c(P_0,P) \leq \delta}E_P[\ell(\beta^TX)].$ In particular,
there exist random variables $\{G_\delta: \delta \in (0,\delta_1)\}$
such that
\begin{itemize}
\item[a)] the law of $X_\delta^\ast := X + \sqrt{\delta} G_\delta
A(X)^{-1}\beta$ is $P^\ast_\delta;$
\item[b)] $0 < \sqrt{\delta}G_{\delta} < \sqrt{\delta^\prime}%
G_{\delta^\prime}$ whenever $0 < \delta < \delta^\prime < \delta_1$ and $%
\ell^{\prime}(\beta^TX) > 0;$
\item[c)] $\sqrt{\delta^\prime}G_{\delta^\prime} < \sqrt{\delta}G_{\delta} <
0$ whenever $0 < \delta < \delta^\prime < \delta_1$ and $%
\ell^{\prime}(\beta^TX) < 0;$ and
\item[d)] $G_\delta = 0$ whenever $\delta \in (0,\delta_1)$ and $%
\ell^{\prime}(\beta^TX) = 0.$
\end{itemize}
Therefore, $\Vert X_{\delta}^\ast - X \Vert \leq \Vert
X_{\delta^\prime}^\ast - X \Vert,$ $P_0-$almost surely, whenever $0 < \delta
< \delta^\prime < \delta_1.$
\label{Prop-Comp-Stat}
\end{theorem}

The proof of Theorem \ref{Prop-Comp-Stat} is presented in Section
\ref{Sec-Pf-Thm-WC}. Interestingly, Theorem \ref{Prop-Comp-Stat} asserts that
the trajectory $\{X_{\delta}^{\ast}:\delta\in\lbrack0,\delta_{1})\}$ is a
straight-line, $P_{0}-$almost surely, with probability mass being transported
to farther distances as $\delta$ increases in $[0,\delta_{1}).$ A pictorial
description of this phenomenon can be inferred from Figure
\ref{fig-logistic-fig1} in Section \ref{Sec-Eg} devoted to numerical demonstrations.

\section{Algorithmic implications of the strong convexity properties.}

\label{Sec-SA-Algo}
A key component of this section is a stochastic gradient based iterative
scheme that exhibits the following desirable convergence properties:

\begin{itemize}
\item[a)] The proposed scheme enjoys optimal rates of convergence among the class of iterative algorithms that utilize first-order oracle information and possesses per-iteration effort not dependent on the size of the support of
$P_{0}.$

\item[b)] Compared with the `non-robust' counterpart $\inf_{\beta\in B}
E_{P_{0}}[\ell(\beta^{T}X)],$ the proposed first-order method yields similar
(or) superior rates of convergence for the DRO formulation \eqref{DRO_B0}.
\end{itemize}

In the case of data-driven problems where $P_{0}$ is taken to be the empirical
distribution, the size of the support of $P_{0}$ is simply the size of the
data set. In such cases, Property a) above is a particularly pleasant property
as it allows Wasserstein distance based DRO formulations to be amenable for
big data problems that have become common in machine learning and operations
research. Alternative approaches that directly solve the resulting convex
program reformulations without resorting to stochastic gradients suffer from a
large problem size when employed for large data sets (see, for example,
\cite{shafieezadeh-abadeh_distributionally_2015, MohajerinEsfahani2017}).
Further, the proposed stochastic gradients based approaches are also
immediately applicable to problems where $P_{0}$ has uncountably infinite support.

Property b) above makes sure that computational intractability is not a reason
that should deter the use of DRO approach towards optimization under
uncertainty. In fact Property b) describes that it may be computationally more
advantageous, in addition to the desired robustness, to work with the DRO
formulation \eqref{DRO_B0} compared to its stochastic optimization counterpart
$\inf_{\beta\in B} E_{P_{0}}[\ell(\beta^{T}X)].$ As we shall see in Section
\ref{Sec-First-SGD}, this computational benefit for the proposed stochastic
gradient descent scheme is endowed by the strong convexity properties of the
dual objective $f_{\delta}(\beta,\lambda)$ derived in Theorem
\ref{Thm-Str-Convexity}. Guided by the strong convexity structure of
$f_{\delta}(\beta,\lambda),$ we also discuss enhancements to the vanilla SGD
scheme in Sections \ref{Sec-two-timescale-SGD} and \ref{Sec-GSS-SGD}.

\vspace{-0.1in}

\subsection{Extracting first-order information.}

\label{Sec-Oracle-Info-Extr} Recall the univariate maximization
\eqref{Univar-Opt} that defines $\ell_{rob} \xspace(\beta,\lambda;x)$ for
$\beta\in B, \lambda\geq0, x \in\mathbb{R}^{d}$ and the set of maximizers
$\Gamma^{\ast}(\beta,\lambda;x)$ in \eqref{Eqn-Optimal-Set}. With the DRO
objective \eqref{Rob-Obj-Fn} being related to the dual objective $f_{\delta
}(\beta,\lambda) := E_{P_{0}}[\ell_{rob} \xspace(\beta,\lambda;X)]$ as in
Theorem \ref{Thm-strong-duality}, the minimization can be restricted to the
effective domain,
\begin{align}
\label{Defn-U}\mathbb{U} := \left\{  (\beta,\lambda) \in B \times
\mathbb{R}_{+} : \ E_{P_{0}}\left[ \ell_{rob}(\beta,\lambda;X)\right]  <
\infty\right\} .
\end{align}
Lemma \ref{Lem-eff-dom-f} below, whose proof is presented in Appendix
\ref{Sec-App-Proofs}, provides a characterization of the effective domain
$\mathbb{U}.$ Here recall the earlier definition that $\lambda_{thr}
\xspace(\beta)$ is the $P_{0}-$essential supremum of $\sqrt{\delta}\kappa
\beta^{T} A(x)^{-1}\beta.$ Define,
\begin{align*}
\mathbb{U}_{1} := \left\{  (\beta,\lambda) \in B \times\mathbb{R}_{+}:
\lambda> \lambda_{thr} \xspace(\beta)\right\}  \quad\text{ and }
\quad\mathbb{U}_{2} := \left\{ (\beta,\lambda) \in B \times\mathbb{R}_{+}:
\lambda\geq\lambda_{thr} (\beta)\right\} .
\end{align*}

\begin{lemma}
Suppose that Assumptions \ref{Assump-c} - \ref{Assump-Obj-Conv} hold. Then
for any $\beta \in B, \lambda \geq 0$ and $x \in \mathbb{R}^d,$
\begin{itemize}
\item[a)] $\Gamma^\ast(\beta,\lambda;x)$ is nonempty and $\ell_{rob} \xspace%
(\beta,\lambda;x)$ is finite if $\lambda>\kappa\sqrt{\delta}\beta
A(x)^{-1}\beta;$ and
\item[b)] $\Gamma^\ast(\beta,\lambda;x)$ is empty and $\ell_{rob} \xspace%
(\beta,\lambda;x) = \infty$ if $\lambda < \kappa\sqrt{\delta}\beta
A(x)^{-1}\beta.$
\end{itemize}
Consequently, $\mathbb{U}_1 \subseteq \mathbb{U} \subseteq \mathbb{U}_2.$ %
\label{Lem-eff-dom-f}
\end{lemma}

\begin{lemma}
Suppose that Assumptions \ref{Assump-c}a and \ref{Assump-Obj-Conv} hold.
Then the function $\ell_{rob} \xspace(\beta,\lambda;x)$ is convex in $%
(\beta,\lambda) \in B \times \mathbb{R}_+$ for any $x \in \mathbb{R}^d.$ %
\label{Lem-conv-lrob}
\end{lemma}

Proposition \ref{prop:firstorderinfo} below utilizes envelope theorem (see
\cite{milgrom2002envelope}) to characterize the gradients of $\ell_{rob}
(\cdot).$ Recall that we use $\partial_{-}\ell(u), \partial_+ \ell(u)$ to denote the left and right derivatives of  $\ell(\cdot)$ when evaluated at $u \in \mathbb{R}.$

\begin{proposition}
Suppose that $\ell: \mathbb{R} \rightarrow \mathbb{R}$ satisfies
Assumption \ref{Assump-Obj-Conv} and is of the form
$\ell(u) = \max_{i=1,\ldots,K}\ell_i(u)$ for continuously
differentiable $\ell_i:\mathbb{R} \rightarrow \mathbb{R}$ and a positive integer $K.$  The
following statements hold for $P_0-$almost every $x$:
	\begin{itemize}
\item[a)] The set of maximizers,
$\Gamma^\ast(\beta,\lambda;x)\neq \varnothing,$ for any
$(\beta,\lambda) \in \mathbb{U}_1.$
\item[b)] The maps
$\lambda\mapsto \ell_{rob}(\beta,\lambda;x)$,
$\beta_{j}\mapsto \ell_{rob}(\beta,\lambda;x)$ are
absolutely continuous for
$(\beta,\lambda) \in \mathbb{U}_1,$ and their directional
derivatives are given by,
\begin{subequations}
\label{eqn-directional-derivative}
\begin{align}
\frac{\partial_{-} \ell_{rob}}{\partial \beta_{j} }(\beta,\lambda;x)
&= \min_{\gamma
\in\Gamma^\ast(\beta,\lambda;x)}\partial_{-}\ell\left(\beta^T(x
+ \sqrt{\delta} \gamma A(x)^{-1}\beta)\right) (x +
\sqrt{\delta}\gamma A(x)^{-1}\beta)_j,
\label{eqn-directional-derivative-a}\\
\frac{\partial_{+} \ell_{rob}}{\partial \beta_{j}}(\beta,\lambda;x)
&= \max_{\gamma \in\Gamma^\ast(\beta,\lambda;x)}
\partial_+ \ell\left(\beta^T(x
+ \sqrt{\delta} \gamma A(x)^{-1}\beta)\right) (x +
\sqrt{\delta}\gamma A(x)^{-1}\beta)_j,
\label{eqn-directional-derivative-b}\\
\quad \frac{\partial_{-}
\ell_{rob}}{\partial\lambda}(\beta,\lambda;x)
&=  \min_{\gamma \in\Gamma^\ast(\beta,\lambda;x)}
-\sqrt{\delta}\left( \gamma^2 \beta^T A(x)^{-1}\beta -
1\right),
\label{eqn-directional-derivative-c}\\
\frac{\partial_{+} \ell_{rob}}{\partial \lambda}(\beta,\lambda;x)
&= \max_{\gamma \in\Gamma^\ast(\beta,\lambda;x)}-\sqrt{\delta
}\left( \gamma^2 \beta^T A(x)^{-1}\beta - 1\right).
\label{eqn-directional-derivative-d}
			\end{align}
		\end{subequations}
		Furthermore, $\lambda\mapsto
\ell_{rob}(\beta,\lambda;x)$ is differentiable if and
only if
$\{\frac{\partial_+ F}{\partial \lambda} (\gamma,
\beta, \lambda;x), \frac{\partial_{-} F}{\partial
\lambda} (\gamma, \beta, \lambda;x): \gamma \in
\Gamma^{\ast}(\beta,\lambda;x)\}$ is a
singleton. Likewise, for $j$ in $\{1,\ldots,d\}$
$\beta_{j}\mapsto \ell_{rob}(\beta,\lambda;x)$ is
differentiable if and only if the respective set
$\{\frac{\partial_+ F}{\partial \beta_j} (\gamma,
\beta, \lambda;x), \frac{\partial_{-} F}{\partial
\beta_j} (\gamma, \beta, \lambda;x): \gamma \in
\Gamma^{\ast}(\beta,\lambda;x)\}$ is a singleton. When
all these sets are singleton , if we let
$\tilde{x} := x + \sqrt{\delta}gA(x)^{-1}\beta$ for
any $g \in \Gamma^{\ast}(\beta,\lambda;x)$ then the
derivative is given by,
\begin{align}
\quad\quad\quad \frac{\partial \ell_{rob}}{\partial \beta}(\beta,\lambda;x)
=\ell^{\prime}\left(\beta^T\tilde{x}\right)\tilde{x}
\quad \text{and}\quad
\frac{\partial  \ell_{rob}}{\partial
\lambda}(\beta,\lambda;x)  =
-\sqrt{\delta }\left( g^2 \beta^T
A(x)^{-1}\beta - 1\right).
\label{derivative}
\end{align}
\end{itemize}
\label{prop:firstorderinfo}
\end{proposition}

A proof of Proposition \ref{prop:firstorderinfo} can be found in  Appendix
\ref{Sec-App-Proofs}. Recall that a simple subgradient descent (or) stochastic subgradient descent  for solving
the `non-robust' problem  $\inf_{\beta\in B} E_{P_{0}}\left[  \ell(\beta
^{T}X)\right] $ assumes  access to first-order oracle evaluations $\ell
(\cdot)$ and $ \partial_{+}\ell(\cdot),\partial_{-}\ell(\cdot).$ Likewise, due
to the  characterization in Proposition \ref{prop:firstorderinfo}, all the
function evaluation information required to implement a stochastic
subgradient descent type iterative scheme for minimizing its  robust
counterpart $f_{\delta}(\beta,\lambda)$ are evaluations of  $\ell(\cdot)$ and
$\partial_{+}\ell(\cdot),\partial_{-}\ell(\cdot).$ Indeed, when it is feasible
to  exchange the gradient (or subgradient) and the expectation  operators in
$\nabla_{(\beta,\lambda)}E_{P_{0}}[\ell_{rob} \xspace(\beta,\lambda;  X)]$
(as in Proposition \ref{Prop-sDelta-gradients} in Section  \ref{Sec-First-SGD} below),
the subgradients of  $\ell_{rob} \xspace(\beta,\lambda;X)$ yield noisy
subgradients of  $f_{\delta}(\beta,\lambda).$  For a given  $(\beta,\lambda) \in\mathbb{U}_{1},$ a
univariate optimization  procedure such as bisection (or) Newton-Raphson
methods is  used to solve \eqref{Univar-Opt}. 



\vspace{-0.1in}

\subsection{A stochastic gradient descent scheme for differentiable
$f_{\delta}(\cdot)$.}
\label{Sec-First-SGD}
For ease of notation, we write $\theta$ in  place of $(\beta,\lambda) \in B \times\mathbb{R}_{+}.$ We describe the  algorithm
initially assuming that the conditions in Theorem  \ref{result-smoothness} are
satisfied. Then as a consequence of  Theorem \ref{result-smoothness}, we have
that $f_{\delta}(\cdot)$ is  differentiable over the set $\mathbb{V}$. Here,
recall the  characterization of the set $\mathbb{V}$ in Proposition
\ref{Prop-V-cont-OPT} and the constants $K_{1},K_2$ therein and the constant $R_\beta$ in Assumption \ref{Assump-Compact}. Define the  set,
\begin{align}
\label{Defn-W}\mathbb{W} := \{(\beta,\lambda) \in B \times\mathbb{R}:
K_{1}\Vert\beta\Vert\leq\lambda
\leq K_2R_\beta\}.
\end{align}
See that $\mathbb{W}$ is a closed convex set containing $\mathbb{V}.$
Therefore, when $\delta< \delta_{0},$ as a consequence of Theorem
\ref{Thm-strong-duality} and Proposition \ref{Prop-V-cont-OPT}, we have that
\begin{align*}
\inf_{\beta\in B} \sup_{P: D_{c}(P,P_{0}) \leq\delta} E_{P}\left[  \ell
(\beta^{T}X)\right]  = \inf_{\theta\in\mathbb{W}} f_{\delta}(\theta).
\end{align*}






\begin{proposition}
Suppose that Assumptions \ref{Assump-c} - \ref{Assump-Compact} hold
and $\delta < \delta_0.$ Then
$E_{P_0} [\nabla_{\theta} \ell_{rob}(\theta;X)]$ is well-defined and
\[\nabla_{\theta} f_\delta(\theta) = E_{P_0} [\nabla_{\theta} \ell_{rob}
\xspace(\theta;X)],\] for any
$\theta \in \{(\beta,\lambda): \beta \in B, \lambda >
\lambda_{thr}^\prime \xspace
(\beta)\} \supset \mathbb{W}.$ 
\label{Prop-sDelta-gradients}
\end{proposition}

The proof of Proposition \ref{Prop-sDelta-gradients} is available in Appendix
\ref{Sec-App-Proofs}.

\vspace{-0.1in}

\subsubsection{The iterative scheme.}
\label{sec-iterative-scheme} 
Due to Proposition \ref{Prop-sDelta-gradients},
samples of the random vector $\nabla_{\theta}\ell_{rob} \xspace(\theta; X),$ where $X \sim P_{0},$ are unbiased estimators of the desired gradient
$\nabla_{\theta}f_{\delta}(\theta)$ and are called `stochastic gradients' of
$f_{\delta}(\theta).$ Utilising these noisy gradients, we generate
averaged iterates $\{\bar{\theta}_{k}: k \geq1\}$
according to the following scheme:\newline Fix $\xi\geq0$ and initialize
$\bar{\theta}_{0} = \theta_{0} \in\mathbb{W}.$ For $k > 0,$ given the iterate
$\theta_{k-1}$ from the $(k-1)$-th step,

\begin{itemize}
\item[a)] generate an independent sample $X_{k}$ from the distribution
$P_{0},$

\item[b)] compute $\nabla_{\theta}\ell_{rob} \xspace(\theta_{k};X_{k})$
characterized in \eqref{derivative} by solving $\sup_{\gamma\in\mathbb{R}%
}F(\gamma,\theta;X_{k}),$ and

\item[c)] compute the $k$-th iterate $\theta_{k}$ and its weighted running
average $\bar{\theta}_{k}$ as follows:
\begin{align}
\label{SGD-1}\theta_{k} := \Pi_{\mathbb{W}} \big(\theta_{k-1} - \alpha_{k}
\nabla_{\theta}\ell_{rob} \xspace(\theta_{k-1};X_{k})\big) \quad\text{and }
\quad\bar{\theta}_{k} = \left(  1 - \frac{\xi+ 1}{k + \xi}\right)  \bar
{\theta}_{k-1} + \frac{\xi+ 1}{k + \xi} \theta_{k},
\end{align}
where $\Pi_{\mathbb{W}}(\cdot)$ denotes the projection operation on to the
closed convex set $\mathbb{W}$ and $(\alpha_{k})_{k \geq1}$ is referred to as
the step-size sequence (or) learning rate of the iterative scheme. A closed form expression for the projection $\Pi_{\mathbb{W}}$ is given in Appendix \ref{Sec-Analytic-Projection} and a detailed algorithmic description of the above steps is described in Appendix \ref{sec-algorithm}.
\end{itemize}

\begin{assumption}
The step-size sequence $(\alpha_k)_{k \geq 1}$ is taken to satisfy,
$\alpha_k = \alpha k^{-\tau},$ 
for some constants $\alpha > 0$ and $\tau \in [1/2,1].$ \label%
{Assump-step-sizes}
\end{assumption}

The iterates $(\theta_{k})_{k \geq1}$ are the classical Robbins-Monro iterates
with slower step-sizes (see \cite{robbins1951}). If $\xi= 0$ in the definition
of $\bar{\theta}_{k}$ in \eqref{SGD-1}, the iterate $\bar{\theta}_{k}$ is
simply the running average of $\theta_{1},...\theta_{k-1}$ and the averaging
scheme is the well-known Polyak-Ruppert averaging for stochastic gradient
descent (see \cite{PJ92} and references therein). On the other hand, the
averaging scheme with $\xi> 0$ is referred as polynomial-decay averaging (see
\cite{Shamir:2013:SGD:3042817.3042827}).


\subsubsection{Rates of convergence.\label{subsect_rates_of_convergence}}

Our objective here is to characterize the convergence of $(f_{\delta}%
(\bar{\theta}_{k}))_{k \geq1}$ for the iteration scheme \eqref{SGD-1}. Let
$f_{\ast}:= \inf_{\theta\in B \times\mathbb{R}_{+}} f_{\delta}(\theta)$ be the
optimal value. It is well-known that stochastic gradient descent schemes for
smooth objective functions enjoy $f_{\delta}(\bar{\theta}_{k}) - f_{\ast}=
O_{p}(k^{-1})$ rate of convergence if $f_{\delta}$ is strongly convex and
$f_{\delta}({\theta}_{k}) - f_{\ast}= O_{p}(k^{-1/2})$ if $f_{\delta}$ is
simply convex, for suitable choices of step sizes
(see, for example, \cite{Shamir:2013:SGD:3042817.3042827} and references
therein). While $f_{\delta}(\cdot)$ is convex for all $\delta\geq0,$ it
follows from Theorem \ref{Thm-Str-Convexity} that $f_{\delta}(\cdot)$ is
locally strongly convex in the region containing the optimizer when $\delta<
\delta_{1}.$ As a result, we have the following better rate of convergence for
$f_{\delta}(\bar{\theta}_{k}) - f_{\ast}$ when $\delta< \delta_{1}.$ The proof
of Proposition \ref{Prop-Conv-SGD-1} is presented in Section
\ref{Sec-RatesofConv-Proofs}.

\begin{proposition}
Suppose that Assumptions \ref{Assump-c} - \ref{Assump-Compact} hold. Then we have,
\begin{itemize}
\item[a)] $f_\delta(\bar{\theta}_k) - f_\ast = O_p(k^{-1/2})$ if $\delta <
\delta_0, $ $\xi \geq 1$ in \eqref{SGD-1} and $\tau = 1/2$ in Assumption \ref%
{Assump-step-sizes}; 
\item[b)]
$f_\delta(\bar{\theta}_k) - f_\ast = O_p(
k^{-1})$ if
$\delta < \delta_1,$ $\xi = 0,$ $\tau \in (1/2,1)$ in Assumption
\ref%
{Assump-step-sizes}, and Assumption \ref{Assump-nondegeneracy} is
satisfied.
\end{itemize}
\label{Prop-Conv-SGD-1}
\end{proposition}

For the strongly convex case, the averaged procedure endows the sequence
$(f_{\delta}(\bar{\theta}_{k}))_{k \geq1}$ with the robustness property that
the precise choice of step-size $(\alpha_{k})_{k \geq1}$ does not affect the
convergence behaviour as long as the step size choice satisfies Assumption
\ref{Assump-step-sizes}. Contrast this with the vanilla stochastic
approximation iterates $(\theta_{k})_{k \geq1}$ with step-size $\alpha_{k} =
\alpha k^{-1},$ in which case the constant $\alpha$ has to be chosen larger
than a threshold that depends on the Hessian of $f_{\delta}$ at $\theta$
minimizing $f_{\delta}(\theta),$ in order to have $f_{\delta}(\theta_{k}) -
f_{\ast}= O_{p}(k^{-1})$ (see, for example, \cite{NIPS2011_4316,nemirovski2009robust} for discussions on the effect of step sizes on error
$f_{\delta}(\theta_{k}) - f_{\ast}$).

Recall that $\delta_{0},\delta_{1}$ are positive constants that do not depend
on the size of the support of $P_{0}.$ For data-driven optimization problems,
the radius of ambiguity, $\delta,$ is typically chosen to decrease to zero
with the number of data samples $n$ (see, for example,
\cite{blanchet2016robust,shafieezadeh-abadeh_distributionally_2015}).
Therefore the requirement that $\delta< \delta_{1}$ is typically satisfied in
practice in data-driven applications.

Indeed if $\delta<\delta_{1},$ due to Proposition \ref{Prop-Conv-SGD-1}b), it
suffices to terminate after $O_{p}(\varepsilon^{-1})$ iterations
in order to obtain an iterate $\bar{\theta}_{k}$ that satisfies $f_{\delta
}(\bar{\theta}_{k})-f_{\ast}\leq\varepsilon.$ On the other hand, if
$\delta>\delta_{1},$ we require the usual $O_{p}(\varepsilon^{-2})$ iteration
complexity to obtain $f_{\delta}(\theta_{k})-f_{\ast}\leq\varepsilon$, which
is identical to the sample complexity of stochastic gradient descent for the
non-robust problem $\inf_{\beta}E_{P_{0}}[\ell(\beta^{T}X)]$ in the presence
of convexity (see, for example, \cite{Shamir:2013:SGD:3042817.3042827}). Here,
recall from the discussion following Theorem \ref{Thm-Str-Convexity} that the
non-robust stochastic optimization objective $\inf_{\beta}E_{P_{0}}[\ell
(\beta^{T}X)]$ need not be strongly convex even if $\ell(\cdot)$ is strongly
convex, whereas the corresponding worst-case objective $f_{\delta}%
(\beta,\lambda)$ is jointly strongly convex in $(\beta,\lambda)$ more
generally under the conditions identified in Theorem \ref{Thm-Str-Convexity}.

As a result, if we let $L$ denote the complexity of the univariate line search
that solves $\sup_{\gamma\in\mathbb{R}} F(\gamma, \theta; x)$ for any
$(\beta,\lambda) \in\mathbb{W},$ then the computational effort involved in
solving \eqref{DRO_B0} scales as $O_{p}(\varepsilon^{-1}L)$ when
$\delta< \delta_{1}$ and $O_{p}(\varepsilon^{-2}L)$ when $\delta\in[\delta
_{1},\delta_{0}).$ As mentioned earlier, this complexity does not scale with
the size of the support of $P_{0}$ for a given $\delta.$ See Appendix
\ref{Sec-Line-Search} for a brief discussion on $L,$ the complexity introduced
by line search schemes.

The analysis of stochastic gradient descent with small bias can be done without significant complications under regularity conditions. The following result summarizes the overall rate of convergence analysis for the classical Robbins-Monro iterates ($\theta_k: k \geq 1$), including bias induced by the line search, in the strongly convex case. The proof
of Proposition \ref{Prop-Conv-SGD-Over-All} is presented in Appendix
\ref{Sec-App-Proofs}.

\begin{proposition}
 Suppose that Assumptions \ref{Assump-c} - \ref{Assump-step-sizes} hold and
 $\delta<\delta_1$. At the $k$-th iteration, the bisection method is employed with at least $\tau\log_2(k)-\log_2(\alpha) + 2\log_2(1+\|X_{k}\|)$ cuts to compute $\nabla_{\theta}\ell_{rob} \xspace(\theta_{k-1};X_{k})$.
 Then we have,
 \begin{itemize}
 \item[a)] $f_\delta(\theta_k) - f_\ast = O_p(k^{-\tau})$ if $\tau \in (1/2,1)$ in Assumption \ref{Assump-step-sizes};
 \item[b)] $f_\delta(\theta_k) - f_\ast = O_p(k^{-1})$ if $\alpha$ is larger than the smallest eigenvalue of $\nabla_{\theta}^2f_\delta(\theta_\ast)$ and $\tau=1$ in Assumption \ref{Assump-step-sizes};
\end{itemize}
\label{Prop-Conv-SGD-Over-All}
\end{proposition}
\begin{remark}
Proposition \ref{Prop-Conv-SGD-Over-All} indicates that if the bisection method is applied with $O(\log_2(k))$ cuts at $k$-th iterates, then the classical Robbins-Monro algorithm still achieves the optimal $O_p(1/k)$ rate even if the bias of line search is taken into consideration. Assumption in part b) on requiring a lower bound on $\alpha$ is standard. Typically, avoiding an estimate of such a lower bound can be done by Polyak-Ruppert-Juditsky  averaging  and  choosing $\tau\in(1/2,1)$. This is most often studied in the case of unbiased gradients. An adaptation is required for the case of biased gradients. While we believe that such an adaptation should be quite doable, we do not pursue it in this paper as it would be a significant distraction from our objective. Our goal here is to showcase the applicability of the structural results in Section \ref{Sec-Convexity} towards designing efficient algorithms for DRO based on flexible cost functions.
\end{remark}

To complete this discussion, recall that the dual formulation,
\[
\inf_{\lambda\geq0} E_{P_{0}}\left[ \sup_{\gamma\in\mathbb{R}} F(\gamma
,\beta,\lambda;X)\right] ,
\]
that we are working with is is a result of the change of variables $c =
\sqrt{\delta}\gamma\beta^{T} A(X)^{-1} \beta$ and $\lambda\sqrt{\delta}$ to
$\lambda$ in the proof of Theorem \ref{Thm-strong-duality}. Evidently, these
change of variables involve scaling by a factor $\sqrt{\delta}.$ It is a
consequence of this scaling by $\sqrt{\delta}$ that
an optimal $\lambda_{\ast}(\beta)$ is bounded, thus allowing the optimization
to be restricted to values of $\lambda$ over a compact interval $[0,K_{2}%
R_{\beta}]$ regardless of how small the radius of ambiguity $\delta$ is.
Moreover, if we let $g_{\delta}(x)$ denote a maximizer for the inner
maximization $\sup_{\gamma\geq0}F(\gamma,\beta,\lambda_{\ast}(\beta);x)$ for
any $\delta,x$ and a fixed $\beta\in B,$ we shall also witness in Proposition
\ref{prop:secondorder-info-smoothcase}b that $g_{\delta}(X) = O_{p}(1),$ as $\delta
\rightarrow0.$ These two properties ensure that the inner and outer
optimization problems $\inf_{\lambda\geq0} E_{P_{0}}\left[ \sup_{\gamma
\in\mathbb{R}}  F(\gamma,\beta,\lambda;X)\right] $ are well-conditioned and
their solutions remain scale-free (with respect to $\delta$).

For algorithms that directly proceed with the dual reformulation in
\cite[Theorem 1]{blanchet_quantifying_2016} or \cite[Theorem 1]%
{gao2016distributionally} without employing the above described scaling of
variables by factor $\sqrt{\delta},$ the resulting dual formulation will have
the property that the solutions to the inner and outer optimization problems
are $O_{p}(\sqrt{\delta})$ and $O(\delta^{-1/2})$ respectively. Consequently,
the local strong convexity coefficient of the dual reformulation obtained
without scaling can be shown to be $O(\delta),$ which is inferior when
compared to the $O(\sqrt{\delta})$ strong convexity coefficient that we have
identified in Theorem \ref{Thm-strong-duality}. Indeed, the focus on strong
convexity and its effect of computational performance in this paper has helped
bring out this nuanced and important effect of the scaling that appears to be
absent in the existing algorithmic approaches for Wasserstein DRO.

\vspace{-0.1in}

\subsection{Enhancements to the SGD scheme in Section \ref{Sec-First-SGD}.}

\label{Sec-Mod-SGD} Our focus in this section is to describe natural
enhancements to the vanilla SGD scheme described in Section
\ref{Sec-First-SGD} by utilizing the convexity characterizations in Section
\ref{Sec-Convexity}.

\vspace{-0.1in}

\subsubsection{A two-time scale stochastic approximation scheme.}

\label{Sec-two-timescale-SGD} Since $\lambda$ is an auxiliary variable
introduced by the duality formulation, it is rather natural to update the
variables $\beta$ and $\lambda$ at different learning rates (step sizes) as
follows: Given iterate $(\beta_{k-1},\lambda_{k-1}),$ generate a sample
$X_{k}$ independently from $P_{0}$ in order to update as follows:
\begin{subequations}
\begin{align}
\tilde{\beta}_{k}  & = \beta_{k-1} - \alpha_{k} \frac{\partial f_{\delta}%
}{\partial\beta}(\beta_{k-1},\lambda_{k-1};X_{k})\label{Two-Time-Scale-1}\\
\tilde{\lambda}_{k}  & = \lambda_{k-1} - \gamma_{k} \frac{\partial f_{\delta}%
}{\partial\lambda}(\beta_{k-1},\lambda_{k-1};X_{k}), \text{ and }%
\label{Two-Time-Scale-2}\\
\theta_{k}  & = \Pi_{\mathbb{W}} \left(  (\tilde{\beta}_{k},\tilde{\lambda
}_{k})\right) .
\end{align}
where the step-sizes $(\alpha_{k})_{k \geq1}, (\gamma_{k})_{k \geq1}$ satisfy
the step-size requirement in Assumption \ref{Assump-step-sizes} with $\tau
\in(1/2,1)$ and $\alpha_{k}/\gamma_{k} \rightarrow0.$ Since $\alpha_{k}$ is
very small relative to $\gamma_{k},$ the iterates $\beta_{k}$ remain
relatively static compared to $\lambda_{k},$ thus having an effect of fixing
$\beta_{k}$ and running \eqref{Two-Time-Scale-2} for a long time. As a result,
the iterates $\lambda_{k}$ appear ``most of the time'' as $\lambda_{\ast
}(\beta_{k})$ in the view of $\beta_{k},$ thus resulting in effective updates
of the form,
\end{subequations}
\begin{align*}
\beta_{k} = \beta_{k-1} - \alpha_{k} \frac{\partial f_{\delta}}{\partial\beta}
(\beta_{k-1}, \lambda_{\ast}(\beta_{k-1}); X_{k}).
\end{align*}

Once again, we consider the averaged iterates $\bar{\theta}_{k},$ defined as
in \eqref{SGD-1} with $\xi= 0.$ Similar to Section \ref{Sec-First-SGD}, if we
let $f_{\ast}:= \inf_{\theta\in B \times\mathbb{R}_{+}} f_{\delta}(\theta),$
it can be argued that $f_{\delta}(\bar{\theta}_{k}) - f_{\ast}= O_{p}(k^{-1})$ in the presence of strong convexity (see \cite[Theorem
2]{mokkadem2006}) that holds in the $\delta< \delta_{1}$ case.
As a result, if $\delta< \delta_{1},$ it suffices to terminate after
$O_{p}(\varepsilon^{-1})$ iterations in order to obtain an
iterate $\bar{\theta}_{k}$ that satisfies $f_{\delta}(\bar{\theta}_{k}) -
f_{\ast}\leq\varepsilon. $
We leave it as a question for future research to develop a precise
understanding of the effect of two time scales in affecting the convergence behaviour.


\vspace{-0.1in}

\subsubsection{Line search based SGD scheme.}

\label{Sec-GSS-SGD} When $\delta< \delta_{0},$ Theorem \ref{result-smoothness}
asserts that $f_{\delta}(\beta,\lambda)$ satisfies strong convexity in the
variable $\beta$ for every fixed $\lambda.$ This strong convexity in variable
$\beta$ holds even if $f_{\delta}(\beta,\lambda)$ may not be jointly strongly
convex in $(\beta,\lambda)$ (for example, when $\delta\in[\delta_{1}%
,\delta_{0})).$ We make use of this observation in this section to describe an
SGD scheme that a) quickly evaluates $h(\lambda) := \inf_{\beta\in B}%
f_{\delta}(\beta,\lambda)$ for any given $\lambda$ and b) utilizes univariate
line search for minimizing $h(\cdot)$ in a suitable interval.

Since $f_{\delta}(\cdot)$ is a convex function, the partial minimization
$h(\lambda) := \inf_{\beta\in B}f_{\delta}(\beta,\lambda)$ defines a
univariate convex function in $\lambda.$ For any fixed $\lambda> 0,$ consider
stochastic gradient descent iterates of the form,
\begin{align*}
\beta_{k}  & := \beta_{k-1} - \alpha_{k} \frac{\partial f_{\delta}}%
{\partial\beta} (\beta_{k-1},\lambda;X_{k}), \quad\text{ and } \quad\bar
{\beta}_{k} := \frac{1 }{k}\sum_{i=1}^{k}\beta_{i},
\end{align*}
where $(X_{k})_{k \geq1}$ are i.i.d. samples of $P_{0}$ and the step-sizes
$(\alpha_{k})_{k \geq1}$ satisfy the requirement in Assumption
\ref{Assump-step-sizes} with $\tau\in(1/2,1)$ and $\xi= 0.$ Then it follows
from the strong convexity characterization in Theorem \ref{result-smoothness}
that $f_{\delta}(\bar{\beta}_{k},\lambda) - h(\lambda) = O_{p}(k^{-1})$ if $\delta< \delta_{0}.$ With the ability to evaluate
the function $h(\lambda) = \inf_{\beta\in B}f_{\delta}(\beta,\lambda)$ within
desired precision, any standard line search method, such as triangle section
method (see Algorithm 3 in \cite{Triangle_Search}), that exploits convexity of
$h(\cdot)$ to achieve linear convergence for line search can be employed to
evaluate $\min_{\lambda} h(\lambda)$ to any desired precision.

With line searches requiring identification of an interval (where the minimum
is attained) to begin with, we restrict the line search over $\lambda$ to the
interval $[0,K_{2}R_{\beta}]$. This is because, due to
Proposition \ref{Prop-V-cont-OPT} and that $\Vert\beta\Vert\leq R_{\beta}$, we
have that the interval $[0,K_{2}R_{\beta}]$ contains optimal $\lambda_{\ast}(\beta)$
for every $\beta\in B.$ It can be argued that the described approach results
in iteration complexity of $O_p(\varepsilon^{-1}\text{poly}%
(\log\varepsilon^{-1}))$ to solve $\min f_{\delta}(\beta,\lambda)$ within
$\varepsilon-$precision when $\delta<\delta_{0}.$ We do not pursue this
derivation here as our objective is to simply demonstrate the versatility of
applications of the structural insights given by the results in Section
\ref{Sec-Convexity}.

Likewise, one could consider a variety of algorithms that accelerate SGD at a
greater computational cost per iteration; such algorithms utilize either
variance reduction (see, for example, \cite{NIPS2013_4937, NIPS2014_5258}), or
momentum based acceleration (see \cite{Allen-Zhu:2017:KFD:3055399.3055448}).
The strong convexity results in Section \ref{Sec-Convexity} could be used to
establish improved rates of convergence for such extensions as well.

\vspace{-0.1in}

\subsection{SGD for nondifferentiable $f_{\delta}$.}
\label{Sec-Large-delta} 
The function $f_{\delta}(\cdot)$ need not be differentiable when the radius of ambiguity $\delta$ exceeds $\delta_{0}$ (or) when the set $B$ is not bounded.  The
iterative algorithms described in Sections \ref{Sec-First-SGD} and
\ref{Sec-Mod-SGD} rely on restricting the iterates $\theta_{k}$ to the set
$\mathbb{W}.$ Such an approach is not feasible when $\delta> \delta_{0}.$ In that case, with the characterization of the
effective domain of $f_{\delta}$ as in Lemma \ref{Lem-eff-dom-f},
define the family of closed convex sets, $(\mathbb{U}_{\eta}: \eta\geq0)$ as,
\begin{align}
\label{Defn-Ueta}\mathbb{U}_{\eta}:= \left\{  (\beta,\lambda) \in B
\times\mathbb{R}_{+}: \lambda\geq\lambda_{thr} \xspace(\beta) + \eta\right\} .
\end{align}
Let $\partial f_{\delta}(\beta,\lambda)$ and $\partial\ell_{rob}
\xspace (\beta,\lambda;x),$ respectively, be the set of subgradients of
$f_{\delta}(\cdot)$ and $\ell_{rob} \xspace(\ \cdot\ ;x)$ at $(\beta
,\lambda).$ Likewise, let $\partial\ell(u) := \text{conv}\{\partial_{-}%
\ell(u)/\partial u , \, \partial_{+}\ell(u)/\partial u \}$ denote the
subgradient set of the univariate function $\ell(\cdot)$ evaluated at $u.$
Then it follows from Proposition \ref{prop:firstorderinfo}b that the set,
\begin{align}
D(\beta,\lambda;x) := \text{conv}\left\{
\begin{pmatrix}
\partial\ell(\beta^{T}\tilde{x})\tilde{x}\\
\sqrt{\delta}\left( 1- g^{2}\beta^{T}A(x)^{-1}\beta\right)
\end{pmatrix}
:
\begin{array}
[c]{c}%
\tilde{x} = x + \sqrt{\delta} gA(x)^{-1}\beta,\\
g \in\Gamma^{\ast}(\beta,\lambda;x)
\end{array}
\right\} \label{noisy-subgradient}%
\end{align}
comprise the subgradient set $\partial\ell_{rob} \xspace(\beta,\lambda;x).$
Similar to Proposition \ref{Prop-sDelta-gradients}, Proposition
\ref{Prop-sDelta-subgradients} below helps in characterizing noisy
subgradients of $f_{\delta}(\cdot).$

\begin{proposition}
Suppose that Assumptions \ref{Assump-c}-\ref{Assump-Obj-Conv} are satisfied and the loss  $\ell(\cdot)$ is of the form $\ell(u) = \max_{i=1,\ldots,K}\ell_i(u)$ for continuously differentiable $\ell_i:\mathbb{R} \rightarrow \mathbb{R}$ and a positive integer $K.$
For any $\eta > 0$ and
fixed $%
(\beta, \lambda) \in \mathbb{U}_\eta,$ let $(X,h(\beta,\lambda;X))$
be such that $X \sim P_0$ and
$h(\beta,\lambda,X) \in D(\beta,\lambda;X),$ $P_0-$%
almost surely. Then $E[h(\beta,\lambda;X)]$ is well-defined and $%
E[h(\beta,\lambda;X)] \in \partial f_\delta(\beta,\lambda).$
\label{Prop-sDelta-subgradients}
\end{proposition}

The proof of Proposition \ref{Prop-sDelta-subgradients} is available in
Appendix \ref{Sec-App-Proofs}. Following Proposition
\ref{Prop-sDelta-subgradients}, consider an iterative scheme utilizing noisy
subgradients as follows. Given fixed $\eta> 0, \xi\geq1$ and iterate
$\theta_{k-1} = (\beta_{k-1},\lambda_{k-1})$ from $(k-1)$-st iteration, the $k
$-th iterate is computed as follows:
\begin{align}
\theta_{k} := \Pi_{\mathbb{U}_{\eta}} \big(\theta_{k-1} - \alpha_{k}
H_{k}\big) \quad\text{and } \quad\bar{\theta}_{k} = \left(  1 - \frac{\xi+
1}{k + \xi}\right)  \bar{\theta}_{k-1} + \frac{\xi+ 1}{k + \xi} \theta
_{k},\label{SGD-Nonsmooth}%
\end{align}
where the step-size sequence $(\alpha_{k})_{k \geq1}$ satisfies Assumption
\ref{Assump-step-sizes} with $\tau= 1/2$ and $H_{k}$ is computed as follows:

\begin{itemize}
\item[a)] Generate a sample $X_{k}$ independently from the distribution
$P_{0};$

\item[b)] Pick any $g \in\Gamma^{\ast}(\beta,\lambda;X_{k})$ by solving the
univariate search $\sup_{\gamma\in\mathbb{R}}F(\gamma, \beta_{k-1}%
,\lambda_{k-1};X_{k});$

\item[c)] Let $\tilde{X}_{k} := X_{k} + \sqrt{\delta} g A(X_{k})^{-1}\beta,$
and take $H_{k} \in D(\beta_{k-1},\lambda_{k-1};X_{k})$ as,
\begin{align*}
H_{k} :=
\begin{pmatrix}
L^{\prime}\tilde{X}_{k}\\
\sqrt{\delta}\left( 1- g^{2}\beta_{k-1}^{T}A(X_{k})^{-1}\beta_{k-1}\right)
\end{pmatrix}
,
\end{align*}
where $L^{\prime}$ is selected uniformly at random from the interval
$[\partial_{-}\ell(\beta_{k-1}^{T}\tilde{X}_{k})/\partial u , \, \partial
_{+}\ell(\beta_{k-1}^{T}\tilde{X}_{k})/\partial u] =: \partial\ell(\beta
_{k-1}^{T}\tilde{X}_{k}).$
\end{itemize}

It is immediate from \eqref{noisy-subgradient} that $H_{k} \in D(\beta
_{k-1},\lambda_{k-1};X_{k}).$ Then due to Proposition
\ref{Prop-sDelta-subgradients}, we have that $EH_{k} \in\partial f_{\delta}(\beta
_{k-1},\lambda_{k-1}).$
Due to the convexity of $f_{\delta}(\cdot)$ characterized in Theorem
\ref{Thm-Str-Convexity}, we have the following rates of convergence for
$f_{\delta}(\bar{\theta}_{k}) - f_{\ast},$ as $k \rightarrow\infty.$ The proof
of Proposition \ref{Prop-Conv-SGD-2} is presented in Appendix
\ref{Sec-App-Proofs}.

\begin{proposition}
Suppose that Assumptions \ref{Assump-c}-\ref{Assump-Obj-Conv} are satisfied and the loss  $\ell(\cdot)$ is of the form $\ell(u) = \max_{i=1,\ldots,K}\ell_i(u)$ for continuously differentiable $\ell_i:\mathbb{R} \rightarrow \mathbb{R}$ and a positive integer $K.$ In addition, suppose that the constants $\xi$ in \eqref{SGD-Nonsmooth} and $\tau$ in Assumption \ref{Assump-step-sizes} are such that $\xi \geq 1$ and $\tau = 1/2.$ Then we have
$f_\delta(\bar{\theta}_k) - f_\ast \leq \eta \sqrt{\delta} +
O_p(k^{-1/2}).$ \label%
{Prop-Conv-SGD-2}
\end{proposition}

Consequently, if we choose $\eta$ small enough and use $L$ to denote the
computational effort needed to solve the line search $\sup_{\gamma}%
F(\gamma,\beta,\lambda;X)$ for any $(\beta,\lambda) \in\mathbb{U}_{\eta},$
then the total computational effort needed to obtain estimates of $f_{\ast}$
within $\varepsilon-$precision is $O_{p}(L\varepsilon^{-2}).$ A brief
description of the complexity $L$ introduced by the line search can be found
in Appendix \ref{Sec-Line-Search}.


\section{Numerical experiments}
\label{Sec-Eg}
In this section, we provide some illustrative examples in the contexts of supervised learning and portfolio optimization. All the numerical examples were carried out in a laptop computer with a 2.2GHz Inter Core i7 CPU and 16GB memory. We keep in mind that our goal in this section is to demonstrate empirically the structural properties that we derived and their implications for algorithmic performance. We are not concerned with a specific choice of $\delta$, which is typically done via cross validation in a typical data driven setting.

\subsection{Illustrative examples from supervised learning.}
\label{subsec-supervised-earling}
 The out-of-sample performance advantages of utilizing optimal transport costs with Mahalanobis distances have been demonstrated comprehensively with real data classification examples in \cite{blanchet2017data}.
Therefore, in the interest of space and to avoid repetition, we restrict the focus in this subsection to reporting the results of stylized numerical experiments which accomplish the following enumerated goals: 1) compare the iteration complexity of the iterative scheme  proposed
in Section \ref{Sec-First-SGD} for the DRO formulation  \eqref{DRO_B0} with that of the benchmark stochastic gradient  descent for its non-robust
counterpart \eqref{Stoch_Opt_Model}; 2) provide a visualization of the worst case distribution; and 3)  study the iteration complexity when the twice  differentiability assumption (made in order to prove Theorem  \ref{Thm-Str-Convexity}) is relaxed.

\vspace{-0.1in}

\subsubsection{Modifications of notations for supervised learning.}
As supervised learning problems typically involve a response variable in
addition to the predictor variables $X,$ we first discuss how the DRO
formulation in \eqref{DRO_B0} can be utilized in the presence of the
additional response variable. Let us use $Y$ to denote the response variable
in the rest of this section. We begin by
treating the response $Y$ as a random parameter of the loss function
$\ell(\cdot)$, so the assumptions applied to $\ell(\cdot)$ should be replaced
by that of $\ell(\,\cdot\ ;Y)$ when considering problems with response
variable $Y.$ In addition, the reference measure $P_{0}\in\mathcal{P}%
(\mathbb{R}^{d}\times\mathbb{R})$ is modified to characterize the joint
distribution of $(X,Y)$. Further, as we assume the ambiguity only appears on
the predictors $X,$ we defined the optimal transport between $P\in
\mathcal{P}(\mathbb{R}^{d}\times\mathbb{R})$ and $P_{0}\in\mathcal{P}%
(\mathbb{R}^{d}\times\mathbb{R})$ can be modified as,
\begin{align*}
D_{c}\left(  P_{0},P\right)  =\inf\left\{ E_{\pi}\left[ c\left(  X,X^{\prime
}\right)  \right]  :
\begin{array}
[c]{ll}%
\pi\in\mathcal{P}\big(\mathbb{R}^{d}\times\mathbb{R}\times\mathbb{R}^{d}%
\times\mathbb{R}\big), \pi(Y=Y^{\prime}) = 1, & \\
\pi_{(X,Y)}=P_{0},\pi_{(X^{\prime},Y^{\prime})}=P. &
\end{array}
\right\} ,
\end{align*}
where $\pi$ is the joint distribution of $(X,Y,X^{\prime},Y^{\prime}).$

Using the modified model, if $\ell(\cdot;y)$ satisfies the assumptions of
$\ell(\cdot)$ for $P_{0}-$almost every $y$, then all the results and
algorithms developed in the previous sections are still valid. The proof of
the generalized result is essentially same as before, as we just need to
replace $\ell(\cdot)$ by $\ell(\cdot;Y)$ in the proof as well.

\vspace{-0.1in}

\subsubsection{Logistic regression.}

We consider the case of binary classification, where the data is given by
$\{(X_{i},Y_{i})\}_{i = 1}^{n}$, with predictor $X_{i}\in\mathbb{R}^{d}$ and
label $Y_{i}\in\{-1,1\}$. In this case, the logistic loss function is
\[
\ell(u;y) = \log\left( 1+\exp(-yu)\right) .
\]
We are interested in solving the distributionally robust logistic regression
problem,
\begin{align*}
\inf_{\beta\in B} \sup_{P: D_{c}(P_{n},P) \leq\delta} E_{P}\left[  \ell
(\beta^{T}X;Y)\right]
\end{align*}
where $P_{0} = P_{n}(dx,dy):= \frac{1}{n}\sum_{i=1}^{n}\delta_{\{(X_{i},Y_{i}%
)\}}(dx,dy)$ is the empirical measure of data.

 In Appendix \ref{Sec-table} we demonstrate that the assumptions in Section \ref{Sec-Convexity} are naturally satisfied by the logistic loss $\ell(\cdot\,;y),$ and therein we also include computation of related constants. Consequently, all of the algorithms and theoretical results developed in this paper are applicable to the logistic regression example.





We design a numerical experiment to test the performance of our algorithm on
distributionally robust logistic regression. The data is generated from normal
distribution, with different mean for each class and same variance. 
The total
number of data points ranges among $n\in\{64, 256, 1024\}$, and the dimension of data is $d=32$.

We implement the iterative scheme provided in Section
\ref{sec-iterative-scheme} to solve the ordinary logistic regression (with
$\delta=0)$ and its distributionally robust counterpart ($\delta>0)$. In the numerical experiment we choose $A(x) = \mathbb{I}_d$.
To compare the rates of convergence of these two models, same learning rate (or
step size) on $\beta$ is adapted. The parameter $\tau$ in Assumption
\ref{Assump-step-sizes} is chosen to be 0.55. We use the value of loss
function at $10^{5}$ iterations as the approximate optimal loss, then we plot
the optimality gap (Error) versus number of iterations for DRO-model and
ordinary logistic model in Figure \ref{fig-logistic-fig2}.

\begin{figure}[ptb]
\centering
\begin{subfigure}[]{5.3cm}
    \centering
	\includegraphics[width = 5.3cm]{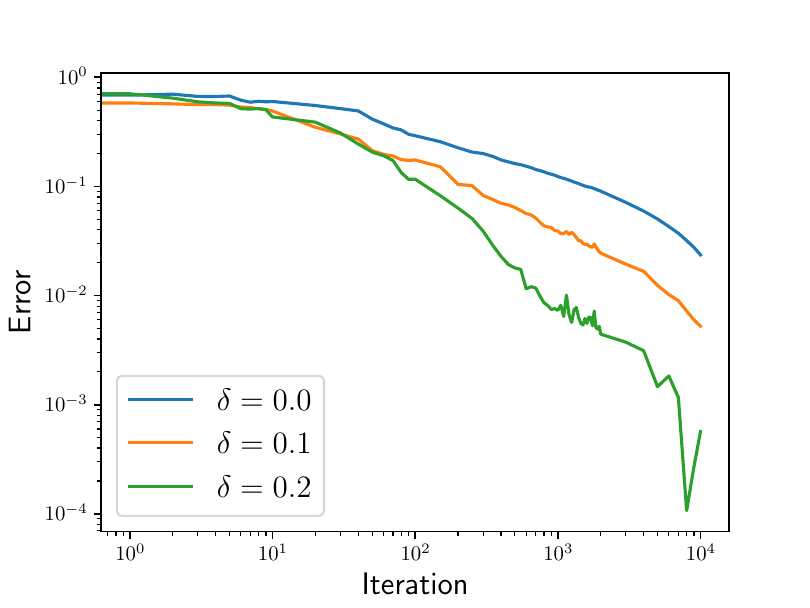}
	\caption{$n = 64$}
\end{subfigure}
\begin{subfigure}[]{5.3cm}
    \centering
    \includegraphics[width=5.3cm]{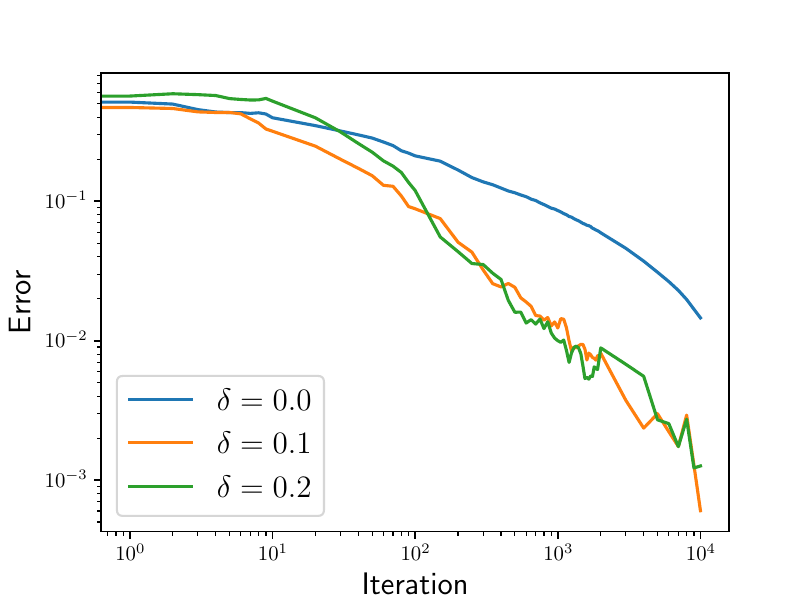}
    \caption{$n = 256$}
\end{subfigure}
\begin{subfigure}[]{5.3cm}
    \centering
	\includegraphics[width = 5.3cm]{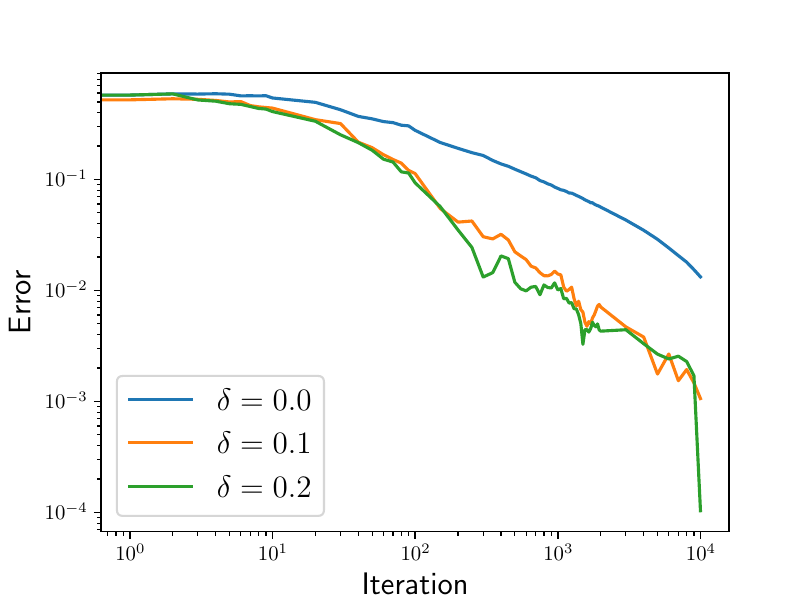}
	\caption{$n = 1024$}
\end{subfigure}
\caption{Convergence of loss function for logistic regression}
\label{fig-logistic-fig2}
\end{figure}


\begin{figure}[ptb]
\vspace{-0.15in} \hspace{-0.81in}
\includegraphics[width = 8in]{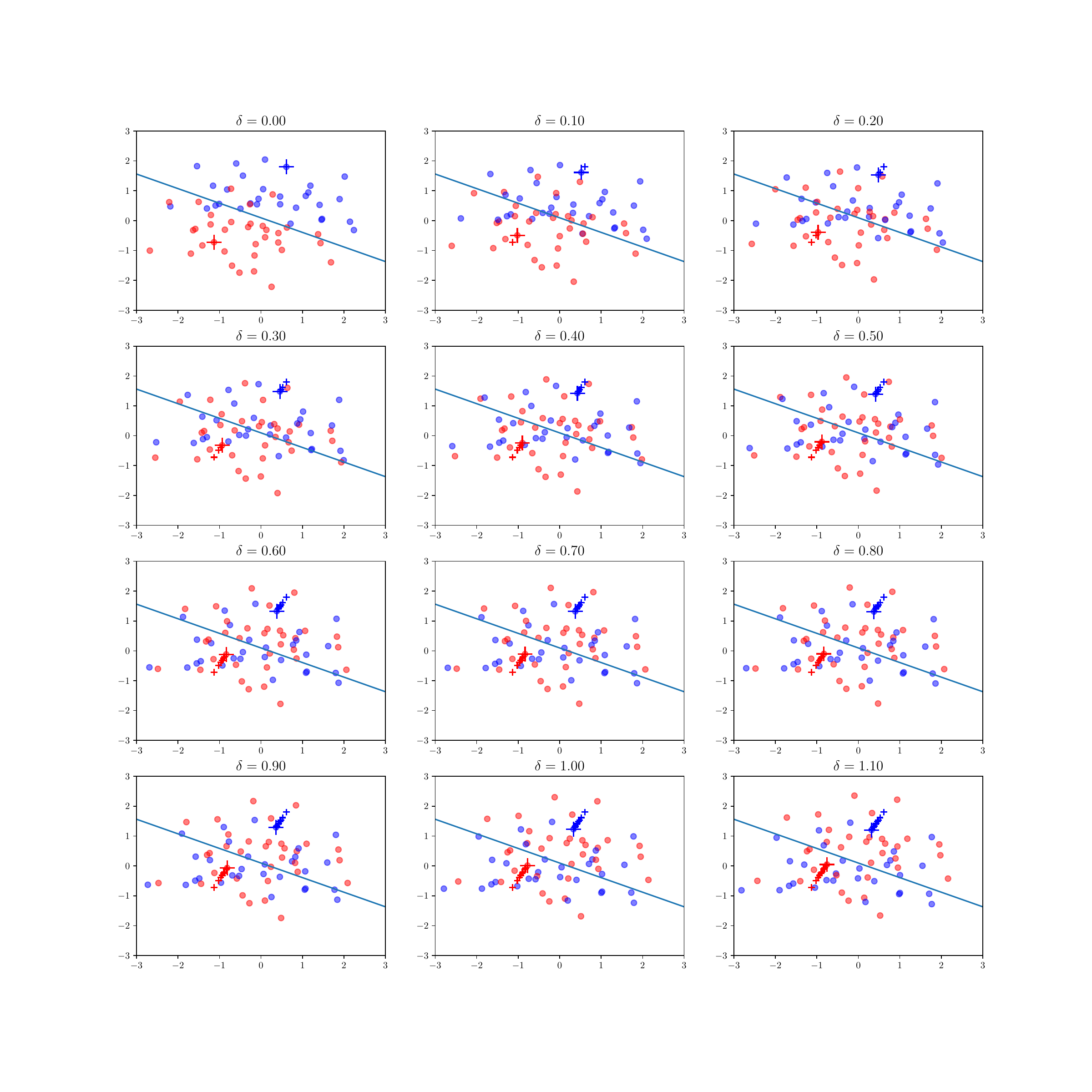} \vspace
{-0.8in}\caption{Decision boundary and worst case distribution. To facilitate tracking the change of $X_{\delta}^{\ast}$ when $\delta$ is increasing, we select one point from each class and use a big $+$ to mark its position. We also employ a small $+$ to mark its previous position when $\delta$ is smaller so that the trajectory of the point is visible. We can observe, as predicted by our theoretical results, that $X_{\delta}^{\ast}$ moves parametrically in a linear direction as $\delta$ changes. Moreover, the speed of displacement is decreasing, which is consistent with the $\sqrt{\delta}$ scaling size discussed in Theorem \ref{Prop-Comp-Stat}. It is worth noting the dynamics of the worse-case distribution, which transports the different classes in opposite directions in order to maximize the loss for misclassification.}
\label{fig-logistic-fig1}%
\end{figure}

Next, in the sequence of subplots in Figure \ref{fig-logistic-fig1}, we attempt to visualize how the worst case distributions $\{X_{\delta}^{\ast}:\delta > 0\}$ change as the radius $\delta$ is increased. In the first subplot corresponding to $\delta = 0,$ we have 64 independent samples of $X \in \mathbb{R}^2$ and the decision boundary obtained from the ordinary logistic regression.  The dots in different color denote the data in different classes: on the lower left side the data is classified to be red and on the upper right side the data is classified to be blue. Naturally, when $\delta= 0$ most of the data points are correctly classified.
Then, fixing the decision boundary to be the same as that obtained from the ordinary logistic regression, we increase the transportation budget $\delta$ and display the respective worst case
distribution computed with $\beta$ fixed to that obtained from the  ordinary logistic regression estimator. The
worst case distributions $X_{\delta}^{\ast}$ for different $\delta$ are visualized in the
subsequent plots. We can observe that more and more points are misclassified when $\delta$ is increasing, and in the last plot the misclassification rate is
larger than 50\%. In addition, the trajectory of $X_{\delta}^{\ast}$ forms a straight line moving towards the wrong side of the decision boundary, which are aligned with our observations pertaining to comparative statics in Theorem \ref{Thm-WC-Dist} (see Section \ref{Sec-Comp-Statics}).

\vspace{-0.1in}

\subsubsection{Linear regression.}
\label{sec-linear-regression} Now we turn to consider the example of linear
regression with squared loss function. In this data is given by $\{(X_{i}%
,Y_{i})\}_{i = 1}^{n}$, with predictor $X_{i}\in\mathbb{R}^{d}$ and label
$Y_{i}\in\mathbb{R}$. We consider the squared loss function $\ell(u;y) =
(y-u)^{2}$ in this example, and the reference measure is defined as the
empirical measure $P_0 = P_{n}(dx,dy):= \frac{1}{n}\sum_{i=1}^{n}\delta_{\{(X_{i}%
,Y_{i})\}}(dx,dy)$. Then, the distributionally robust linear regression problem
is defined as
\begin{align*}
\inf_{\beta\in B} \sup_{P: D_{c}(P_{n},P) \leq\delta} E_{P}\left[  \ell
(\beta^{T}X;Y)\right] .
\end{align*}


Following a similar argument as in the example of logistic regression, it is
not hard to verify the squared loss function satisfies all the assumptions regarding the loss function. We refer the interested readers to Appendix \ref{Sec-table} for verification of assumptions and computation of related constants.

Actually, in this example, the dual objective function can be computed in
closed form. The distributionally robust linear regression problem is
equivalent to
\[
\inf_{\beta\in B}\inf_{\lambda\geq0} \left\{
\lambda\sqrt{\delta}+
\frac{1}{n} \sum_{i=1}^{n}%
\frac{\lambda(\beta^{T}X_{i}-Y_{i})^{2}}{\lambda-\sqrt{\delta}\beta^{T}%
A(X_{i})^{-1}\beta} \right\}
\]

Now we explain the setting of our numerical experiment in this example. The
dimension of data is $d = 16$, 

and we randomly generate three different training datasets of size $n\in\{64,256,1024\}$. The matrix appears in the cost function is chosen as $A(x) = \mathbb{I}_d$. We apply the iterative scheme in Section
\ref{sec-iterative-scheme} to solve the ordinary linear regression model (with
$\delta= 0$) and its distributionally robust counterpart ($\delta>0$). Again,
we adapt the same learning rate for both model and chosen parameter $\tau=
0.55$ in Assumption \ref{Assump-step-sizes}. The plot of optimality gaps
(Error) versus iterations for DRO-model and ordinary linear regression model
is given in Figure \ref{fig-linear}.

\begin{figure}[ptb]
\centering
\begin{subfigure}[]{5.3cm}
    \centering
	\includegraphics[width = 5.3cm]{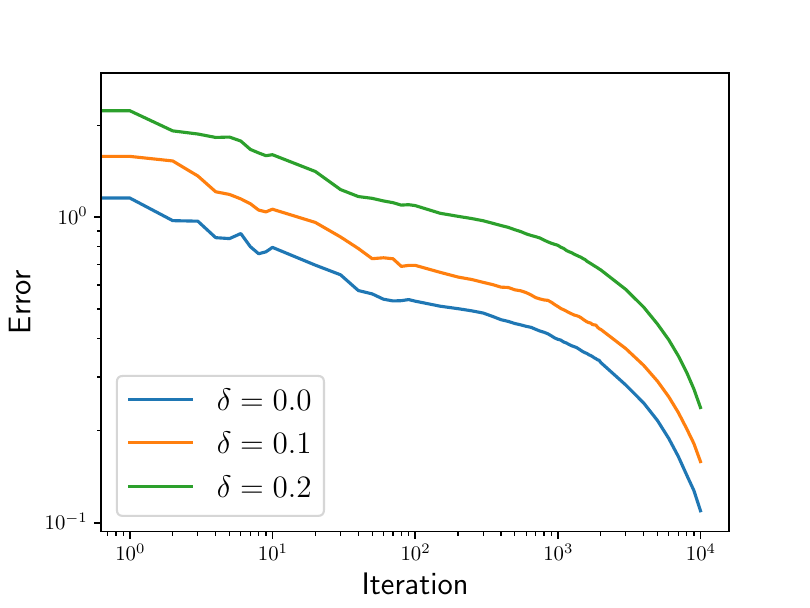}
	\caption{$n = 64$}
\end{subfigure}
\begin{subfigure}[]{5.3cm}
    \centering
    \includegraphics[width=5.3cm]{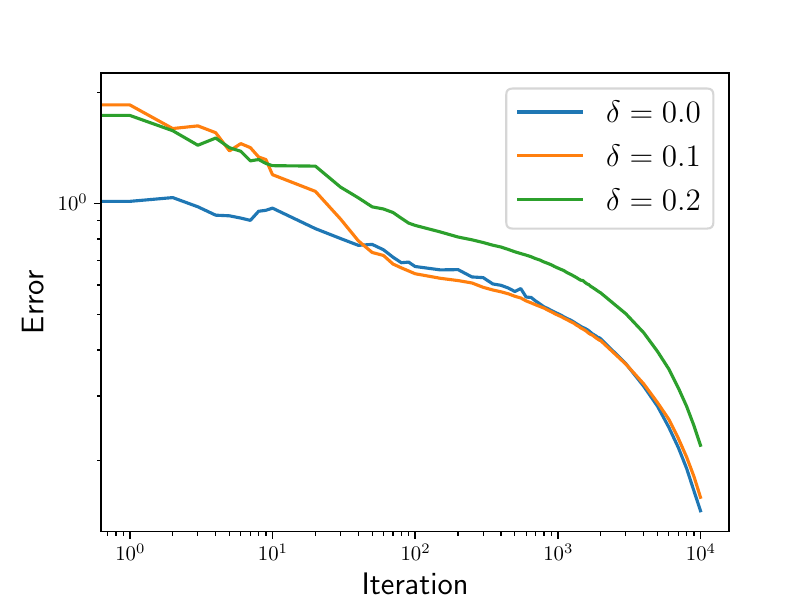}
    \caption{$n = 256$}
\end{subfigure}
\begin{subfigure}[]{5.3cm}
    \centering
	\includegraphics[width = 5.3cm]{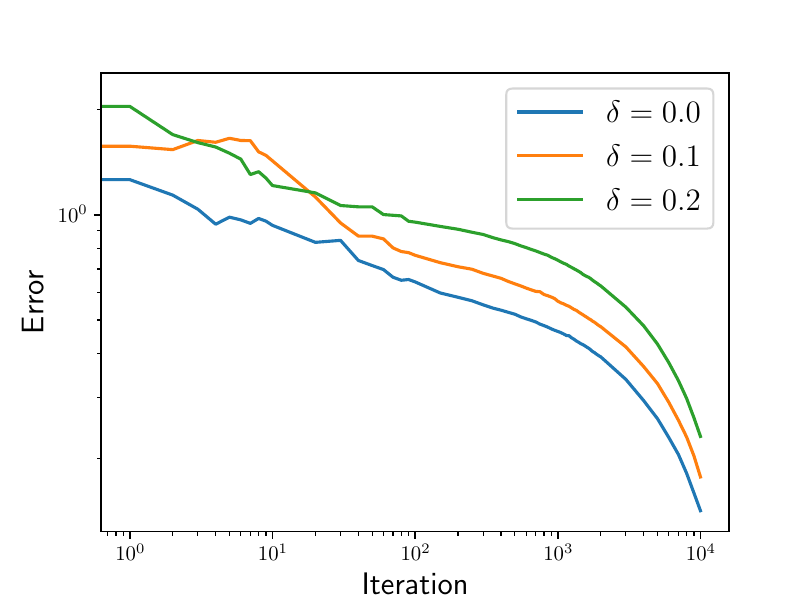}
	\caption{$n = 1024$}
\end{subfigure}
\caption{Convergence of loss function for linear regression}
\label{fig-linear}
\end{figure}

\subsubsection{Support vector machines.}

We consider the case of binary classification, where the data is given by
$\{(X_{i},Y_{i})\}_{i = 1}^{n}$, same as the data in the example of logistic
regression. The hinge loss function is $\ell(u;y) = \max\left( 0,1-yu\right) .
$ We are interested in solving the distributionally robust hinge loss
minimization problem,
\begin{align*}
\inf_{\beta\in B} \sup_{P: D_{c}(P_{n},P) \leq\delta} E_{P}\left[  \ell
(\beta^{T}X;Y)\right] ,
\end{align*}
where $P_0 = P_{n}(dx,dy):= \frac{1}{n}\sum_{i=1}^{n}\delta_{\{(X_{i},Y_{i}%
)\}}(dx,dy)$ is the empirical measure of data.

The algorithm to solve DRO with piecewise continuously differentiable function is discussed in Section \ref{Sec-Large-delta}. We present the procedure of verification of related assumptions and computation of constants in Appendix \ref{Sec-table}.



In the numerical experiment, we use the same data as the example of logistic
regression. Again, we set the learning rate to be same for DRO and non-DRO algorithms. Figure \ref{fig-svm-fig2} shows the path of optimality
gaps of loss functions during iterations. We use the value of loss function at
$10^{5}$ iteration as the approximate optimal loss given training samples,
and plot the optimality gap (Error) versus number of iterations in Figure
\ref{fig-svm-fig2}.

\begin{figure}[ptb]
\centering
\begin{subfigure}[]{5.3cm}
    \centering
	\includegraphics[width = 5.3cm]{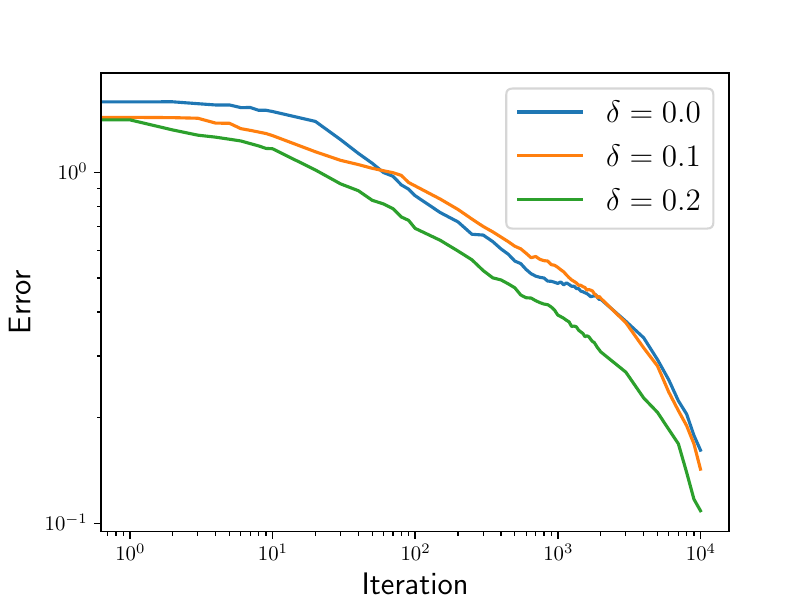}
	\caption{$n = 64$}
\end{subfigure}
\begin{subfigure}[]{5.3cm}
    \centering
    \includegraphics[width=5.3cm]{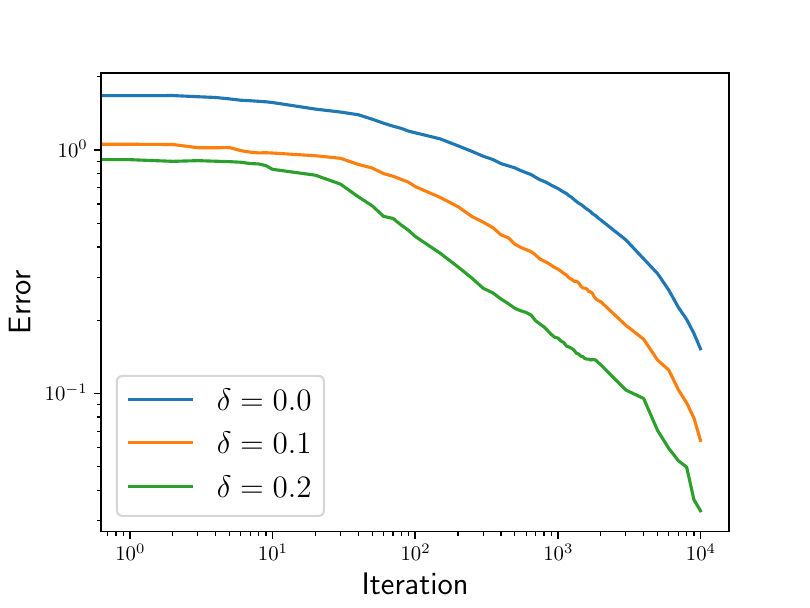}
    \caption{$n = 256$}
\end{subfigure}
\begin{subfigure}[]{5.3cm}
    \centering
	\includegraphics[width = 5.3cm]{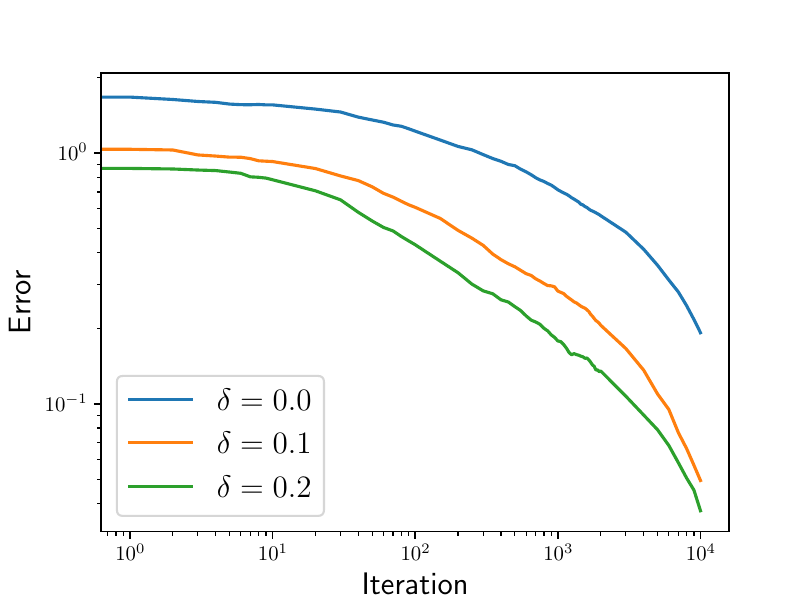}
	\caption{$n = 1024$}
\end{subfigure}
\caption{Convergence of loss function for support vector machines}
\label{fig-svm-fig2}
\end{figure}

\subsubsection{Comparison against conic programming reformulation}
\label{sec-sgd-socp}

Here we provide a comparative numerical example against a direct convex optimization approach \cite[Proposition 4]{hanasusanto2018conic}. For data-driven DRO with piecewise linear convex loss function of the form $\ell(u; y) = \max_{k = 1,\ldots,K}\{a_k\cdot (u-y) + b_k\},$ and matrix appears in the cost function chosen as $A(x) = \mathbb{I}_d$, the second order cone program (SOCP) reformulation in \cite[Proposition 4]{hanasusanto2018conic} obtained by letting  $P_0 = P_{n}(dx,dy):= \frac{1}{n}\sum_{i=1}^{n}\delta_{\{(X_{i},Y_{i})\}}(dx,dy)$ is given as below in \eqref{eq-socp-reform}.
\begin{align}
\label{eq-socp-reform}
\begin{array}{ll}
\inf~ & \lambda\delta + \frac{1}{n} \sum_{i = 1}^n s_i\\
\mathrm{s.t.}~ &  \beta \in B,\; \lambda\in\mathbb{R}_{+},\; s \in \mathbb{R}^{n}_+, \\
& s_i+a_k Y_i -b_k - a_k \beta^{T}X_i \geq 0
\qquad \forall k = 1,\ldots, K,\; \forall i = 1,\ldots, n,
\\
& \left\Vert
  \begin{array}{c}
  2c_k\beta\\
  s_i+a_k Y_i -b_k - a_k \beta^{T}X_i - \lambda
  \end{array}
  \right\Vert_2
  \leq s_i+a_k Y_i -b_k - a_k \beta^{T}X_i + \lambda\qquad
  \begin{array}{l}
  \forall k = 1,\ldots, K,\\
  \forall i = 1,\ldots, n. 
  \end{array}
\end{array}
\end{align}
If the loss function is not piecewise linear, such as square or logistic loss, one may solve the SOCP reformulation corresponding to a piecewise linear approximation. We invoke the linear regression model in Section \ref{sec-linear-regression} as an example for comparing the numerical performances of the direct convex optimization approach and the proposed SGD approach. We approximate the square loss function $\ell(u;y) = (u-y)^2$ with piecewise linear functions comprising $K = 9$ and $K=19$ linear functions in separate instances. The linear functions are chosen to be the tangent line of the loss function $\ell(u;y) = (u-y)^2$ with distinct integer supporting points $u$ satisfying $|u|\leq (K-1)/2$. We reformulated the resulting DRO with approximated loss as SOCP \eqref{eq-socp-reform}, which thereafter is solved using MOSEK\cite{mosek}. For the SGD approach, we terminate the algorithm if its optimality gap is smaller than the optimality gap of SOCP solution (the SOCP solution is suboptimal due to the linear approximation error). The data generating process of $\{(X_{i},Y_{i})\}_{i = 1}^{n}$ is same as Section \ref{sec-linear-regression}, with varying sample size $n$ to test the scalability of the algorithms. 

\begin{figure}[ptb]
\centering
\centering
\begin{subfigure}[]{8cm}
    \centering
	\includegraphics[width = 8cm]{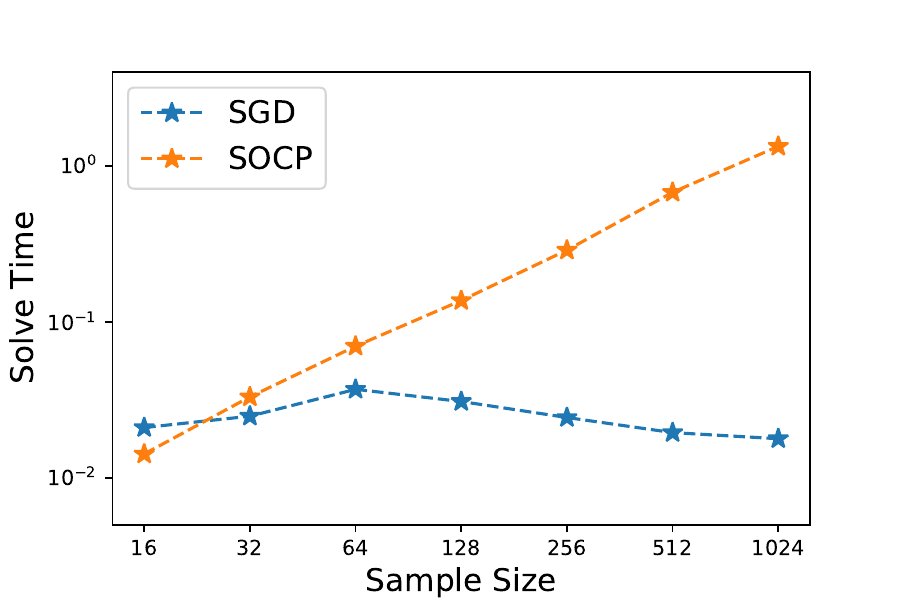}
	\caption{$K = 9$}
\end{subfigure}
\begin{subfigure}[]{8cm}
    \centering
    \includegraphics[width=8cm]{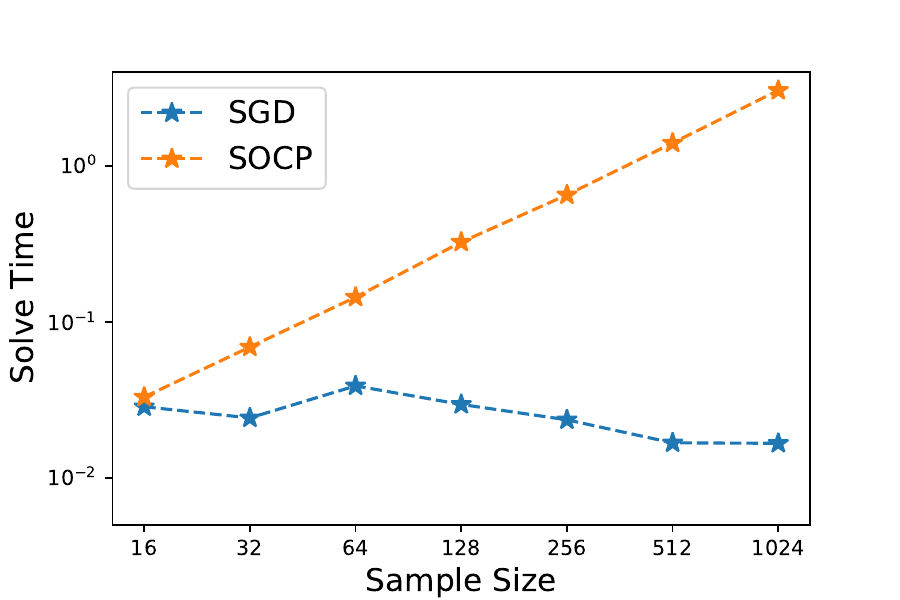}
    \caption{$K = 19$}
\end{subfigure}
\caption{Comparison of computational efforts for SGD and SOCP approaches for sample sizes $n\in\{16,32,64,128,256,512,1024\}.$
The number of linear functions are (a) $K = 9$ and (b) $K = 19$.}
\label{fig-socp-sgd}
\end{figure}

We compare the required time to solve SGD and SOCP in Figure \ref{fig-socp-sgd}. One can quickly remark that the SGD approach outperforms the SOCP approach for medium and large sample sizes. Though SOCP is a more efficient method for minimal sample size, its computational complexity rapidly deteriorates when $n$ is increasing, due to the $n\times K$ of cone constraints involved in the problem. In contrast, the computational time required by SGD is independent of the sample size.

\subsection{Portfolio optimization\label{subsec_portfolio}}
In this section, we demonstrate an example application of the proposed DRO framework in the context of mean-variance portfolio optimization. Suppose that $X$ is an  $\mathbb{R}^d-$valued random vector representing the relative monthly returns of $d$ securities.
Let us use $P^{\ast}$ to denote the probability distribution of $X.$ The classical Markowitz mean-variance model suggests that the portfolio choices lying on the efficient frontier can be determined by solving an optimization problem of the form,
\begin{equation}
\label{eq-mean-variance}\min_{\beta:\beta^{T}\mathbf{1} = 1} \text{Var}%
_{P^{\ast}}[\beta^{T}X] - \zeta\cdot E_{P^{\ast}}[\beta^{T}X],
\end{equation}
where $\beta$ is a $d-$dimensional weight vector and $\zeta\in[0,\infty)$ is
a suitable parameter choice determining the extent of risk-aversion. By adding an additional variable $\mu \in
\mathbb{R}$ representing the mean return of the portfolio, formulation
\eqref{eq-mean-variance} can be rewritten as the following stochastic optimization problem
with affine decision rules:
\begin{equation}
\label{eq-mean-variance-first-order}\inf_{\mu} \inf_{\beta:\beta^{T}\mathbf{1}
= 1}E_{P^{\ast}}[(\beta^{T}X-\mu)^{2} - \zeta\cdot\beta^{T}X].
\end{equation}
In practice, the probability distribution $P^{\ast}$ is not known  and it  is  common to work with historical returns data to arrive at a suitable portfolio choice. Suppose that we use  $P_{n}:=\frac{1}{n}\sum_{i = 1}^{n}\delta_{\{X_{i}\}}$ to denote the empirical distribution  corresponding to  $n$ historical return samples $\{X_1,\ldots,X_n\}.$
Due to the discrepancy between the ground-truth measure $P^{\ast}$ and the
reference measure $P_0 = P_{n}$, we consider the following distributionally robust variant of 
\eqref{eq-mean-variance-first-order}:
\begin{equation}
\label{eq-mean-variance-DRO}\inf_{\mu} \inf_{\beta:\beta^{T}\mathbf{1} =
1}\sup_{P:D_{c}(P_{0},P) \leq\delta}E_{P_{0}}[(\beta^{T}X-\mu)^{2} -
\zeta\cdot\beta^{T}X].
\end{equation}
As with most data-driven DRO formulations, the insertion of the inner supremum allows quantifying the impact of the 
model mismatch between the empirical distribution and plausible model variations which are a result of future market interactions. Additional information about such future variations can typically be inferred from current market data, in the form of,
for example, the implied volatility which can be elicited from the derivative prices. In such instances, a suitable choice of state-dependent Mahalanobis cost function $c(\cdot)$ in the proposed 
framework allows us to include this additional market information in the ambiguity set
$\{P:D_{c}(P_{0},P) \leq\delta\}$ which corresponds to the set of plausible model variations. To demonstrate  this idea in the  portfolio example, suppose that  we observe the implied volatility time  series $\{V_{i}: i  =1,\ldots,n\}$ in addition to the returns data  $\{X_i:i=1,\ldots,n\};$ here, $V_{i}$ is a positive scalar that represents the implied volatilities corresponding to the $i$-th observation $X_i.$ Let $\overline{V}  =  n^{-1}\sum_{i=1}^n V_i$ denote the average implied volatility. Corresponding to every point $X_i$ in the support of $P_n,$ we take the  state-dependent Mahalanobis cost to  be $c(X_i,x)  = (X_i - x)^T  A_i (X_i - x),$ where 
\begin{align}
 A_i  = \frac{\overline{V}}{V_i} \mathbb{I}_d, \quad i = 1,\ldots,n.
\label{Ax-Portfolio-Eg}
\end{align}
The rationale behind this choice is the hypothesis that a large implied volatility is suggestive of the anticipation of larger price uncertainty in future returns by the collective market. As a result,  the inverse proportionality relationship $A_i \propto V_i^{-1} \mathbb{I}_d$ in \eqref{Ax-Portfolio-Eg} is such that it is cheaper to perturb returns (or  transport mass) for observations with higher  implied volatility. The normalization by $\overline{V}$ is introduced to allow comparisons with the choice of standard Euclidean squared norm (corresponding to the choice $A(x)  := \mathbb{I}_d$) as the transportation cost function.

To test the effectiveness  of the DRO formulation \eqref{eq-mean-variance-DRO} with real data, we randomly pick 20 stocks from the constituents of S\&P 500 as the stock pool. The weights of the portfolio are adjusted on a monthly basis during the test period constituting the years 2000 - 2017. For every month in this test period, the portfolio  weights are obtained by training the formulation  \eqref{eq-mean-variance-DRO} with the respective stock pool data from the previous 10 years.  For example,
at the beginning of January 2000, the training data $\{X_1,\ldots,X_n\}$ for the model \eqref{eq-mean-variance-DRO} is the monthly historical returns of the selected 20 stocks observed during the period January 1990 - January 2000 (thus, $n=119$ and $d = 20$). The CBOE volatility index (VIX), which is a popular gauge of the stock market's forward looking
volatility implied by S\&P 500 index options, is used to inform the market implied volatility. The parameter $\delta$ is treated as a hyper-parameter and the out-of-sample efficient frontier is generated by considering different values of the parameter $\zeta.$ In Figure \ref{fig-frontier1}, we report the mean and the standard deviation of the portfolio returns (during the test period 2000-2017) obtained from 100 random stock pool choices. 

The data used for computing an optimal
portfolio is different from the data used for evaluating the portfolio, which is the reason we address the efficient frontiers in  Figure \ref{fig-frontier1} as ``out-of-sample''. These out-of-sample efficient frontiers reveal that the DRO formulation \eqref{eq-mean-variance-DRO} with state-dependent Mahalanobis cost choice (as in \eqref{Ax-Portfolio-Eg}) performs uniformly better than that obtained with  the Euclidean distance choice (corresponding to constant $A(x) = \mathbb{I}_d,$ addressed as constant model in Figure \ref{fig-frontier1}). We also observe that a larger value of
distributional uncertainty $\delta$ results in larger mean annualized return. Unlike the case of an efficient frontier generated and tested with samples from the same probability distribution, the negative slopes in the out-of-sample efficient frontiers in Figure \ref{fig-frontier1} suggest that the out-of-sample effects (such as non-stationarity in data) are significant.  

As a sanity check to verify our implementation, we also report the results of the same experiment with simulated data constituting i.i.d. training  and test samples (see Figure \ref{fig-frontier2}) for the choice $A(x) =\mathbb{I}_d.$  In this simulation experiment, the DRO model is observed to produce less efficient portfolios relative to non-robust formulations, which is not surprising given that the experiment has been designed with simulated data and there is little model error. The efficient frontiers of the DRO model, as expected for relatively small values of $\delta$, have positive slopes in out-of-sample simulated frontiers, and is consistent with the observations of the classical Markowitz theory. These experiment results can be viewed as underscoring the need for DRO model formulations such as the one we  study in this paper. In addition to historical returns data, these model formulations incorporate the flexibility to use additional information such as implied volatilities to elicit collective market expectations about future uncertainty.

\begin{figure}[ptb]
\centering
\begin{subfigure}{3in}
		\includegraphics[width = 3in]{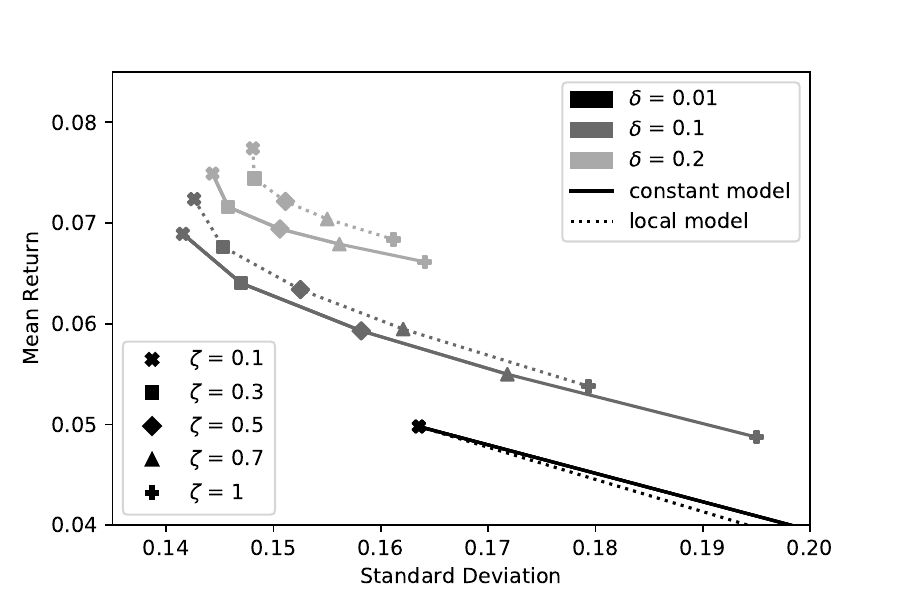}
		\caption{Real Data Experiment}
		\label{fig-frontier1}
	\end{subfigure}
\begin{subfigure}{3in}
		\includegraphics[width = 3in]{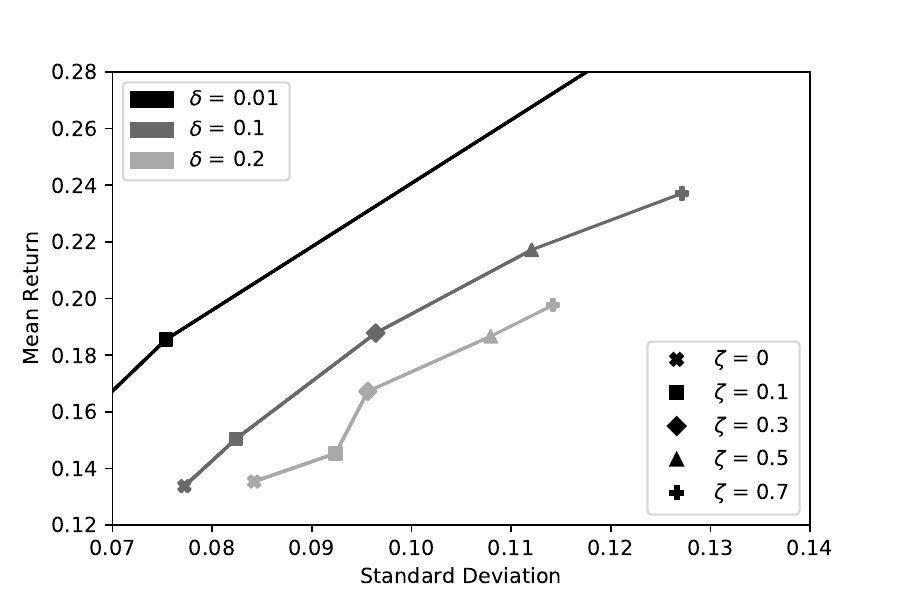}
		\caption{Simulated Data Experiment}
		\label{fig-frontier2}
	\end{subfigure}
\caption{Out of sample efficient frontier. The mean and the standard deviation
are annualized. We use solid lines to represent models with constant optimal
transport cost function, and use dashed lines to represent the models with
state-dependent Mahalanobis optimal transport cost function. The different choices of
$\delta$ are denoted by different colors. The values of $\zeta$ are
represented by different shapes of the markers.}%
\end{figure}

\section{Proofs of main results.}

\label{Sec-Proofs} We shall provide proofs of all the main results in Sections
\ref{Sec-Reform-Convexity} - \ref{Sec-SA-Algo} in this section. The proofs of
auxiliary results, which are technical in nature, are provided in the
subsequent technical appendix Section \ref{Sec-App-Proofs} for ease of reading.

\vspace{-0.1in}

\subsection{Proofs of the results on dual reformulation and convexity.}

\label{sec:proofs-dualreform-convex} In this section, we shall see the proofs
of of Theorems \ref{Thm-strong-duality} - \ref{thm:convex} and Lemma
\ref{Lem-conv-lrob}.

\proof{\textbf{Proof of Theorem \ref{Thm-strong-duality}}.} Since $c(\cdot)$
is lower semicontinuous and $\ell(\cdot)$ is upper semicontinuous, it follows
from the the strong duality result in Theorem 1 of
\cite{blanchet_quantifying_2016} that
\begin{align*}
\sup_{P: D_{c}(P_{0},P) \leq\delta} E_{P}[\ell(\beta^{T}X)]  & = \inf
_{\lambda\geq0} E_{P_{0}}\left[  \sup_{\Delta\in\mathbb{R}^{d}} \left\{
\ell(\beta^{T}(X + \Delta)) - \lambda\left( \Delta^{T}A(X)\Delta-\delta\right)
\right\} \right] \\ & =\inf_
{\lambda\geq0} E_{P_{0}}\left[ \sup_{c \in\mathbb{R}} \left\{  \ell(\beta^{T}X + c) - \lambda\left( \inf_{\Delta: \beta^{T} \Delta= c} \Delta^{T}A(X)\Delta- \delta\right)  \right\} \right] ,
\end{align*}
and that the infimum on the right hand side is attained for every $\beta\in
B.$ Since
\[
\inf\{ \Delta^{T} A(X) \Delta: \beta^{T}\Delta= c\} = c^{2}/(\beta^{T}
A(X)^{-1}\beta)
\]
for $\beta\neq\mathbf{0},$ changing variables  as in $c = \sqrt{\delta}%
\gamma\beta^{T} A(X)^{-1} \beta$ and from  $\lambda\sqrt{\delta}$ to $\lambda$
lets us conclude that
\begin{align}
\sup_{\Delta\in\mathbb{R}^{d}} \left\{  \ell(\beta^{T}(X + \Delta)) -
\lambda\left( \Delta^{T}A(X)\Delta-\delta\right) \right\}  = \sup_{\gamma
\in\mathbb{R}} F(\gamma,\beta,\lambda;X) =: \ell_{rob} \xspace(\beta
,\lambda;X),\label{Inter-SD}%
\end{align}
thus resulting in $\sup_{P: D_{c}(P_{0},P) \leq\delta} E_{P}[\ell(\beta^{T}X)]
= \inf_{\lambda\geq0} E_{P_{0}}[\ell_{rob} \xspace(\beta,\lambda;X)].$ This
completes the proof of Theorem \ref{Thm-strong-duality}. \hfill\Halmos
\endproof
\ 

The proof of Theorem \ref{thm:convex} follows immediately as a consequence of
Lemma \ref{Lem-conv-lrob} (stated in Section \ref{Sec-Oracle-Info-Extr}) and
Lemma \ref{Lem-Finite-Exp} below, whose proof is furnished in the technical
Appendix \ref{Sec-App-Proofs}.

\begin{lemma}
	Suppose that Assumptions \ref{Assump-c}, \ref{Assump-Obj-Conv} hold.
	Consider any $\varepsilon > 0,$ $x \in \mathbb{R}^d$ and $\beta \in B.$ If $%
	\lambda \geq (\kappa + \varepsilon)\sqrt{\delta}\beta^TA(x)^{-1}\beta,$ then
	there exist positive constants $C_1,C_2$ such that
\begin{itemize}
\item[a)] any $g \in \Gamma^\ast(\beta,\lambda;x)$ satisfies
$\sqrt{\delta}\vert g \vert \beta^TA(x)^{-1}\beta \leq 1 +
C_1 \varepsilon^{-1}(1 + \vert \beta^Tx \vert);$ and
\item[b)]
$\ell_{rob} \xspace(\beta,\lambda;x) \leq
\lambda\sqrt{\delta} + C_2( 1 + \varepsilon +
\varepsilon^{-1})(1 + \vert \beta^Tx\vert)^2.$
	\end{itemize}	
	\label{Lem-Finite-Exp}
\end{lemma}

We shall first see the proof of Lemma \ref{Lem-conv-lrob} before proceeding to
the proof of Theorem \ref{thm:convex}.
\proof{\textbf{Proof of Lemma \protect\ref{Lem-conv-lrob}}.} Take any
$\theta_{1} := (\beta_{1},\lambda_{1})$ and $\theta_{2} := (\beta_{2}%
,\lambda_{2})$ in $B \times\mathbb{R}_{+}.$ Given $\alpha\in[0,1],$ it follows
from \eqref{Inter-SD} that $\ell_{rob} \xspace(\alpha\theta_{1} +
(1-\alpha)\theta_{2};x)$ equals
\begin{align}
& \sup_{\Delta\in\mathbb{R}^{d}} \left\{  \ell\left( (\alpha\beta_{1} +
(1-\alpha)\beta_{2})^{T}(x + \Delta)\right)  - \left( \alpha\lambda_{1} +
(1-\alpha) \lambda_{2}\right) \left( \Delta^{T}A(x)\Delta- \delta\right)
\right\} \nonumber\\
& \quad= \left( \alpha\lambda_{1} + (1-\alpha) \lambda_{2}\right)
\delta\nonumber\\
& \quad\quad+ \sup_{\Delta\in\mathbb{R}^{d}} \left\{  \ell\left( \alpha
\beta_{1}^{T}(x + \Delta) + (1-\alpha)(\beta_{2}^{T}(x + \Delta))\right)  -
\left( \alpha\lambda_{1} + (1-\alpha) \lambda_{2}\right) \Delta^{T}%
A(x)\Delta\right\} .\label{Inter-Conv-lrob}%
\end{align}
Since $\ell(\cdot)$ is convex, we have $\ell(\alpha u_{1} + (1-\alpha)u_{2})
\leq\alpha\ell(u_{1}) + (1-\alpha)\ell(u_{2})$ for $u_{1},u_{2} \in
\mathbb{R}.$ Combining this with the fact that $\sup_{\Delta}(\alpha
f_{1}(\Delta)+ (1-\alpha)f_{2}(\Delta)) \leq\alpha\sup_{\Delta}f_{1}(\Delta) +
(1-\alpha)\sup_{\Delta}f_{2}(\Delta)$ for any two functions $f_{1},f_{2},$ we
have that the term involving supremum in \eqref{Inter-Conv-lrob} is bounded
from above by,
\begin{align*}
\alpha\sup_{\Delta\in\mathbb{R}^{d}} \left\{  \ell(\beta_{1}^{T}(x + \Delta))
- \lambda_{1}\Delta^{T}A(x)\Delta\right\}  + (1-\alpha) \sup_{\Delta
\in\mathbb{R}^{d}} \left\{  \ell(\beta_{2}^{T}(x + \Delta)) - \lambda
_{2}\Delta^{T}A(x)\Delta\right\} .
\end{align*}
This observation, in conjunction with \eqref{Inter-Conv-lrob}, establishes
that $\ell_{rob} \xspace(\alpha\theta_{1} + (1-\alpha)\theta_{2};x) \leq
\alpha\ell_{rob} \xspace(\theta_{1};x) + (1-\alpha)\ell_{rob} \xspace(\theta
_{2};x),$ thus verifying the desired convexity of $\ell_{rob} \xspace(\,\cdot
\, ;x).$ \hfill\Halmos
\endproof
\ 

\proof{\textbf{Proof of Theorem \ref{thm:convex}}.} Since $f_{\delta}%
(\beta,\lambda) := E_{P_{0}}[\ell_{rob} \xspace(\beta,\lambda;X)],$ the
convexity of $f_{\delta}(\cdot)$ follows as a consequence of Lemma
\ref{Lem-conv-lrob} and linearity of expectations. The fact that $f_{\delta
}(\cdot)$ is proper follows from the observation that $\ell_{rob}%
(\beta,\lambda;X)$ is almost surely finite for all $\lambda$ sufficiently
large (see Lemma \ref{Lem-Finite-Exp}b) and the assumption that $E_{P_{0}%
}\Vert X \Vert^{2} < \infty$ (see Assumption \ref{Assump-Obj-Conv}).
\hfill\Halmos \endproof

\vspace{-0.1in}

\subsection{Bounds for dual optimizer $\lambda_{\ast}(\beta)$ and a proof of
Proposition ~\ref{Prop-V-cont-OPT}.}

\label{Sec-Prep-results} It follows from Theorem \ref{Thm-strong-duality} that
$\arg\min_{\lambda\geq0} f_{\delta}(\beta,\lambda)$ is nonempty for any
$\beta\in B.$ Lemma \ref{Lem-threshold} - \ref{Lem-Lbars} below, whose proofs
are provided in Appendix \ref{Sec-App-Proofs}, are useful towards establishing
bounds for any $\lambda_{\ast}(\beta)$ in $\arg\min_{\lambda\geq0} f_{\delta
}(\beta,\lambda)$ (see Lemma \ref{Lemma-Lambda-Bound}). In turn, these bounds
are useful towards identifying the region $\mathbb{V}$
in the main results Proposition \ref{Prop-V-cont-OPT} and Theorem
\ref{result-smoothness}.

\begin{lemma}
Suppose that Assumptions \ref{Assump-c} - \ref{Assump-Obj-Conv} are
satisfied and $\beta \in B.$ Then for any
$\lambda_\ast(\beta) \in \arg%
\min_{\lambda \geq 0} f_\delta(\beta,\lambda),$ we have
$\Gamma^\ast(\beta,\lambda_\ast(\beta);x) \neq \varnothing,$ for $P_0-$almost every
$x \in \mathbb{R}%
^d. $ Moreover,
\begin{align*}
\frac{\partial_+ f_\delta}{\partial \lambda}(\beta,\lambda_\ast(\beta))
= \sqrt{\delta} \left( 1 - E_{P_0}\left[ \beta^TA(X)^{-1}\beta \min_{\gamma \in
\Gamma^\ast(\beta,\lambda_\ast(\beta); X)} \gamma^2\right]\right).
\end{align*}
\label{Lem-threshold}
\end{lemma}

\begin{lemma}
\label{Lemma-Gamma} Suppose that Assumptions \ref{Assump-c} -
\ref{Assump-Obj-Conv} are satisfied and
$\Gamma^\ast(\beta,\lambda,x)$ is not empty for a given
$\beta \in B,$ $x \in \mathbb{R}^d$ and $\lambda \geq 0.$ Then for
any $\gamma \in \Gamma^\ast(\beta,\lambda;x),$ we have,
$\gamma = {\ell^{\prime}(\beta^Tx +
\sqrt{\delta}\gamma\beta^TA(x)^{-1}\beta )}/{(2\lambda)},$
and  consequently,
\begin{align} \label{Lower-bound-g-i}
\vert \gamma \vert \geq
\frac{\vert\ell^{\prime}(\beta^T x)\vert}{2\lambda}.
\end{align}
\end{lemma}

\begin{lemma}
Suppose that Assumptions \ref{Assump-Obj-Conv} -
\ref{Assump-Compact} are satisfied. Then there exist positive
constants $\underline{L},\overline{L}$ such that
$\underline{L} \leq E_{P_0}[\ell^\prime(\beta^TX)^2] \leq
\overline{L}$ for every $\beta \in B.$
\label{Lem-Lbars}
\end{lemma}

\begin{lemma}
Suppose that Assumptions \ref{Assump-c} - \ref{Assump-Obj-2diff} are satisfied.
Then any minimizer $\lambda_{\ast}(\beta) \in \mathnormal{arg}\min_{\lambda
\geq 0}f_\delta(\beta,\lambda)$ satisfies 
$\lambda_{\mathrm{min}}(\beta)\leq \lambda_{\ast}(\beta)\leq \lambda_{%
\mathrm{max}}(\beta),$ 
where
\begin{align*}
\lambda_{\mathrm{min}}(\beta) &:= \frac{1}{2} \rho_{\max}^{-1/2}\Vert \beta
\Vert \sqrt{E_{P_0}\left[ \ell^{\prime}(\beta^T X)^2\right]} \text{ and } \\
\lambda_{\mathrm{max}}(\beta) &:= \rho_{\min}^{-1/2}\|\beta\|\sqrt{E_{P_0}%
\left[\ell^{\prime}(\beta^T X)^2\right]} + \frac{1}{2}\sqrt{\delta}%
M\rho_{\min}^{-1}\|\beta\|^2.
\end{align*}
\label{Lemma-Lambda-Bound}
\end{lemma}
\vspace{-0.2in} 

\proof{\textbf{Proof of Lemma \protect\ref{Lemma-Lambda-Bound}.}}
\textbf{Lower bound.}
Combining the observations in Lemma \ref{Lem-threshold} - \ref{Lemma-Gamma}
and the first order optimality condition that $\partial_{+} f_{\delta}%
(\beta,\lambda_{\ast}(\beta))/\partial\lambda\geq0,$ we obtain,
\begin{align*}
0 \leq\frac{\partial_{+}f_{\delta}}{\partial\lambda}(\beta,\lambda_{\ast
}(\beta)) \leq\sqrt{\delta}\left( 1 - E_{P_{0}}\left[  \beta^{T}A(X)^{-1}%
\beta\frac{\ell^{\prime}(\beta^{T}X)^{2}}{4\lambda_{\ast}(\beta)^{2}}\right]
\right) .
\end{align*}
Because of Assumption \ref{Assump-c}b, the above inequality results in,
\begin{align*}
\lambda_{\ast}(\beta) \geq\frac{1}{2}E_{P_{0}}^{1/2}\left[  \ell^{\prime
}(\beta^{T}X)^{2}\beta^{T}A(X)^{-1}\beta\right]  \geq\frac{1}{2}\rho_{\max
}^{-1/2} \Vert\beta\Vert\sqrt{E_{P_{0}}\left[  \ell^{\prime}(\beta^{T}%
X)^{2}\right] } =:\lambda_{\min}(\beta).
\end{align*}

\textbf{Upper bound.} As $\ell^{\prime\prime}(\cdot) \leq M$ due to Assumption
\ref{Assump-Obj-2diff}, we have that
$\ell_{rob} \xspace(\beta,\lambda;X) - \ell(\beta^{T}X)$ is bounded from above
by,
\begin{align*}
& \sup_{\gamma\in\mathbb{R}}\left\{  \ell\left( \beta^{T}X + \gamma
\sqrt{\delta} \beta^{T} A(X)^{-1}\beta\right) - \ell\left( \beta^{T}X \right)
- \lambda\sqrt{\delta} \beta^{T} A(X)^{-1}\beta\gamma^{2}\right\} \\
& \quad\quad\leq\sup_{\gamma\in\mathbb{R}}\left\{  \ell^{\prime} \left(
\beta^{T}X\right) \sqrt{\delta}\beta^{T} A(X)^{-1}\beta\gamma+ \frac{1}{2}M
\left( \gamma\sqrt{\delta} \beta^{T} A(X)^{-1}\beta\right) ^{2}- \lambda
\sqrt{\delta} \beta^{T} A(X)^{-1}\beta\gamma^{2}\right\} \\
& \quad\quad= \frac{\sqrt{\delta} \beta^{T} A(X)^{-1}\beta[\ell^{\prime}%
(\beta^{T}X)]^{2}}{(4\lambda- 2M\sqrt{\delta} \beta^{T} A(X)^{-1}\beta)^{+}}.
\end{align*}
Next, since $\lambda_{\ast}(\beta)\sqrt{\delta} + E_{P_{0}}\left[ \ell
(\beta^{T}X)\right]  \leq f_{\delta}(\beta,\lambda_{\ast}(\beta)) =
\inf_{\lambda\geq0}E_{P_{0}}[\ell_{rob} \xspace(\beta,\lambda;X)],$ we use the
above result and the bounds in Assumption \ref{Assump-c}b to write,
\begin{align*}
\lambda_{\ast}(\beta) & \leq\inf_{\lambda\geq0}\left\{  \lambda+ \delta
^{-1/2}E_{P_{0}}\left[ \ell_{rob} \xspace(\beta,\lambda;X) - \ell(\beta
^{T}X)\right] \right\} \\
& \leq\inf_{\lambda> \frac{1}{2}\sqrt{\delta}M\rho_{\min}^{-1}\|\beta
\|_{2}^{2}}\left\{  \lambda+E_{P_{0}}\left[ \frac{ \beta^{T} A(X)^{-1}%
\beta[\ell^{\prime}(\beta^{T}X)]^{2}}{4\lambda- 2M\sqrt{\delta} \beta^{T}
A(X)^{-1}\beta}\right] \right\} \\
& \leq\inf_{\lambda> \frac{1}{2}\sqrt{\delta}M\rho_{\min}^{-1}\|\beta
\|_{2}^{2}}\left\{  \lambda+\frac{ \rho_{\min}^{-1}\|\beta\|^{2}}{4\lambda-
2M\sqrt{\delta} \rho_{\min}^{-1}\|\beta\|^{2}} E_{P_{0}}\left[ \ell^{\prime
}(\beta^{T}X)^{2}\right]  \right\}
\end{align*}
The expression in the right hand side is a one dimensional convex optimization
problem which can be solved in closed form to obtain,
\begin{align*}
\lambda_{\ast}(\beta) \leq\frac{1}{2}\sqrt{\delta}M\rho_{\min}^{-1}%
\|\beta\|^{2} + \rho_{\min}^{-1/2}\|\beta\|\sqrt{E_{P_{0}}\left[  \ell
^{\prime}(\beta^{T} X)^{2}\right] } =: \lambda_{\text{max}}(\beta).
\end{align*}
This completes the proof of Lemma \ref{Lemma-Lambda-Bound}. \hfill\Halmos
\endproof
\ 

\proof{\textbf{Proof of Proposition \protect\ref{Prop-V-cont-OPT}.}} For a
given $\beta\in B,$ it follows from Lemma \ref{Lemma-Lambda-Bound} that any
optimal $\lambda_{\ast}(\beta)$ lies in the interval $[\lambda_{\min}(\beta),
\lambda_{\max}(\beta)].$ Recalling the definitions of $R_{\beta}$ from
Assumption \ref{Assump-Compact} and the characterization of $\overline{L}$ and
$\underline{L}$ in Lemma \ref{Lem-Lbars}, we have from Lemma
\ref{Lemma-Lambda-Bound} above that $\lambda_{\min}(\beta) \geq K_{1}
\Vert\beta\Vert$ and $\lambda_{\max}(\beta) \leq K_{2} \Vert\beta\Vert,$
where
\begin{align}
K_{1} := \frac{1}{2} \sqrt{\underline{L}\rho_{\max}^{-1}} \quad\text{ and }
\quad K_{2} := \frac{1}{2}\sqrt{\delta}MR_{\beta}\rho_{\min}^{-1} + \sqrt{
\rho_{\min}^{-1}\overline{L}}.\label{defn-K1-K2}%
\end{align}
Thus we obtain that $(\beta,\lambda_{\ast}(\beta))\in\mathbb{V}$ for all
$\beta\in B$. \hfill\Halmos
\endproof

\subsection{Verifying smoothness and strong convexity of the dual DRO
objective.}

\label{subsec:proofs-thms-smooth-str-convex} In this section, we provide
proofs of Theorems \ref{result-smoothness} - \ref{Thm-Str-Convexity}. We
accomplish this primarily by identifying the Hessian matrix of the dual DRO
objective $f_{\delta}(\beta,\lambda) = E_{P_{0}}\left[ \ell_{rob}%
(\beta,\lambda;X)\right] .$

Recall the definition of the functions $\ell_{rob}(\cdot)$ and $F(\cdot)$ in
Theorem \ref{Thm-strong-duality}. Let $S_{X}$ be the support of the
distribution $P_{0}.$ For a given $(\beta,\lambda)$ and $x \in S_{X},$ we use
the set $\Gamma^{\ast}(\beta,\lambda;x)$ to denote the respective set of
maximizers $\arg\max_{\gamma}F(\gamma,\beta,\lambda;x)$ (see
(\ref{Eqn-Optimal-Set})). A characterization of the gradient of the function
$\ell_{rob}(\beta,\lambda;x)$ is derived in Proposition
\ref{prop:firstorderinfo} with the help of envelope theorem. Likewise, if the
loss $\ell(\cdot)$ is twice differentiable, implicit function theorem allows
us to characterize the Hessian of $\ell_{rob}(\beta,\lambda;x).$ To accomplish
this, define
\begin{align*}
\mathcal{U} := \left\{  (\beta,\lambda,x) \in B \times\mathbb{R}_{+} \times
S_{X}: \Gamma^{\ast}(\beta,\lambda;x) \neq\varnothing, \ \varphi(\gamma,
\beta,\lambda;x) > 0 \text{ for some } \gamma\in\Gamma^{\ast}(\beta,\lambda;x)
\right\} ,
\end{align*}
where
\begin{align*}
\varphi(\gamma,\beta,\lambda;x) := 2\lambda- \sqrt{\delta}\beta^{T}%
A(x)^{-1}\beta\ell^{\prime\prime}\left( \beta^{T}x + \sqrt{\delta}\gamma
\beta^{T}A(x)^{-1}\beta\right) .
\end{align*}
Further consider the set valued map $x \mapsto\mathcal{U}(x)$ to be the
projection,
\begin{align*}
\mathcal{U}(x) := \left\{ (\beta,\lambda): (\beta,\lambda,x) \in\mathcal{U}
\right\} .
\end{align*}
Then, as a consequence of implicit function theorem, the function $\ell
_{rob}(\beta,\lambda;x)$ is twice differentiable for every $(\beta,\lambda)$
in the interior of $\mathcal{U}(x).$ Indeed, this follows from the observation
that $\partial^{2} F/\partial\gamma^{2}(\cdot) = -2\sqrt{\delta} \beta
^{T}A(x)^{-1}\beta\varphi(\cdot)$ is negative when $(\beta,\lambda,x)
\in\mathcal{U}.$ Next, consider any measurable selection $g:\mathcal{U}
\rightarrow\mathbb{R}$ such that
\begin{align}
g(\beta,\lambda;x) \in\Gamma^{\ast}(\beta,\lambda;x) \quad\text{ and }
\quad\varphi(g(\beta,\lambda;x),\beta,\lambda,x) >
0,\label{meas-sel-twice-diff}%
\end{align}
for $P_{0}-$almost every $x$ and almost every $(\beta,\lambda) \in
\mathcal{U}(x).$ The existence of such a measurable selection follows from
Jankov-Von Neumann theorem (see, for example, \cite[Proposition 7.50]%
{bertsekas1978stochastic}). To proceed further, define,
\begin{align}
T_{g}(x) := x + \sqrt{\delta} & g(\beta,\lambda;x)A(x)^{-1}\beta, \quad\bar
{T}_{g}(x) := x + 2\sqrt{\delta}g(\beta,\lambda;x) A(x)^{-1}\beta, \text{ and
}\nonumber\\
& \varphi_{g}(\beta,\lambda;x) := \varphi\left(  g(\beta,\lambda
;x),\beta,\lambda;x\right) ,\label{Str-Conv-Pf-Defns-1}%
\end{align}
for any $(\beta,\lambda,x) \in\mathcal{U},$ where the dependence on
$(\beta,\lambda)$ is hidden in the notation of the transport maps
$T_{g}(x),\bar{T}_{g}(x)$ and has to be understood implicitly. Likewise, once
the choice of measurable selection $g(\cdot)$ is fixed, we often suppress the
arguments $(\beta,\lambda;x)$ while writing the functions such as
$g(\beta,\lambda;x)$ and $\varphi_{g}(\beta, \lambda;x)$ in order to reduce
clutter in the resulting expressions; for example, we simply write
$\varphi_{g}$ and $g$, respectively, for $\varphi_{g}(\beta,\lambda;x)$ and
$g(\beta,\lambda;x).$

\begin{proposition}
Suppose that Assumptions \ref{Assump-c} - \ref{Assump-Obj-2diff} are
satisfied, $\mathcal{U}$ is not empty, and
$g:\mathcal{U} \rightarrow \mathbb{R}$ is a measurable selection
satisfying (\ref{meas-sel-twice-diff}). Then for almost every
$x \in S_X,$ $(\beta,\lambda) \in \textnormal{int}(\mathcal{U}(x)),$
we have,
\begin{align*}
\frac{\partial^2 \ell_{rob} \xspace}{\partial \beta^2}&(\beta,\lambda;x) =
2\sqrt{\delta}\lambda g^2A(x)^{-1} + \frac{2\lambda\ell^{\prime\prime}
\left(\beta^TT_g(x) \right) }{\varphi_g} \bar{T}_g(x)\bar{T}_g(x)^T, \quad \frac{%
\partial^2\ell_{rob} \xspace}{\partial \lambda^2}(\beta,\lambda;x) = \frac{4%
\sqrt{\delta} g^2 \beta^TA(x)^{-1}\beta}{\varphi_g}, \\
&\quad\quad\quad \frac{\partial^2\ell_{rob} \xspace}{\partial \lambda
\partial \beta}(\beta,\lambda;x) = -2\sqrt{\delta} g^2 \left( A(x)^{-1}\beta
+ \frac{\beta^T A(x)^{-1}\beta \ell^{\prime\prime}(\beta^TT_g(x))}{g\varphi_g} \bar{T}_g(x) \right),
\end{align*}
where $T_g(\cdot), \bar{T}_g(\cdot),\varphi_g$ are defined as in
\eqref{Str-Conv-Pf-Defns-1}. Moreover, we have
\begin{align}
\nabla_\theta^2\ell_{rob} \xspace(\theta;x) -\Lambda(\theta;x)B(x)
\succeq 0,
\label{PDness-2}
\end{align}
where
\begin{align}
\Lambda(\beta,\lambda;x) &:= \frac{4 \left(\beta^TT_g(x)\right)^2 \ell^{\prime
\prime}(\beta^T T_g(x))}{1 + \bar{T}_g(x)^T A(x) \bar{T}_g(x)
\ell^{\prime\prime}(\beta^T T_g(x))/(\sqrt{\delta}g^2\varphi)}
\frac{1}{2\lambda \varphi_g + 4\beta^TA(x)^{-1}\beta}
\label{PD-coeff}
\end{align}
and
\begin{align*}
B(x) &=
\begin{bmatrix}
A(x)^{-1}+\frac{\ell^{\prime\prime}(\beta^TT_g(x))}{\sqrt{\delta}
g^2\varphi}\bar{T}_g(x)\bar{T}_g(x)^T &  & & \mathbf{0} \\
&  & & \\
\mathbf{0}^T &  & & 1
\end{bmatrix}.
\end{align*}
\label{prop:hessian-lrob}
\end{proposition}

The proofs of Proposition \ref{prop:hessian-lrob} and Lemma
\ref{Lem-Str-conv-Thr} below are provided in the technical Appendix
\ref{Sec-App-Proofs}. For every $\beta\in B,$ recall that we have defined
$\lambda_{thr}^{\prime}(\beta)$ to be the $P_{0}-$essential supremum of
$\sqrt{\delta}M\beta^{T}A(x)^{-1}\beta/2.$

\begin{lemma}
Suppose Assumptions \ref{Assump-c} - \ref{Assump-Obj-2diff} are
satisfied. Then the map $\gamma \mapsto F(\gamma,\beta,\lambda;x)$
is strongly concave for every
$\beta \in B, \lambda > \lambda_{thr}^\prime \xspace(\beta)$ and $%
P_0-$almost every $x.$ Consequently, $\Gamma^\ast(\beta,\lambda;x)$
is singleton for every $\beta \in B,$
$\lambda > \lambda_{thr}^\prime(\beta)$ and
\[\left\{(\beta,\lambda): \beta \in B, \lambda >
\lambda_{thr}^\prime(\beta) \right\} \ \subseteq \ \mathcal{U}(x)\] for
$P_0-$almost every $x.$
\label{Lem-Str-conv-Thr}
\end{lemma}
The proof of Lemma \ref{Lem-Str-conv-Thr} is available in the technical Appendix \ref{Sec-App-Proofs}.


Proposition \ref{prop:secondorder-info-smoothcase} below allows us to
characterize the Hessian matrix of the dual DRO objective $f_{\delta}(\cdot).$
To state Proposition \ref{prop:secondorder-info-smoothcase}, define
\begin{align*}
\delta_{0} := \rho_{\min}^{2}\underline{L}R_{\beta}^{-2}M^{-2}\rho_{\max}%
^{-1}, \quad\text{ and } \quad\varphi_{\min} := \sqrt{\underline{L}}\rho_{\max
}^{-1/2} - \sqrt{\delta}R_{\beta}M\rho_{\min}^{-1},
\end{align*}
where the constants $\rho_{\min},\rho_{\max}$ are as in Assumption
\ref{Assump-c}b, $\underline{L},\overline{L}$ in Lemma \ref{Lem-Lbars},
$R_{\beta}$ in Assumption \ref{Assump-Compact} and $M$ in Assumption
\ref{Assump-Obj-2diff}. Recall the definition of the constants $K_{1},K_{2}$
in \eqref{defn-K1-K2} and that of the previously defined sets,
\begin{align*}
\mathbb{W} := \left\{ (\beta,\lambda) \in B \times\mathbb{R}_{+}: K_{1}%
\Vert\beta\Vert\leq\lambda
\leq K_{2}R_{\beta}
\right\}  \quad\text{ and } \quad\mathbb{V} :=
\left\{ (\beta,\lambda) \in B \times\mathbb{R}_{+}: K_{1}\Vert\beta\Vert
\leq\lambda\leq K_{2}\Vert\beta\Vert\right\} ,
\end{align*}
which contain the partial minimizers $\{(\beta,\lambda_{\ast}(\beta)):
\beta\in B\}$ when Assumptions \ref{Assump-c} - \ref{Assump-Compact} are
satisfied (see Proposition \ref{Prop-V-cont-OPT}). The proof of Proposition
\ref{prop:secondorder-info-smoothcase} is provided in the technical Appendix
\ref{Sec-App-Proofs}.

\begin{proposition}
Suppose Assumptions \ref{Assump-c} - \ref{Assump-Compact} are
satisfied and $\delta < \delta_0.$ Then
\begin{itemize}
\item[a)]
$\mathbb{V} \subseteq \mathbb{W} \subset \{(\beta,\lambda):
\beta \in B, \lambda > \lambda_{thr}^\prime(\beta)\} \subseteq
\mathcal{U}(x)$ for $P_0-$almost every $x;$
\item[b)] any map $g:\mathcal{U} \rightarrow \mathbb{R}$ satisfying
(\ref{meas-sel-twice-diff}) is uniquely specified for almost every
$(\beta,\lambda,x)$ in the subset
$\mathbb{W} \times S_X \subseteq \mathcal{U},$ and it satisfies the following relationships: 
for $P_0-$almost every $x$, we have $\varphi_g(\beta,\lambda;x) > \varphi_{\min}\Vert \beta \Vert$ if $(\beta,\lambda)\in \mathbb{W}$, and 
\begin{align}
\vert g(\beta,\lambda;x)\vert
\geq 
\frac{\vert\ell^\prime(\beta^Tx)\vert}{2K_2\Vert \beta \Vert}  
\mbox{ if }(\beta,\lambda)\in \mathbb{V},
\qquad 
\vert g(\beta,\lambda;x)\vert \leq
\frac{\vert\ell^\prime(\beta^Tx)\vert}{\varphi_{\min}\Vert \beta \Vert}
\mbox{ if }(\beta,\lambda)\in \mathbb{W}.
\label{gamma-UB-LB}
\end{align}
\item[c)] with $X \sim P_0,$ the collection
$\{g^2(\beta,\lambda;X),(T_g(X))^2, (\bar{T}_g(X))^2, \ell(\beta^TT_g(X)),
\ell^\prime(\beta^TT_g(X))^2: (\beta,\lambda) \in \mathbb{V}\}$ is
$L_2-$bounded; and
\item[d)] the Hessian matrix
$\nabla^2_\theta f_{\delta}(\theta) = E_{P_0}\left[
\nabla_\theta^2\ell_{rob}(\theta;X)\right]$ for every
$\theta \in \mathbb{V},$ where the Hessian
$\nabla_\theta^2\ell_{rob}(\theta;x)$ can be taken to be specified
in terms of the second order partial derivative expressions in
Proposition \ref{prop:hessian-lrob}.
\end{itemize}
\label{prop:secondorder-info-smoothcase}
\end{proposition}

The proofs of Theorem \ref{result-smoothness} - \ref{Thm-Str-Convexity}
provided next are reliant on the observations made in proposition
\ref{prop:secondorder-info-smoothcase} above.

\proof{\textbf{Proof of Theorem \protect\ref{result-smoothness}.}} a) It
follows from Part c) of Proposition \ref{prop:secondorder-info-smoothcase} and
the expressions of partial derivatives in Proposition \ref{prop:hessian-lrob}
that the norms of the respective entries (Frobenius norm $\Vert\cdot\Vert_{F}$
in case of matrix, or $\ell_{2}-$norm in case of vector), $\Vert\partial
^{2}\ell_{rob}/\partial\beta^{2}(\beta,\lambda;X)\Vert_{F},$ $\Vert
\partial^{2}\ell_{rob}/\partial\beta\partial\lambda(\beta,\lambda;X)\Vert,$
$\partial^{2}\ell_{rob}/\partial\lambda^{2}(\beta,\lambda;X),$ are all bounded
in $L_{2}-$norm over the set $(\beta,\lambda) \in\mathbb{V}.$ Consequently, we
have from Part d) of Proposition \ref{prop:secondorder-info-smoothcase} that
$\partial^{2}f_{\delta}/\partial\beta^{2},$ $\partial^{2}f_{\delta}%
/\partial\beta\partial\lambda,$ $\partial^{2}f_{\delta}/\partial\lambda^{2}$
are all bounded over $(\beta,\lambda) \in\mathbb{V}.$ As a result, the
Frobenius norm of the Hessian matrix $\nabla_{\theta}^{2}f_{\delta}(\theta)$
is bounded over $\theta= (\beta,\lambda) \in\mathbb{V}$ and hence the function
$f_{\delta}(\cdot)$ is smooth over the interior of $\mathbb{V}.$

b) To argue that $\partial^{2}f_{\delta}/\partial\beta^{2}$ is positive
definite, we proceed as follows: First observe that $A(x)^{-1} \succeq
\rho_{\max}^{-1}\mathbb{I}_{d}$ for $P_{0}-$almost every $x$ (that is,
$A(x)^{-1} - \rho_{\max}^{-1}\mathbb{I}_{d}$ is positive semidefinite). Next,
recall from (\ref{Lower-bound-g-i}) in Lemma \ref{Lemma-Gamma} and Lemma
\ref{Lem-Lbars} that $\vert g(\beta,\lambda;x) \vert\geq\vert\ell^{\prime
}(\beta^{T}x) \vert/(2\lambda)$ and $\underline{L} > 0.$ Then it follows from
Part d) of Proposition \ref{prop:secondorder-info-smoothcase} and the
expression of $\partial^{2}\ell_{rob} \xspace/\partial\beta^{2}$ from
proposition \ref{prop:hessian-lrob} that for any $(\beta,\lambda)
\in\mathbb{V},$
\begin{align*}
\frac{\partial^{2} f_{\delta}}{\partial\beta^{2}}(\beta,\lambda) = E_{P_{0}%
}\left[  \frac{ \partial^{2} \ell_{rob} \xspace}{\partial\beta^{2}}%
(\beta,\lambda;X)\right]  \succeq\sqrt{\delta}\frac{E_{P_{0}}[\ell^{\prime
}(\beta^{T}X)^{2}]}{2\lambda} \rho_{\max}^{-1}\mathbb{I}_{d} \succeq
\sqrt{\delta}\frac{\kappa_{0}}{\lambda}\mathbb{I}_{d},
\end{align*}
where $\kappa_{0} := 2^{-1}\underline{L}\rho_{\max}^{-1} > 0,$ thus proving
Theorem \ref{result-smoothness}. \hfill\Halmos \endproof
\ 

In order the proceed with the proof of \ref{Thm-Str-Convexity}, define
\[
\delta_{1} := \min\{\delta_{0}/4,  c_{1}^{2}c_{2}^{2}p^{2}\rho_{\min}^{2}%
\rho_{\max}^{-1}\underline{L}\overline{L}^{-2}/256\}.
\]

\proof{\textbf{Proof of Theorem \protect\ref{Thm-Str-Convexity}.}} Using the
bounds of $\vert g(\cdot)\vert$ and $\varphi_{g}(\cdot)$ from Proposition
\ref{prop:secondorder-info-smoothcase}b along with other immediate bounds such
as $\varphi_{g} \leq2\lambda, \lambda\in[K_{1}\Vert\beta\Vert, K_{2}\Vert
\beta\Vert],$ and $\beta^{T} A(x)^{-1} \beta\leq\rho_{\min}^{-1}\Vert
\beta\Vert^{2},$ the expression for $\Lambda(\beta,\lambda;x)$ from
(\ref{PD-coeff}) simplifies to,
\begin{align}
\Lambda & (\beta,\lambda;x) = \frac{4\sqrt{\delta}(g\beta^{T}T_{g}(x))^{2}%
}{{2\lambda\sqrt{\delta}g^{2}}/ {\ell^{\prime\prime}(\beta^{T}T_{g}(x))} +
\bar{T}_{g}(x)^{T}A(x)\bar{T}_{g}(x)} \cdot\frac{1}{2\lambda+ 4\beta
^{T}A(x)^{-1}\beta/\varphi_{g}}\label{Inter-Lambda-Simp-eq}\\
& \geq\frac{4\sqrt{\delta}(\beta^{T}T_{g}(x) \ell^{\prime}(\beta^{T}
x)/(2K_{2}\Vert\beta\Vert))^{2}}{{2K_{2}\sqrt{\delta}\ell^{\prime}(\beta
^{T}x)^{2}/ (\varphi_{\min}^{2}\Vert\beta\Vert}{\ell^{\prime\prime}(\beta
^{T}T_{g}(x)))} + \overline{T}_{g}(x)^{T}A(x)\overline{T}_{g}(x)}\cdot\frac
{1}{2K_{2}\Vert\beta\Vert+ 4\rho_{\min}^{-1}\Vert\beta\Vert^{2} /
(\varphi_{\min}\Vert\beta\Vert)}\nonumber\\
& \geq\sqrt{\delta} C_{0}\frac{\Vert\beta\Vert^{-2}\left( \beta^{T} T_{g}(x)
\ell^{\prime}(\beta^{T} x)\right) ^{2}} {{2 K_{2}\sqrt{\delta}}{\varphi_{\min
}^{-2}} { \ell^{\prime}(\beta^{T}x)^{2}}/{\ell^{\prime\prime}(\beta^{T}%
T_{g}(x))} + \overline{T}_{g}(x)^{T}A(x)\overline{T}_{g}(x) \Vert\beta\Vert
},\label{lambda-simplified}%
\end{align}
where $C_{0} := (2K_{2} + 4 \varphi_{\min}^{-1}\rho_{\min}^{-1})^{-1}.$ Next,
since $\beta^{T} T_{g}(x) = \beta^{T}x + \sqrt{\delta}g \beta^{T}%
A(x)^{-1}\beta,$ we obtain from the bounds in (\ref{gamma-UB-LB}) that
\begin{align}
\vert\beta^{T}T_{g}(x) \ell^{\prime}(\beta^{T}x) \vert & \geq\vert\beta^{T}x
\ell^{\prime}(\beta^{T}x) \vert- \sqrt{\delta} \vert g \ell^{\prime}(\beta
^{T}x)\vert\beta^{T}A(x)^{-1}\beta\nonumber\\
& \geq\vert\beta^{T}x \ell^{\prime}(\beta^{T}x) \vert- \sqrt{\delta}
\frac{\ell^{\prime}(\beta^{T}x)^{2}}{\varphi_{\min}\Vert\beta\Vert} \Vert
\beta\Vert^{2}\rho_{\min}^{-1} \geq\left(  c_1c_2 - \frac{4\sqrt{\delta}
\overline{L}}{p \varphi_{\min} \rho_{\min}}\right)  \Vert\beta\Vert
,\label{Lambda-num}%
\end{align}
whenever $X \in A_{1} \cap A_{2};$ here, the sets $A_{1}$ and $A_{2}$ are
defined as follows:
\[
A_{1}:=\{x: \vert\beta^{T}x \ell^{\prime}(\beta^{T}x) \vert> c_{1}c_{2}
\Vert\beta\Vert\} \quad\text{ and} \quad A_{2} := \{ x: \ell^{\prime}%
(\beta^{T}x)^{2} \leq4\overline{L}/p\},
\]
where the constants $c_{1},c_{2},p$ are given by Assumption
\ref{Assump-nondegeneracy}. Since $E_{P_{0}}[\ell^{\prime}(\beta^{T}X)^{2}]
\leq\overline{L}$ for any $\beta\in B,$ we have from Markov's inequality that
$\inf_{\beta\in B}P_{0}(X \in A_{2}) \geq1 - p/4.$ Consequently, it follows
from Assumption \ref{Assump-nondegeneracy} and union bound that $\inf
_{\beta\in B}P_{0}(X \in A_{1} \cap A_{2}) \geq3p/4.$

Recall that $\delta_{0} := {\rho_{\text{min}}^{2}\underline{L}}{R_{\beta}%
^{-2}M^{-2}\rho_{\max}^{-1}}.$ In addition, note that when $\delta\leq
\delta_{0}/4$, we have $\varphi_{\min} = \sqrt{\underline{L}} \rho_{\max
}^{-1/2} - \sqrt{\delta} R_{\beta} M \rho_{\min}^{-1}\geq\frac{1}{2}%
\sqrt{\underline{L}} \rho_{\max}^{-1/2}.$
Further, since $\delta<\delta_{1} \leq c_{1}^{2}c_{2}^{2}p^{2}\rho_{\min}%
^{2}\rho_{\max}^{-1}\underline{L}\overline{L}^{-2}/256,$ we have
\begin{align}
c_{1}c_{2} - 4\sqrt{\delta}\overline{L}p^{-1}\varphi_{\min}^{-1}\rho_{\min
}^{-1} \geq c_{1}c_{2}/2.\label{eq-delta-1}%
\end{align}

Next, if we choose $C_{1} > 0$ large enough such that the set $A_{3} := \{x:
\Vert x \Vert\leq C_{1}\}$ satisfies $P_{0}(X \in A_{3}) \geq1-p/4,$ then we
have $\inf_{\beta\in\Xi} P_{0}(X \in A_{1} \cap A_{2} \cap A_{3}) \geq p/2.$
The denominator in \eqref{lambda-simplified} is bounded from above as follows
whenever $x \in A_{1} \cap A_{2} \cap A_{3}$ and $\lambda\in[K_{1}\Vert
\beta\Vert, K_{2} \Vert\beta\Vert]:$ recalling that $T_{g}(x) := x +
\sqrt{\delta}g(\beta,\lambda;x)A(x)^{-1}\beta$ and $\bar{T}_{g}(x) := x +
2\sqrt{\delta} g(\beta,\lambda;x)A(x)^{-1}\beta$, it follows from the bounds
of $\vert g \vert$ in (\ref{gamma-UB-LB}) that
\begin{align*}
\Vert\bar{T}_{g}(x) \Vert & \leq\Vert x \Vert+ 2\sqrt{\delta}\vert g \vert
\rho_{\min}^{-1}\Vert\beta\Vert\leq C_{1} + 4\sqrt{\delta\bar{L}p^{-1}} \left(
\frac{1}{2}\sqrt{\underline{L}} \rho_{\max}^{-1/2}\right) ^{-1} \rho_{\min
}^{-1} =: C_{2},
\end{align*}
and similarly, $\Vert T_{g}(x) \Vert\leq C_{2}$ for $x \in A_{2}\cap A_{3}.$
Since $\Vert\beta^{T} T_{g}(x) \Vert\leq R_{\beta}C_{2} < \infty$ when $x \in
A_{2} \cap A_{3},$ if we let $C_{3} := \inf_{\vert u \vert\leq R_{\beta}C_{2}}
\ell^{\prime\prime}(u) > 0,$ we obtain that the denominator in
\eqref{lambda-simplified} is bounded from above by $C_{4} := 8K_{2}%
\delta^{1/2}\bar{L}p^{-1}C_{3}^{-1}(\frac{1}{2}\sqrt{\underline{L}} \rho
_{\max}^{-1/2})^{-2} + \rho_{\max}C_{2}R_{\beta}$ whenever $x \in A_{2} \cap
A_{3}.$ Combining this observation with that of
\eqref{lambda-simplified},\eqref{Lambda-num} and \eqref{eq-delta-1}, we obtain
that $\Lambda(x) \geq\sqrt{\delta}C \mathbf{1}_{\{x \in A_{1} \cap A_{2} \cap
A_{3}\}} $ for $C := (1/2)C_{0}c_{1}c_{2} {C}_{4}^{-1}.$

Finally, since $P_{0}(A_{1} \cap A_{2} \cap A_{3}) \geq p/2,$ we have
$E_{P_{0}}[\Lambda(\beta,\lambda;X)B(X)] \succeq\sqrt{\delta}\kappa_{1}
\mathbb{I}_{d+1}$ where $\kappa_{1} := pC\rho_{\max}^{-1}/2.$ As a consequence,
we have that $\nabla_{\theta}^{2}f_{\delta}(\theta) \succeq\sqrt{\delta}%
\kappa_{1}\mathbb{I}_{d+1}$ in Theorem \ref{Thm-Str-Convexity}. \hfill\Halmos \endproof

\begin{remark}
\textnormal{Suppose that $c_1c_2 = 0$ is the only non-negative number
for which the probability requirement in Assumption
\ref{Assump-nondegeneracy} is satisfied. In this case, we have
from the upper bound for $g$ in Proposition
\ref{prop:secondorder-info-smoothcase}b that $g\beta^TX = 0,$
$P_0$ almost surely. As a result, the numerator of $\Lambda(x)$ in
the right hand side of %
\eqref{Inter-Lambda-Simp-eq} is bounded from above by
$4\sqrt{\delta} (0 + \sqrt{\delta}g^2\beta^TA(x)^{-1}\beta)^2 \leq
4\delta^{3/2}\ell^\prime(\beta^Tx)^2\varphi_{\min}^{-2}\rho_{\min}^{-2},$
$%
P_0-$almost surely. Since the denominator of $\Lambda(x)$ is
bounded away from zero by a constant not dependent on $\delta,$ it
follows that $%
E_{P_0}[\Lambda(X)] = \kappa_3\delta^{3/2},$ for some non-negative
constant $%
\kappa_3.$ Since $\delta^{3/2} = o(\sqrt{\delta})$ as
$\delta \rightarrow 0,$ it is not possible to derive a positive
constant $\kappa_{1}$ that is not dependent on $\delta$ as in the
statement of Theorem \ref{Thm-Str-Convexity}.}
\label{Rem-nonzero-bX-need}
\end{remark}

\subsection{Proofs of the results pertaining to the structure of the worst
case distribution.}

\label{Sec-Pf-Thm-WC} In this section we provide proofs of Theorem
\ref{Thm-WC-Dist} and Theorem \ref{Prop-Comp-Stat} which shed light on the
structure of the adversarial distribution(s) attaining the supremum in
$\sup_{P:D_{c}(P_{0}, P) \leq\delta} E_{P}[\ell(\beta^{T}X)].$

\proof{\textbf{Proof of Theorem \protect\ref{Thm-WC-Dist}.}} Recall from
Assumption \ref{Assump-Obj-Conv} that $\ell(u)$ is convex and grows
quadratically or sub-quadratically as $\vert u \vert\rightarrow\infty.$
Therefore there exists $\lambda\geq0$ such that $f_{\delta}(\beta,\lambda) < \infty,$
and subsequently, $\inf_{\lambda}f_{\delta}(\beta,\lambda) < \infty.$ According to
Theorem ~\ref{Thm-strong-duality}, there exist a dual optimizer,
$\lambda_{\ast}(\beta)$ in $\text{arg}\min_{\lambda\geq0} f_{\delta}(\beta,\lambda)$
for any $\beta\in B.$

a) When $\lambda_{\ast}(\beta) = 0:$ We have $\inf_{\beta,\lambda}
f_{\delta}(\beta,\lambda) = f_{\delta}(\beta,0) = \sup_{u\in\mathbb{R}} \ell(u).$ Due to the
convexity of $\ell(\cdot),$ the finiteness of the optimal value $f_{\delta}(\beta,0) =
\sup_{u} \ell(u)$ implies that $\ell(\cdot)$ is a constant function. In this
case, any distribution $P$ satisfying $D_{c}(P,P_{0})\leq\delta$ is a worst
case distribution attaining the supremum in $\sup_{P:D_{c}(P,P_{0}) \leq
\delta} E_{P_{0}}[\ell(\beta^{T}X)].$

b) It follows from the characterization of the effective domain of $f_{\delta}(\cdot)$
in Lemma \ref{Lem-eff-dom-f} that $f_{\delta}(\beta,\lambda) = \infty$ when $\lambda<
\lambda_{thr} \xspace(\beta).$ Therefore, $\lambda_{\ast}(\beta) \geq
\lambda_{thr} \xspace(\beta).$

c) When $\lambda_{\ast}(\beta) > \lambda_{thr} \xspace(\beta):$ Recall from
Proposition \ref{prop:firstorderinfo} the expressions for $\partial_{+}
\ell_{rob}/\partial\lambda$ and $\partial_{-}\ell_{rob} /\partial\lambda.$
Further we have $f_{\delta}(\beta,\lambda) < \infty$ for $(\beta,\lambda) \in\mathbb{
U}_{1} := \{ (\beta,\lambda): \beta\in B, \, \lambda> \lambda_{thr}
\xspace(\beta)\}.$ Then it follows from \cite[Proposition 2.1]{Bertsekas1973}
that the left and right derivatives $\partial_{+}f_{\delta}/\partial\lambda$ and
$\partial_{-}f_{\delta}/\partial\lambda$ satisfy,
\begin{align*}
\frac{\partial_{+} f_{\delta}}{\partial\lambda}(\beta,\lambda)  & = \sqrt{\delta
}\left(  1 - E_{P_{0}}\left[ \beta^{T}A(X)^{-1}\beta\inf_{g \in\Gamma^{\ast
}(\beta,\lambda;X)} g^{2} \right]  \right)  \text{ and }\\
\frac{\partial_{-} f_{\delta}}{\partial\lambda}(\beta,\lambda)  & = \sqrt{\delta
}\left(  1 - E_{P_{0}}\left[ \beta^{T}A(X)^{-1}\beta\sup_{g \in\Gamma^{\ast
}(\beta,\lambda;X)} g^{2} \right]  \right) ,
\end{align*}
for $(\beta,\lambda) \in\mathbb{U}_{1}.$ Since $\lambda_{\ast}(\beta) >
\lambda_{thr} \xspace(\beta),$ we have from Lemma \ref{Lem-Finite-Exp}a and
the continuous differentiability of $\ell(\cdot)$ that $\Gamma^{\ast}%
(\beta,\lambda_{\ast}(\beta);x)$ is compact for $P_{0}-$almost every $x.$
Consequently, there exist measurable selections $g_{+}(\beta,\lambda_{\ast
}(\beta);x)$ and $g_{-}(\beta,\lambda_{\ast}(\beta);x)$ such that $g^{2}%
_{+}(\beta,\lambda_{\ast}(\beta);x) = \sup_{g \in\Gamma^{\ast}(\beta
,\lambda_{\ast}(\beta);X)} g^{2}$ and $g_{-}(\beta,\lambda_{\ast}(\beta);x) =
\inf_{g \in\Gamma^{\ast}(\beta,\lambda_{\ast}(\beta);X)} g^{2}$ (see
\cite[Proposition 7.50b]{bertsekas1978stochastic}). Letting $g_{+}%
(\beta,\lambda_{\ast}(\beta);X) = G_{+}$ and $g_{-}(\beta,\lambda_{\ast}%
(\beta);X) = G_{-},$ we obtain that,
\begin{align*}
\frac{\partial_{+} f_{\delta}}{\partial\lambda}(\beta,\lambda_{\ast}(\beta))  & =
\sqrt{\delta}\left(  1 - E_{P_{0}}\left[ G_{-}^{2} \beta^{T}A(X)^{-1}%
\beta\right]  \right)  \quad\text{ and } \quad\\
\frac{\partial_{-} f_{\delta}}{\partial\lambda}(\beta,\lambda_{\ast}(\beta))  & =
\sqrt{\delta}\left(  1 - E_{P_{0}}\left[ G_{+}^{2} \beta^{T}A(X)^{-1}%
\beta\right] \right) .
\end{align*}
Since $\lambda_{\ast}(\beta) \in\text{arg}\min_{\lambda\geq0} f_{\delta}(\beta
,\lambda),$ we have from the first order optimality condition that
$\partial_{+} f_{\delta}/\partial\lambda(\beta,\lambda_{\ast}(\beta)) \geq0$ and
$\partial_{-} f_{\delta}/\partial\lambda(\beta,\lambda_{\ast}(\beta)) \leq0.$ Thus
$\underline{c} = E_{P_{0}}[G_{-}^{2}\beta^{T}A(X)^{-1}\beta] \leq1$ and
$\overline{c} = E_{P_{0}}[G_{+}^{2}\beta^{T}A(X)^{-1}\beta] \geq1.$ With $G :=
ZG_{-} + (1-Z)G_{+}$ and $Z$ being an independent Bernoulli random variable
with $P(Z = 1) = (\overline{c}
- 1)/(\overline{c} -\underline{c}),$ we have that $E_{P_{0}}[G^{2}\beta^{T}A(X)^{-1}\beta] =
1.$ In addition, since $G \in\Gamma^{\ast}(\beta,\lambda;X)$ $P_{0}-$a.s., we
have that
\begin{align*}
X^{\ast}\in\text{arg} \max_{x^{\prime}\in\mathbb{R}^{d}} \left\{  \ell
(\beta^{T}x^{\prime}) - \lambda_{\ast}(\beta) c(X,x^{\prime})\right\}
\quad\text{ and } \quad E\left[ c(X,X^{\ast})\right]  = E[(\sqrt{\delta}
G)^{2}\beta^{T}A(X)^{-1}\beta] = \delta.
\end{align*}
As the complementary slackness conditions in Theorem 1 of
\cite{blanchet_quantifying_2016} are satisfied, we have that the distribution
of $X^{\ast}$ attains the supremum in $\sup_{P:D_{c}(P,P_{0}) \leq\delta}%
E_{P}[\ell(\beta^{T}X)].$

d) When $\lambda_{\ast}(\beta) = \lambda_{thr} \xspace(\beta):$ The worst case
distribution $P^{\ast}(\beta)$ attaining the supremum in $\sup_{P:D_{c}%
(P,P_{0}) \leq\delta}E_{P}[\ell(\beta^{T}X)]$ may not exist as demonstrated in
the following example. Suppose that
$\ell(u):=u^{2} - |u|(1-e^{-|u|}),$ $\|\beta\| = 1$, $P_{0}(dx) =
\delta_{\{\mathbf{0}\}}(dx),$ $\delta> 0$ and $A(x) = \mathbb{I}_{d}.$ For
this example, $\ell(\cdot)$ satisfies Assumption \ref{Assump-Obj-Conv} with
$\kappa= 1$ and $c(\cdot)$ satisfies Assumption \ref{Assump-c} with
$\rho_{\max} = \rho_{\min} = 1$. For any $\lambda\geq\lambda_{thr}
\xspace(\beta) = \sqrt{\delta},$ we have $\Gamma^{\ast}(\beta,\lambda
;\mathbf{0}) = \{\mathbf{0}\}$, and it follows that $f_{\delta}(\beta,\lambda) =
\lambda\sqrt{\delta}$ when $\lambda\geq\lambda_{thr} \xspace(\beta).$
Therefore $\lambda_{\ast}(\beta) = \lambda_{thr} \xspace(\beta) = \sqrt
{\delta}$ and the dual optimal value $f_{\delta}(\beta,\lambda_{\ast}(\beta)) =
\delta.$ However, this value is not attainable by $E_{P}[\ell(\beta^{T}X)]$
for for any $P$ satisfying $D_{c}(P,P_{0}) \leq\delta.$ This is because, we
have $E\Vert X \Vert^{2} \leq\delta$ for any $P$ such that $D_{c}(P,P_{0})
\leq\delta,$ and as a result we have $E_{P}[\ell(\beta^{T}X)] < \delta$ as in
the following series of inequalities:
\[
E_{P} \left[ \ell(\beta^{T} X)\right]  = E_{P}\left[  (\beta^{T} X)^{2} -
\vert\beta^{T} X \vert(1- \exp(-\vert\beta^{T} X \vert))\right]  < E_{P}%
(\beta^{T} X)^{2} \leq E_{P} \Vert X \Vert^{2} \leq\delta.
\]

e) When $\lambda_{\ast}(\beta) > \lambda_{thr}^{\prime}\xspace(\beta):$ In
this case, it follows from Lemma \ref{Lem-Str-conv-Thr} that the map
$\gamma\mapsto F(\gamma, \beta, \lambda_{\ast}(\beta);x)$ is strongly concave
for $P_{0}-$almost every $x.$ As a result, $\Gamma^{\ast}(\beta,\lambda_{\ast
}(\beta);X) $ is singleton, $P_{0}-$almost surely. As a result, the random
variables, $G,G_{+}, G_{-},$ identified in Part c satisfy that $P_{0}(G =G_{+}
= G_{-}) = 1$ and $E[G^{2}\beta^{T}A(X)^{-1}\beta] = 1.$ Therefore
$E[c(X,X^{\ast})] = \delta.$ Moreover, the above described uniqueness in
optimizer means that $X^{\ast}= X + \sqrt{\delta}GA(X)^{-1}\beta$ is the
unique element in $\text{arg}\max_{x^{\prime}\in\mathbb{R}^{d}}\{ \ell
(\beta^{T}x^{\prime}) - \lambda_{\ast}(\beta) c(X,x^{\prime})\},$ $P_{0}%
-$almost surely. Since any distribution $\bar{P}$ attaining the supremum in
$\sup_{P:D_{c}(P,P_{0}) \leq\delta}E_{P}[\ell(\beta^{T}X)]$ must satisfy that
if $\bar{X} \sim\bar{P}$ then $\bar{X} \in\text{arg}\max_{x^{\prime}%
\in\mathbb{R}^{d}}\{ \ell(\beta^{T}x^{\prime}) - \lambda_{\ast}(\beta)
c(X,x^{\prime})\}.$ As a result we must have that $\bar{X} = X^{\ast},$
$P_{0}-$almost surely. This verifies that the distribution of $X^{\ast}$ is
the unique choice that attains the supremum in $\sup_{P:D_{c}(P,P_{0})
\leq\delta}E_{P}[\ell(\beta^{T}X)].$ \hfill\Halmos
\endproof  \ 

\proof{\textbf{Proof of Theorem \protect\ref{Prop-Comp-Stat}.}} Since
$\beta\in B$ is fixed throughout the proof, we hide the dependence on $\beta$
from the parameters $\lambda_{\ast}(\beta)$ and $g(\beta,\lambda;x) $ in the
notation. Instead, to capture the dependence on $\delta,$ we let
$\lambda_{\ast}(\delta)$ be the choice of $\lambda$ that solves $\min
_{\lambda\geq0} f_{\delta}(\beta,\lambda)$ for a given choice of $\delta\in
(0,\delta_{1});$ here the minimizing $\lambda_{\ast}(\delta)$ is unique
because of the strong convexity characterization in Theorem
\ref{Thm-Str-Convexity}. For every $\delta< \delta_{1},$ we have from Part (a)
of Proposition \ref{prop:secondorder-info-smoothcase} that $\lambda_{\ast
}(\delta) > \lambda_{thr}^{\prime}\xspace(\beta).$ Then, we obtain the
following reasoning from Part (e) of Theorem \ref{Thm-WC-Dist}:

\begin{itemize}
\item[i)] For every $\delta< \delta_{1},$ the distribution of $X_{\delta
}^{\ast}= X + \sqrt{\delta}G_{\delta}A(x)^{-1}\beta$ is the unique choice that
attains the supremum in $\sup_{P:D_{c}(P,P_{0}) \leq\delta_{i}}E_{P}%
[\ell(\beta^{T}X)],$ with
$G_{\delta}:= g(\delta, \lambda_{\ast}(\delta);X),$
where $g(\delta,\lambda;x)$ is the unique real number that maximizes
$F(\gamma,\beta,\lambda; x)$ for $P_{0}-$almost every $x$ and $\lambda>
\lambda_{thr}^{\prime}\xspace(\beta);$

\item[ii)] Moreover, we have that $E[c(X,X^{\ast}_{\delta})] = \delta,$ and
consequently, $g(\delta,\lambda_{\ast}(\delta);X)$ satisfies  $E_{P_{0}}%
[g^{2}(\delta,\lambda_{\ast}(\delta); X)\beta^{T}A(X)^{-1}\beta] =  1.$
\end{itemize}


Following the implicit function theorem application in the proof of
Proposition \ref{prop:hessian-lrob} (see appendix Section \ref{Sec-App-Proofs}%
), we obtain that
\begin{align*}
\frac{\partial g}{\partial\delta}(\delta,\lambda_{\ast}(\delta);x) = -
\frac{\partial^{2}F/\partial\delta}{\partial^{2}F/\partial\gamma}%
(g(\delta,\lambda_{\ast}(\delta);x), \beta,\lambda_{\ast}(\delta);x) =
\frac{\ell^{\prime\prime}(\beta^{T}X_{\delta}^{\ast})g\beta^{T}A(X)^{-1}\beta
}{2\sqrt{\delta}\varphi_{g}},
\end{align*}
where $g$ and $\varphi$ in the right hand side denote, respectively,
$g(\delta,\lambda_{\ast};x)$ and $\varphi_{g}(\beta,\lambda_{\ast};x) :=
2\lambda_{\ast}(\delta) - \sqrt{\delta}\beta^{T}A(X)^{-1}\beta\ell
^{\prime\prime}(\beta^{T} X_{\delta}^{\ast}) > \varphi_{\min}\Vert\beta\Vert>
0$ (see Proposition \ref{prop:secondorder-info-smoothcase}b).

Next, define $H(\delta, \lambda) := E_{P_{0}}[g(\delta,\lambda;X)^{2}\beta
^{T}A(X)^{-1}\beta] - 1.$ Since $\lambda_{\ast}(\delta)$ satisfies $H(\delta,
\lambda_{\ast}(\delta)) = 0,$ a similar application of the implicit function
theorem results in,
\begin{align*}
\frac{\partial\lambda_{\ast}(\delta)}{\partial\delta} = -\frac{\partial
H/\partial\delta}{\partial H/\partial\lambda}(\delta,\lambda_{\ast}(\delta)) =
\frac{E_{P_{0}}[\ell^{\prime\prime}(\beta^{T}{X}^{\ast}_{\delta})(g\beta
^{T}A(X)^{-1}\beta)^{2}/\varphi]}{4\sqrt{\delta}E_{P_{0}}[g^{2}\beta
^{T}A(X)^{-1}\beta/\varphi]}.
\end{align*}
If we let $L(\delta) := \sqrt{\delta}g(\delta,\lambda_{\ast}(\delta);x),$ then
with an application of chain rule and use of above expressions for $\partial
g/\partial\delta, \partial\lambda_{\ast}(\delta)/\partial\delta$ and that of
$\partial g/\partial\lambda$ in the proof of Proposition
\ref{prop:hessian-lrob} (see (\ref{partial-g-lambda})), we obtain that
\begin{align*}
\frac{\partial L}{\partial\delta}(\delta) = \frac{g}{2\sqrt{\delta}} +
\frac{g\beta^{T}A(X)^{-1}\beta\ell^{\prime\prime}(\beta^{T}X_{\delta}^{\ast}%
)}{2\varphi} - \frac{g}{2\varphi}\frac{E_{P_{0}}[\ell^{\prime\prime}(\beta
^{T}{X}^{\ast}_{\delta})(g\beta^{T}A(X)^{-1}\beta)^{2}/\varphi]}{E_{P_{0}%
}[g^{2}\beta^{T}A(X)^{-1}\beta/\varphi]},
\end{align*}
if $g \neq0.$ When $\delta< \delta_{1},$ we have $\varphi> \varphi_{\min}%
\Vert\beta\Vert> 0$ (see \ref{prop:secondorder-info-smoothcase}b). Moreover,
$\beta^{T}A(X)^{-1}\beta\leq R_{\beta}\rho_{\min}^{-1}\Vert\beta\Vert$ and
$\ell^{\prime\prime}(\cdot) \in(0,M]$ (see Assumptions \ref{Assump-c} -
\ref{Assump-Obj-2diff}). As a result, we obtain that
\begin{align*}
\frac{2}{g}\frac{\partial L}{\partial\delta}(\delta) > \frac{1}{\sqrt{\delta}}
- \frac{\rho_{\min}^{-1}MR_{\beta}\Vert\beta\Vert}{\varphi_{\min} \Vert
\beta\Vert} = \frac{1}{\sqrt{\delta}} - \frac{1}{\sqrt{\delta_{0}} -
\sqrt{\delta}},
\end{align*}
where the last equality follows from the definitions of $\delta_{0}$ and
$\varphi_{\min}$ in the earlier Subsection
\ref{subsec:proofs-thms-smooth-str-convex}. Since $\delta< \delta_{1}
\leq\delta_{0}/4,$ we have that $2g^{-1}\partial L(\delta)/\partial\delta> 0 $
if $g \neq0$ and $\partial L(\delta)/\partial\delta= 0$ if $g = 0.$ Further,
observe that, as a consequence of the mean value theorem, the first order
optimality condition \eqref{Eq-First-Order-Opt-Gamma} means that
$g(\delta,\lambda_{\ast}(\delta);X) = \ell^{\prime}(\beta^{T}X)/(2\lambda
_{\ast}(\delta) - \sqrt{\delta}\beta^{T} A(X)^{-1}\beta\ell^{\prime\prime
}(\eta)),$ for some $\eta$ between the real numbers $\beta^{T}X$ and
$\beta^{T}{X}_{\delta}^{\ast}.$ Since $2\lambda_{\ast}(\delta) - \sqrt{\delta
}\beta^{T} A(X)^{-1}\beta\ell^{\prime\prime}(\eta) \geq\varphi_{\min}%
\Vert\beta\Vert> 0,$ we have that the sign of $G_{\delta}:= g(\delta
,\lambda_{\ast}(\delta);X)$ matches with that of $\ell^{\prime}(\beta^{T}X).$
As a result, with $L(\delta) := \sqrt{\delta}g(\delta,\lambda;X) =
\sqrt{\delta}G_{\delta},$ the claims made in Proposition \ref{Prop-Comp-Stat}b
- \ref{Prop-Comp-Stat}d are verified. This completes the proof of Theorem
\ref{Prop-Comp-Stat}. \hfill\Halmos \endproof

\vspace{-0.1in}

\subsection{Proofs of the results on rates of convergence.}

\label{Sec-RatesofConv-Proofs} Lemma \ref{Lem-Grad-Sec-Moment} below,
establishing finite second moments for the gradients (or) subgradients
utilized in SGD schemes, is useful towards proving Propositions
\ref{Prop-Conv-SGD-1} and \ref{Prop-Conv-SGD-2}. Recall the definitions of
$\mathbb{U}_{\eta}$ in (\ref{Defn-Ueta}) and $D(\beta,\lambda;X)$ in \eqref{noisy-subgradient}.

\begin{lemma}
Suppose that Assumptions \ref{Assump-c}, \ref{Assump-Obj-Conv} are
satisfied, $\ell(\cdot)$ is continuously differentiable, $\eta > 0$ and $%
E_{P_0}\Vert X \Vert^4 < \infty.$ For any $\theta \in \mathbb{U}_\eta,$ let $%
h(\theta;X)$ be such that $h(\theta;X) \in D(\theta;X),$ $P_0-$almost
surely. Then there exists a positive constant $G_\eta$ such that $%
E_{P_0}\Vert h(\theta;X)\Vert^2 \leq G_\eta$ for any $\theta \in \mathbb{U}%
_\eta.$ \label{Lem-Grad-Sec-Moment}
\end{lemma}

The proof of Lemma \ref{Lem-Grad-Sec-Moment} is presented in Appendix
\ref{Sec-App-Proofs}.

\proof{\textbf{Proof of Proposition \protect\ref{Prop-Conv-SGD-1}.}} a) When
$\delta< \delta_{0},$ it follows from Proposition \ref{prop:firstorderinfo}
and Proposition \ref{Prop-sDelta-gradients} that the subgradient set
$\partial\ell_{rob} \xspace(\beta,\lambda;X) = \{\nabla_{\theta}\ell_{rob}
\xspace(\beta,\lambda;X)\},$ $P_{0}-$almost surely. Since $\lambda>
\lambda_{thr}^{\prime}\xspace(\beta) \geq\lambda_{thr} \xspace(\beta)$ for
every $(\beta,\lambda) \in\mathbb{W}$ (see Proposition
\ref{prop:secondorder-info-smoothcase}a), it follows from Lemma
\ref{Lem-Grad-Sec-Moment} that $\sup_{\theta\in\mathbb{W}} E\Vert
\nabla_{\theta}\ell_{rob} \xspace(\theta;X)\Vert^{2} < \infty,$ when $\delta<
\delta_{0}.$
As a consequence, we have from Theorem 2 and the remark following Theorem 4 in
\cite{Shamir:2013:SGD:3042817.3042827} that $E[f_{\delta}(\theta_{k})]
-f_{\ast}= O(k^{-1/2}\log k)$ and $E[f_{\delta}(\bar{\theta}_{k})] - f_{\ast}=
O(k^{-1/2}),$ as $k \rightarrow\infty.$ Proposition \ref{Prop-Conv-SGD-1}a now
follows as a consequence of Markov's inequality.

b) When $\delta< \delta_{1},$ it follows from the positive definiteness of
Hessian around the unique minimizer $\theta_{\ast}:= \text{arg}\min f_{\delta
}(\theta) $ (see Theorem \ref{Thm-Str-Convexity}) that there exists
$\varepsilon> 0$ satisfying $(\theta- \theta_{\ast})^{T}\nabla_{\theta
}f_{\delta}(\theta) \geq\kappa_{1}\sqrt{\delta} \Vert\theta-\theta_{\ast}%
\Vert^{2}$ for all $\theta\in\mathbb{V} $ and $\Vert\theta- \theta_{\ast}%
\Vert\leq\varepsilon.$ Further, due to the uniqueness of the minimizer, we
also have $(\theta- \theta_{\ast})^{T} \nabla_{\theta}f_{\delta}(\theta) > 0.$
Similar to Part a), as $\lambda> \lambda_{thr}^{\prime}\xspace(\beta)
\geq\lambda_{thr} \xspace(\beta)$ for every $(\beta,\lambda) \in\mathbb{W},$
we have due to Lemma \ref{Lem-Grad-Sec-Moment} that $\sup_{\theta\in
\mathbb{W}} E\Vert\nabla_{\theta}\ell_{rob} \xspace(\theta;X)\Vert^{2} <
\infty.$ Taylor's expansion of $\nabla_{\theta}f_{\delta}(\theta)$ results
in,
\begin{align}
\Vert\nabla_{\theta}f_{\delta}(\theta) - \nabla^{2}_{\theta}f_{\delta}%
(\theta_{\ast})^{T}(\theta- \theta_{\ast})\Vert= o\left(  \Vert\theta-
\theta_{\ast}\Vert\right) ,\label{Inter-Prop-Conv}%
\end{align}
for $\theta\in\mathbb{W}.$ With these conditions being satisfied, it follows
from \cite[Theorem 2]{PJ92} that
$
\sqrt{k}(\bar{\theta}_{k} - \theta_{\ast}) \overset{D}{\longrightarrow}
\mathcal{N}(\mathbf{0}, \Sigma),
$
as $k \rightarrow\infty,$ where $\Sigma:= (\nabla^{2}_{\theta}f_{\delta
}(\theta_{\ast}))^{-1} \text{Cov}[\nabla_{\theta}\ell_{rob} \xspace(\theta
_{\ast};X)] ((\nabla^{2}_{\theta}f_{\delta}(\theta_{\ast}))^{-1})^{T}.$
If we let $Z \sim\mathcal{N}(0,\mathbb{I}_{d+1}),$
then due to continuous mapping theorem, we have that the distribution of
$k (\bar{\theta}_{k} - \theta_{\ast})^{T} \nabla_{\theta}^{2} f_{\delta
}(\theta_{\ast}) (\bar{\theta_{k}}-\theta_{\ast})$
is convergent to that of
\[
Z^{T}\Sigma^{1/2} \nabla^{2}_{\theta}f_{\delta}(\theta_{\ast}) \Sigma^{1/2}Z =
Z^{T} \nabla^{2}_{\theta}f_{\delta}(\theta_{\ast})^{-1/2}\text{Cov}%
[\nabla_{\theta}\ell_{rob} \xspace(\theta_{\ast};X)] \nabla^{2}_{\theta
}f_{\delta}(\theta_{\ast})^{-1/2}Z.
\]
The local strong convexity characterization in Theorem \ref{Thm-Str-Convexity}
yields that that the maximum eigen value of $\nabla^{2}_{\theta}f_{\delta
}(\theta_{\ast})^{-1/2}$ is bounded from above by a constant times
$\delta^{-1/4}.$ As a result of the above described convergence in
distribution, we have that
\begin{align*}
(\bar{\theta}_{k} - \theta_{\ast})^{T} \nabla_{\theta}^{2} f_{\delta}%
(\theta_{\ast})(\bar{\theta}_{k}-\theta_{\ast}) = O_{p}\left(  k^{-1} \right) .
\end{align*}
Now it follows from the local joint strong convexity of $f_{\delta}(\cdot)$ in
Theorem \ref{Thm-Str-Convexity} and \eqref{Inter-Prop-Conv} that
\begin{align*}
f_{\delta}(\bar{\theta}_{k}) -f_{\ast} & \leq\nabla_{\theta}f_{\delta}%
(\bar{\theta}_{k})^{T}(\bar{\theta}_{k} - \theta_{\ast}) - \frac{\kappa
\sqrt{\delta}}{2} \Vert\theta_{k} -\theta_{\ast}\Vert^{2}\\
& = (\bar{\theta}_{k} - \theta_{\ast})^{T} \nabla_{\theta}^{2} f_{\delta
}(\theta_{\ast})(\bar{\theta_{k}}-\theta_{\ast}) - \left( \frac{\kappa
\sqrt{\delta}}{2} + o(1)\right) \Vert\bar{\theta}_{k} - \theta_{\ast}\Vert^{2}
= O_{p}\left(  k^{-1} \right) .
\end{align*}
This completes the proof of Proposition \ref{Prop-Conv-SGD-1}. \hfill\Halmos
\endproof

\section{Conclusions.}

\label{Sec-Conclusions}

Our main objective in this paper has been to set the stage for algorithms and analysis of a flexible class of DRO problems. Our  motivation stem from the observations that i) a flexible choice of the distributional uncertainty region is useful towards to fully exploiting the advantages of DRO in data-driven contexts, and  that ii) the existing computational methods largely pertain to Lipschitz losses and do not scale well with data-size. We show that in the case of affine decision rules and convex loss functions, robustification with a more flexible state-dependent Mahalanobis cost function does not introduce significantly additional computational complexity relative to the non-DRO counter-part (in terms of standard benchmark iterative algorithms used to solve the non-DRO problem). In some cases, interestingly, DRO introduces strong-convexity which results in lower iteration complexity. 

Naturally, the algorithmic approach and structural analysis presented in this paper can be considered in DRO formulations with further general cost functions of the form $c(x,x^{\prime}) = u(x-x^{\prime})$ or $c(x,x^{\prime}) = u(x^{\prime})-u(x) -\nabla u(x)(x'-x)^T$, for a strongly convex function $u(\cdot)$ with Lipschitz-continuous gradients. While such extensions may render the inner maximization in \eqref{Rob-Obj-Fn} as a multi-dimensional optimization problem (as opposed to the line-search in the state-dependent Mahalanobis case), a number of observations and structural properties are expected to continue to hold; for example, observations relating to convexity properties, magnitude of mass transportation in the worst-case distribution being of size $O_{p}(\sqrt{\delta}),$ computation  of stochastic gradients by means of envelope theorem, etc. are expected to generalize  to the above families  of strongly convex, smooth transportation cost functions. We leave this exploration as a question for future research. 

Our philosophy is that by providing a general analysis for a flexible class of cost functions, a modeler will be able  to choose a cost function that enhances out-of-sample performance in a way that is convenient and meaningful for the needs of the modeling situation. While examples of how one may choose the transportation cost function in a data-driven way are available in existing literature (see, for example, \cite{blanchet2017data}), systemic treatment of the contextual choice of transportation cost is an essential question for future research. 

\section*{Acknowledgments.}

Material in this paper is based upon work supported by the Air Force Office of Scientific Research under award number FA9550-20-1-0397. Additional support is gratefully acknowledged from NSF grants 1915967, 1820942 and 1838676, DARPA grant N660011824028, MOE SRG ESD 2018 134 and China Merchants Bank.



\bibliographystyle{informs2014}
\bibliography{DR2}
\clearpage

\begin{APPENDICES}
\section{Proofs of technical results.}
\label{Sec-App-Proofs}
The proofs of technical results in this section are presented in a
logical order determined by their dependence on earlier proved results
(rather than being based on the order in which they appear in the
paper).
\proof{\textbf{Proof of Lemma \protect\ref{Lem-Finite-Exp}.}}  a)
Given $\varepsilon > 0,$ it follows from the growth condition in
Assumption \ref{Assump-Obj-Conv} that there exist a positive constant
$%
C_{\varepsilon}$ satisfying
$\ell(u) \leq (\kappa + \varepsilon/2)u^2 + C_\varepsilon$ for all
$u \in \mathbb{R}.$ Since any $g \in \Gamma^\ast(\beta,\lambda;x)$ is
a maximizer of $F(\cdot,\beta,\lambda;x)$, it follows immediately
that 
$F(g,\beta,\lambda;x) \geq F(0,\beta,\lambda;x).$ 
Recalling the definition of $F(\cdot,\beta,\lambda;x)$ from %
\eqref{Univar-Opt}, the above inequality results in,
	\begin{equation*}
(\kappa+\varepsilon/2)\left(\beta^Tx+ g\sqrt{\delta}\beta^TA(x)^{-1}\beta
\right)^2 + C_{\varepsilon}
-(\kappa+\varepsilon)\delta\left(\beta^TA(X)^{-1}\beta g\right)^2\geq
\ell(\beta^Tx),
	\end{equation*}
	once we utilize that $\lambda \geq (\kappa+\varepsilon)\sqrt{\delta}%
	\beta^TA(x)^{-1}\beta$ and
$\ell(u) \leq (\kappa + \varepsilon/2)u^2 + C_\varepsilon.$
The above inequality can be equivalently written after a few
basic algebraic steps as,
	\begin{equation*}
\left(g\sqrt{\delta}\beta^TA(x)^{-1}\beta - \frac{2\kappa+\varepsilon}{%
\varepsilon}\beta^Tx\right)^2\leq \frac{2}{\varepsilon}C_{\varepsilon}-\frac{%
2}{\varepsilon}\ell(\beta^Tx) + \frac{(\beta^Tx)^2}{\varepsilon^2}%
(2\varepsilon^2 + 6\kappa\varepsilon + 4\kappa^2).
	\end{equation*}
	We first upper bound the right hand side by using
$\ell(\beta^Tx) \geq \ell(0) + \beta^Tx \ell^\prime(0),$ which
holds due to the convexity of $%
	\ell(\cdot).$ Next, utilizing the inequality
$\vert a - b\vert \geq \vert \vert a \vert - \vert b \vert
\vert$ in the left hand side, we arrive at,
	\begin{align*}
\sqrt{\delta}\vert g \vert \beta^TA(x)^{-1}\beta \leq \sqrt{\frac{2}{%
\varepsilon}\left(C_{\varepsilon}+ \vert \ell(0) \vert + \vert
\ell^\prime(0) \vert \vert \beta^Tx \vert \right)} + 4\frac{\kappa +
\varepsilon}{\varepsilon}\vert \beta^Tx \vert.
	\end{align*}
	Since $\sqrt{x} \leq 1+x$ for $x \geq 0,$ the above inequality
verifies Part a) of Lemma \ref{Lem-Finite-Exp}.
	
	b) Utilizing the bounds
$\lambda \geq (\kappa+\varepsilon)\sqrt{\delta}%
	\beta^TA(x)^{-1}\beta$ and
$\ell(u) \leq (\kappa + \varepsilon/2)u^2 + C_\varepsilon$ in
the expression for $F(\cdot)$ in \eqref{Univar-Opt}, we obtain
that
	\begin{align*}
F(g,\beta,\lambda;x) \leq \lambda\sqrt{\delta} + C_\varepsilon + (\kappa +
\varepsilon/2)(\beta^Tx)^2 + 2(\kappa + \varepsilon/2)\vert \beta^Tx \vert
\sqrt{\delta} \vert g \vert\beta^TA(x)^{-1}\beta.
	\end{align*}
	Since $\ell_{rob} \xspace(\beta,\lambda;x) =
F(g,\beta,\lambda;x)$ for $g \in \Gamma(\beta,\lambda;x),$ we
obtain the following bound for $\ell_{rob} %
	\xspace(\beta,\lambda;x)$ once we substitute the bound for
$\sqrt{\delta}%
	\vert g \vert \beta^TA(x)^{-1}\beta$ from Part a):
	\begin{align*}
\ell_{rob} \xspace(\beta,\lambda; x) \leq \lambda\sqrt{\delta} +
C_\varepsilon + (2\kappa + \varepsilon)\vert \beta^Tx \vert \left( 1 + \vert
\beta^Tx \vert\right)\left( 1 + C_1\varepsilon^{-1}\right).
	\end{align*}
	This verifies Part b) of Lemma \ref{Lem-Finite-Exp}.\hfill \Halmos
\endproof
\
\proof{\textbf{Proof of Lemma \protect\ref{Lem-eff-dom-f}.}}
	For any fixed $\beta,\lambda$ and $x,$ it follows from the growth condition
	in Assumption \ref{Assump-Obj-Conv} that i) $\lim_{\gamma \rightarrow \pm
		\infty}F(\gamma, \beta,\lambda;x) = -\infty$ if $\lambda > \kappa\sqrt{\delta%
	}\beta^TA(x)^{-1}\beta$ and ii) $\lim_{\gamma \rightarrow \pm \infty}
	F(\gamma, \beta,\lambda;x) = +\infty$ if $\lambda < \kappa\sqrt{\delta}%
	\beta^TA(x)^{-1}\beta.$ Further, $F(\gamma,\beta,\lambda,x)$ is continuous
	in $\gamma$ because of the continuity of $\ell(\cdot).$ Therefore we obtain
	that $\Gamma^\ast(\beta,\lambda;x) \neq \varnothing$ when $\lambda > \kappa%
	\sqrt{\delta}\beta^TA(x)^{-1}\beta.$ Likewise, $\Gamma^\ast(\beta,\lambda;x)
	= \varnothing$ when $\lambda < \kappa\sqrt{\delta}\beta^TA(x)^{-1}\beta.$
	This completes the proof of Parts a) and b) of Lemma \ref{Lem-eff-dom-f}.
	
	To verify the inclusions in the final statement of Lemma \ref{Lem-eff-dom-f}
	we proceed as follows: Whenever $\lambda < \lambda_{thr} \xspace(\beta)$ we
	have $\ell_{rob} \xspace(\beta,\lambda;x) = +\infty$ with positive
	probability. Therefore, $\mathbb{U}$ is contained in $\{(\beta,\lambda):
	\beta \in B, \lambda \geq \lambda_{thr} \xspace(\beta)\}.$ On the other
	hand, if $\lambda > \lambda_{thr} \xspace(\beta),$ we have $\lambda \geq
	(\kappa + \varepsilon) \sqrt{\delta}\beta^TA(x)^{-1}\beta$ for some $%
	\varepsilon > 0,$ $P_0-$almost every $x.$ Since $\Vert \beta \Vert \leq
	R_\beta$ and $E\Vert X \Vert^2 < \infty,$ it follows Lemma \ref%
	{Lem-Finite-Exp}b that $f_\delta(\beta,\lambda) = E_{P_0}[\ell_{rob} \xspace%
	(\beta,\lambda;X)] < \infty.$ Therefore $\{(\beta,\lambda): \beta \in B,
	\lambda > \lambda_{thr} \xspace(\beta)\}$ is contained in $\mathbb{U}.$ This
	completes the proof of Lemma \ref{Lem-eff-dom-f}.\hfill \Halmos
\endproof
\
\proof{\textbf{Proof of Proposition \protect\ref{prop:firstorderinfo}.}}
a) Since
$\lambda > \kappa\sqrt{\delta}\beta^TA(x)^{-1}\beta$ for
$P_0-$%
	almost surely every $x,$ Proposition \ref{prop:firstorderinfo}
follows directly from Lemma \ref{Lem-eff-dom-f}.\\
b) Consider any fixed $x \in \mathbb{R}^d, C_3 < \infty$ and
$\eta > 0.$ Define the set
$A := \{(\beta,\lambda) \in B \times \mathbb{R}_+: \Vert \beta
\Vert < C_3, \ \lambda \geq \lambda_{thr} \xspace(\beta) +
\eta\}.$ Then for $(\beta,\lambda) \in A,$ we have the
following two conditions satisfied: i)
$\lambda \geq (\kappa+\varepsilon)\sqrt{\delta}%
	\beta^TA(x)^{-1}\beta$ for some $\varepsilon > 0,$ for
$P_0-$almost every $%
	x; $ and ii) $\beta^T A(x)^{-1}\beta $ is bounded away from
zero if $\beta\neq\mathbf{0}$ (due to Assumption \ref{Assump-c}). Therefore, for any
$(\beta,\lambda)$ in $A,$ we have from Lemma
\ref{Lem-Finite-Exp}a that there exists a positive constant
$C_x$ such that
$\Gamma^\ast(\beta,\lambda;x) \subseteq [-C_x, C_x].$ Thus for
$%
	(\beta,\lambda) \in A,$ it suffices to restrict the univariate
optimization problem \eqref{Univar-Opt} within the compact set
$[-C_x,C_x],$ as in, $%
	f_\delta(\beta,\lambda;x) = \sup_{\gamma\in[-C_x,C_x]} F(\gamma,
\beta,\lambda;x).$
	
	Next, for $i=1,\ldots,K$, define
\begin{align*}
G_x(i,\gamma,\beta,\lambda) := \ell_i \left(\beta^Tx + \gamma\sqrt{\delta}
\beta^T A(x)^{-1}\beta \right) - \lambda \sqrt{\delta} \big(\gamma^2 \beta^T
A(x)^{-1}\beta - 1\big).
\end{align*}
Then see that
$F(\gamma,\beta,\lambda;x) =
\max_{i=1,\ldots,K}G_x(i,\gamma,\beta,\lambda)$ and
\[ f_\delta(\beta,\lambda;x) = \sup_{\overset{i \in \{1,\ldots,k\},}
{\gamma\in[-C_x,C_x]}} G_x(i,\gamma, \beta,\lambda). \]
Considering discrete topology for the variable $i,$ see that
$G_x(i,\gamma,\beta,\lambda)$ is upper semicontinuous in
$(\gamma,i)$ and continuously differentiable over variables
$(\beta,\lambda)$. Specifically, for $i \in \{1,\ldots,K\},$ $j\in \left\{1,\dots,d\right\}$ and $\gamma^\ast \in \arg\max_x G_x(i,\gamma^\ast,\beta,\lambda), $ we have from the first-order optimality condition ($\ell_i^\prime(\beta^T x + \gamma^\ast \sqrt{\delta}\beta^T A(x)^{-1}\beta) -2 \lambda \gamma^\ast =0$) that
	\begin{align*}
	\frac{\partial G_x}{\partial \beta_j}&(i,\gamma^\ast,\beta,\lambda) =
	\ell^{\prime}_i(\beta^T x + \gamma^\ast\sqrt{\delta}\beta^T
	A(x)^{-1}\beta)(x +\gamma^\ast\sqrt{\delta}A(x)^{-1}\beta)_j, \text{ and }\\
	 &\frac{\partial G_x}{\partial \lambda}(i,\gamma^\ast,\beta,\lambda) = \sqrt{\delta}
	\big(1-{\gamma^\ast}^2\beta^T A(x)^{-1}\beta\big),
	\end{align*}
	where $(x+ \gamma^\ast \sqrt{\delta} A(x)^{-1}\beta)_j$ is the $j$th element of the vector $x+ \gamma^\ast \sqrt{\delta} A(x)^{-1}\beta.$  Moreover, since
$\ell(u) := \max_{i = 1,\ldots,K} \ell_i(u),$ we have that
\[\frac{\partial_+\ell}{\partial u}(u) = \max_{i \, \in \, \text{argmax}_j
\ell_j(u)} \ell_i^\prime(u) \quad \text{ and } \quad
\frac{\partial_{-}\ell}{\partial u}(u) = \min_{i \, \in \,
\text{argmax}_j \ell_j(u)} \ell_i^\prime(u),\] for any
$u \in \mathbb{R}.$ Equipped with these observations and the
fact that
$\ell_{rob}(\beta,\lambda;x) =
\sup_{\gamma,i}G_x(\gamma,i,\beta,\lambda),$ we arrive at the
following conclusions (i) and (ii) below as a consequence of
Envelope theorem \cite[Corollary 4]{milgrom2002envelope}:
i) When $(\beta,\lambda) \in A,$ the functions
$\lambda\mapsto \ell_{rob}(\beta,\lambda;x)$, $%
	\beta_j\mapsto \ell_{rob}(\beta,\lambda;x)$ are absolutely
continuous, and have left and right derivative given by
\eqref{eqn-directional-derivative-a} -
\eqref{eqn-directional-derivative-d}. Indeed, as an example
for deriving $\partial_+\ell_{rob}/\partial \beta_j$, see that,
\begin{align*}
&\frac{\partial_+ \ell_{rob}}{\partial \beta_j}(\beta,\lambda;x)
=  \max\left\{ \frac{\partial G_x}{\partial \beta_j}(i^\ast,\gamma^\ast,\beta,\lambda):
(i^\ast,\gamma^\ast) \in \arg\max_{i,\gamma} G_x(i,\gamma,\beta,\lambda)\right\}\\
&\quad=  \max\left\{ \frac{\partial G_x}{\partial \beta_j}(i^\ast,\gamma^\ast,\beta,\lambda):
\gamma^\ast  \in \Gamma^\ast(\beta,\lambda;x), \
i^\ast  \in \arg\max_{i=1,..,K} G_x(i,\gamma^\ast,\beta,\lambda)\right\}\\
&\quad= \max_{\gamma  \in \Gamma^\ast(\beta,\lambda;x)}  \left\{\max_{i \in
\arg \max_{j} \ell_j\left( \beta^Tx + \gamma\sqrt{\delta}\beta^TA(x)^{-1}\beta\right)}
\ell^{\prime}_i(\beta^T x + \gamma\sqrt{\delta}\beta^T
A(x)^{-1}\beta)  (x+\gamma\sqrt{\delta}A(x)^{-1}\beta)_j
\right\}
\\
&\quad = \max_{\gamma  \in \Gamma^\ast(\beta,\lambda;x)}  \frac{\partial_+\ell
}{\partial u}\left( \beta^T x + \gamma\sqrt{\delta}\beta^TA(x)^{-1}\beta\right)
(x+\gamma\sqrt{\delta}A(x)^{-1}\beta)_j.
\end{align*}
Likewise,
$\frac{\partial_- \ell_{rob}}{\partial
\beta_j}(\beta,\lambda;x)$ can be seen equal to,
\begin{align*}
&  \min_{\gamma  \in \Gamma^\ast(\beta,\lambda;x)}  \left\{\min_{i \in
\arg \max_{j} \ell_j\left( \beta^T x + \gamma\sqrt{\delta}\beta^TA(x)^{-1}\beta\right)}
\ell^{\prime}_i(\beta^T x + \gamma\sqrt{\delta}\beta^T
A(x)^{-1}\beta)  (x+\gamma\sqrt{\delta}A(x)^{-1}\beta)_j
\right\}
\\
&\quad = \max_{\gamma  \in \Gamma^\ast(\beta,\lambda;x)}  \frac{\partial_-\ell
}{\partial u}\left( \beta^T x + \gamma\sqrt{\delta}\beta^TA(x)^{-1}\beta\right)
(x+\gamma\sqrt{\delta}A(x)^{-1}\beta)_j.
\end{align*}
The directional derivatives with respect to the variable
$\lambda$ can be derived similarly.
ii) Then we have that the partial derivatives exist as in
\eqref{derivative} if and only if the respective sets,
\begin{align*}
&\left\{\frac{\partial G_x}{\partial \beta_j} (i^\ast,\gamma^\ast,
\beta,\lambda): (i^\ast,\gamma^\ast) \in
\arg\max_{i,\gamma}G_x(i,\gamma,\beta,\lambda) \right\},  \\
&\left\{\frac{\partial G_x}{\partial \lambda} (i^\ast,\gamma^\ast,
\beta,\lambda): (i^\ast,\gamma^\ast) \in
\arg\max_{i,\gamma}G_x(i,\gamma,\beta,\lambda) \right\}
\end{align*}
are singleton; this condition of being singleton is satisfied
if and only if the respective sets,
\begin{align*}
&\left\{\frac{\partial_+F}{\partial \beta_j} (\gamma,\beta,\lambda;x),
\frac{\partial_-F}{\partial \beta_j} (\gamma,\beta,\lambda;x):
\gamma \in \Gamma^\ast(\beta,\lambda;x) \right\},\\
&\left\{\frac{\partial_+F}{\partial \lambda} (\gamma,\beta,\lambda;x),
\frac{\partial_-F}{\partial \lambda} (\gamma,\beta,\lambda;x):
\gamma \in \Gamma^\ast(\beta,\lambda;x)\right\}
\end{align*}
are singleton.
Since these expressions hold for any $C_3 ,\eta \in
(0,\infty),$ Proposition \ref{prop:firstorderinfo} stands
verified. \hfill\Halmos \endproof \
The following technical result is useful towards proving Lemma
\ref{Lem-threshold}.
\begin{lemma}
	Suppose that Assumptions \ref{Assump-c},\ref{Assump-Obj-Conv} hold and $%
	\ell(\cdot)$ is continuously differentiable. Then for fixed $\beta \in B$
	and $x \in \mathbb{R},$ the map $\lambda \mapsto \ell_{rob} \xspace%
	(\beta,\lambda;x)$ is right-continuous at $\lambda = \lambda_{thr} \xspace%
	(\beta)$ if $\ell_{rob} \xspace(\beta,\lambda_{thr} \xspace(\beta);x) <
	\infty.$ \label{Lem-Right-Cont-lambda}
\end{lemma}
\proof{\textbf{Proof of Lemma \protect\ref{Lem-Right-Cont-lambda}.}}
	Suppose that $\ell_{rob} \xspace(\beta,\kappa\sqrt{\delta}%
	\beta^TA(x)^{-1}\beta;x)<\infty.$ Then for any $\varepsilon>0$, there exist $%
	\gamma \in \mathbb{R}$ such that 
	$F(\gamma, \beta,\kappa\sqrt{\delta}\beta^TA(x)^{-1}\beta;x)>\ell_{rob} %
	\xspace(\beta,\kappa\sqrt{\delta}\beta^TA(x)^{-1}\beta;x)-\varepsilon.$
	Thanks to the continuity of $F(\gamma,\beta,\lambda;x)$ with respect to $%
	\lambda$,
	\begin{align*}
	&\liminf_{\lambda\downarrow (\kappa\sqrt{\delta}\beta^TA(x)^{-1}\beta)}%
	\ell_{rob} \xspace(\beta,\lambda;x) \geq \lim_{\lambda\downarrow (\kappa\sqrt{\delta}\beta^TA(x)^{-1}\beta)} F(\gamma,\beta,\lambda;x) \\
	&\qquad\qquad= F(\gamma,\beta,\kappa\sqrt{\delta}\beta^TA(x)^{-1}\beta;x)>
	\ell_{rob} \xspace(\beta,\kappa\sqrt{\delta}\beta^TA(x)^{-1}\beta;x)-%
	\varepsilon.
	\end{align*}
	Since $\varepsilon > 0$ is arbitrary, we have
	\begin{align*}
	\liminf_{\lambda\downarrow (\kappa\sqrt{\delta}\beta^TA(x)^{-1}\beta)}
	\ell_{rob} \xspace(\beta,\lambda;x) \geq \ell_{rob} \xspace(\beta,\kappa%
	\sqrt{\delta}\beta^TA(x)^{-1}\beta;x).
	\end{align*}
	Moreover, as $\ell_{rob} \xspace(\beta,\lambda;x) - \lambda\sqrt{\delta}$ is
	decreasing in $\lambda$, we also have that,
	\begin{align*}
	\limsup_{\lambda\downarrow (\kappa\sqrt{\delta}\beta^TA(x)^{-1}\beta)}%
	\ell_{rob} \xspace(\beta,\lambda;x)-\lambda\sqrt{\delta}\leq \ell_{rob} %
	\xspace(\beta,\kappa\sqrt{\delta}\beta^TA(x)^{-1}\beta;x)-\kappa\delta%
	\beta^TA(x)^{-1}\beta,
	\end{align*}
	thus yielding,
	$\limsup_{\lambda\downarrow (\kappa\sqrt{\delta}\beta^TA(x)^{-1}\beta)}%
	\ell_{rob} \xspace(\beta,\lambda;x)\leq \ell_{rob} \xspace(\beta,\kappa\sqrt{%
		\delta}\beta^TA(x)^{-1}\beta;x),$
	and consequently, 
	$$\lim_{\lambda\downarrow (\kappa\sqrt{\delta}\beta^TA(x)^{-1}\beta)}%
	\ell_{rob} \xspace(\beta,\lambda;x) = \ell_{rob} \xspace(\beta,\kappa\sqrt{%
		\delta}\beta^TA(x)^{-1}\beta;x).$$ 
	In addition, $\ell_{rob} \xspace(\beta,\lambda;x)$ is continuous at $\lambda
	$ even if $\lambda > \kappa\sqrt{\delta}\beta A(x)^{-1}\beta$ (due to the
	convexity of $\ell_{rob} \xspace(\cdot;x)$ as in Lemma \ref{Lem-conv-lrob}).
	Therefore, $\lambda \mapsto \ell_{rob} \xspace(\beta,\lambda;x)$ is
	right-continuous at $\lambda = \lambda_{thr} \xspace(\beta)$ if $\ell_{rob} %
	\xspace(\beta,\lambda_{thr} \xspace(\beta);x) < \infty.$
\hfill \Halmos
\endproof
\
\proof{\textbf{Proof of Lemma \protect\ref{Lem-threshold}.}}  Fix any
$\beta \in B.$ It
follows from the characterization of $\mathbb{U}$ in Lemma
\ref{Lem-eff-dom-f} that $f_\delta(\beta,\lambda) = +\infty$ if
$\lambda < \lambda_{thr} \xspace(\beta)$ and $%
f_\delta(\beta,\lambda) < +\infty$ if
$\lambda > \lambda_{thr} \xspace(\beta).$ Therefore, it is
necessary that $f_\delta(\beta,\lambda_\ast(\beta))$ is finite and
$\lambda_\ast(\beta) \geq \lambda_{thr} \xspace(\beta).$
	
	\textbf{Case 1.} Suppose that $\lambda_\ast(\beta) > \lambda_{thr} \xspace%
	(\beta).$ In this case we have from Lemma \ref{Lem-eff-dom-f} that $%
	\Gamma^\ast(\beta,\lambda_\ast(\beta);x)$ is not empty and $\partial_+
	\ell_{rob} \xspace/\partial \lambda(\beta,\lambda_\ast(\beta);x)$ is given
	as in \eqref{eqn-directional-derivative-d},
	for $P_0-$almost every $x.$ Since $f_\delta(\cdot)$ is finite in the neighborhood
	of $\lambda = \lambda_\ast(\beta),$ it follows from \cite[Proposition 2.1]%
	{Bertsekas1973} that
	\begin{align}
	\frac{\partial_+f}{\partial \lambda} (\beta,\lambda_\ast(\beta)) = \sqrt{%
		\delta}- \sqrt{\delta} E_{P_0}\left[\beta^TA(X)^{-1}\beta \min_{\gamma \in
		\Gamma^\ast(\beta,\lambda_\ast(\beta);X)} \gamma^2 \right].
	\label{Inter-right-deriv-f}
	\end{align}
	
	\textbf{Case 2.} Suppose that $\lambda_\ast(\beta) = \lambda_{thr} \xspace%
	(\beta).$ We first argue that $\partial_+ f/\partial \lambda
	(\beta,\lambda_\ast(\beta);x) \in [0,\sqrt{\delta}].$ For this purpose,
	observe that
	\begin{align*}
	f_\delta(\beta,\lambda) = \lambda \sqrt{\delta} + E_{P_0}\left[ \sup_{\gamma \in
		\mathbb{R}} \left\{ \ell\left(\beta^TX + \gamma\sqrt{\delta}%
	\beta^TA(X)^{-1}\beta \right) -\lambda \sqrt{\delta}\gamma^2\beta^TA(X)^{-1}%
	\beta\right\}\right],
	\end{align*}
	as a consequence of the duality representation in Theorem \ref%
	{Thm-strong-duality}. Since the second term in the right hand side of the
	above equality is non-increasing in $\lambda$ and $\lambda_\ast(\beta)$ is a
	minimizer, we have that
	\begin{equation*}
	0 \leq f_\delta(\beta,\lambda_\ast(\beta) + h) - f_\delta(\beta,\lambda_\ast(\beta)) \leq
	\sqrt{\delta}h,
	\end{equation*}
	for $h > 0.$ Due to the convexity of $f,$ we also have that $%
	h^{-1}(f_\delta(\beta,\lambda_\ast(\beta) + h) - f_\delta(\beta,\lambda_\ast(\beta)))$ is
	non-decreasing in $h.$ Therefore the right derivative $\partial_+ f/\partial
	\lambda (\beta,\lambda_\ast(\beta)) \in [0,\sqrt{\delta}].$ As a result, due
	to the convexity of $f_\delta(\beta,\cdot)$ and finiteness of $f_\delta(\beta,\lambda)$
	for any $\lambda > \lambda_\ast(\beta),$ we have from \cite[Proposition 2.1]{Bertsekas1973} and Proposition \ref{prop:firstorderinfo}b that
	\begin{align}
	0 \leq \frac{\partial_+ f}{\partial \lambda}(\beta,\lambda_\ast(\beta)) \leq
	\lim_{\lambda \downarrow \lambda_\ast(\beta)} \frac{\partial_{-} f}{\partial
		\lambda} (\beta,\lambda) \leq \sqrt{\delta}\left(1- \lim_{\lambda \downarrow
		\lambda_\ast(\beta)} E_{P_0}\left[ \beta^TA(X)^{-1}\beta g_{_\lambda}(X)^2%
	\right] \right),  \label{Inter-Case2-eq-Thr}
	\end{align}
	where $g_{_\lambda}(x)$ is such that $g_{_\lambda}(x) \in
	\Gamma^\ast(\beta,\lambda;x),$ $P_0-$almost every $x.$ The existence of
	measurable maps $\{g_\lambda(\cdot): \lambda > \lambda_\ast(\beta)\}$ follow
	from Proposition 7.50(b) of \cite{bertsekas1978stochastic}.
	
	For the chosen $\beta \in B,$  define the set
$A := \{x \in \mathbb{R}^d:
\Gamma^\ast(\beta,\lambda_\ast(\beta);x) = \varnothing\}.$
Take any $x \in A.$ For any sequence $\{g_{_\lambda}(x): \lambda >
	\lambda_{thr} \xspace(\beta)\}$ such that $g_{_\lambda}(x) \in
	\Gamma^\ast(\beta,\lambda;x),$ we next show that $\lim_{\lambda\downarrow%
\lambda_{\ast}(\beta)}g_{_\lambda}^2(x) = +\infty.$ If otherwise, there
	exist a real number $g_{_0}$ and a decreasing sequence $\{\lambda_{n}:n\in%
	\mathbb{N}\}$ satisfying $\lim_{n\rightarrow\infty} \lambda_n =
	\lambda_{\ast}(\beta)$ and $\lim_{n\rightarrow\infty}g_{_{\lambda_n}}(x) =
	g_{_0}.$ Since $\ell_{rob} \xspace(\beta,\lambda;x)$ is right-continuous at $%
	\lambda = \lambda_{thr} \xspace(\beta)$ when $f_\delta(\beta,\lambda_{thr} \xspace%
	(\beta)) < \infty$ (see Lemma \ref{Lem-Right-Cont-lambda}), we have that
	\begin{align}
	\ell_{rob} \xspace(\beta,\lambda_\ast(\beta);x) =
	\lim_{n\rightarrow\infty}\ell_{rob} \xspace(\beta,\lambda_n;x) =
	\lim_{n\rightarrow\infty}F(g_{_{\lambda_n}}(x), \beta,\lambda_n;x) =
	F(g_{_0}, \beta,\lambda_\ast(\beta);x),  \label{Eqn-f-j}
	\end{align}
	where the last equality holds because $F(\gamma, \beta,\lambda;x)$ is a
	continuous function in $(\gamma,\beta,\lambda)$. However, it follows from %
	\eqref{Eqn-f-j} that $g_{_0}\in \Gamma(\beta,\lambda_{\ast}(\beta);x)$,
	which contradicts that $x \in A$ as $\Gamma(\beta,\lambda_{\ast}(\beta);x)$
	is not an empty set if $\limsup_{\lambda \downarrow \lambda_\ast(\beta)}
	g^2_\lambda(x) < \infty.$ Therefore $\lim_{\lambda\downarrow\lambda_{\ast}(%
		\beta)}g_{_\lambda}^2(x) = +\infty$ for $x \in A.$
	
	Applying Fatou's lemma to the right hand side of \eqref{Inter-Case2-eq-Thr},
	we obtain from \eqref{Inter-Case2-eq-Thr} that $E_{P_0}[\beta^TA(X)^{-1}%
	\beta \liminf_{\lambda \downarrow \lambda_\ast(\beta)}g_{_\lambda}^2(X)]
	\leq 1.$ Since $\liminf_{\lambda \downarrow
		\lambda_\ast(\beta)}g_{_\lambda}^2(x) = +\infty$ for $x \in A,$ this
	inequality results in $\infty \times P_0(X \in A) \leq 1.$ Therefore $P_0(X
	\in A) = 0.$ In other words, the set of maximizers $\Gamma^\ast(\beta,%
	\lambda_\ast(\beta);x)$ is not empty, for $P_0-$almost every $x.$
	
	Consequently, an application of envelope theorem (see \cite[Corollary 4]%
	{milgrom2002envelope}) similar to that in Proposition \ref{prop:firstorderinfo}b  results in $\partial_+ \ell_{rob} \xspace/\partial \lambda
	(\beta,\lambda_\ast(\beta);x)= \sqrt{\delta}(1-\beta^TA(x)^{-1}\beta
	\min_{\gamma \in \Gamma^\ast(\beta,\lambda_\ast(\beta);x)}\gamma^2),$ for $x
	\in A.$ Since $\ell_{rob} \xspace(\beta,\lambda_\ast(\beta);x)$ is convex in
	$\lambda_\ast(\beta)$ for every $x$ (see Lemma \ref{Lem-conv-lrob}), we have
	$h^{-1}(\ell_{rob} \xspace(\beta,\lambda_\ast(\beta) + h;x) - \ell_{rob} %
	\xspace(\beta,\lambda_\ast(\beta);x))$ is non-decreasing in $h$ for $h \geq
	0, $ and the limit as $h \rightarrow 0$ is given by $\partial_+\ell_{rob} %
	\xspace (\beta,\lambda_\ast(\beta) ;x)$ for $x \in A.$ With $P_0(X \in A) =
	1 $, due to monotone convergence theorem, it follows that $\partial_+
	f/\partial \lambda (\beta,\lambda_\ast(\beta)) = E_{P_0}[\partial_+
	\ell_{rob} \xspace/\partial \lambda (\beta,\lambda_\ast(\beta);X)],$ thus
	resulting in \eqref{Inter-right-deriv-f}. This completes the proof of Lemma %
	\ref{Lem-threshold}.
\hfill \Halmos
\endproof
\
\proof{\textbf{Proof of Lemma \protect\ref{Lemma-Gamma}.}}
Observe that $\Gamma^{\ast}(\beta,\lambda;x) \neq\varnothing$
implies $\lambda \geq \kappa\sqrt{\delta}\beta^TA(x)^{-1}\beta$ (see Lemma %
\ref{Lem-eff-dom-f}b). With $F(\cdot)$ being defined as in \eqref{Univar-Opt}%
, any $g \in \Gamma^\ast(\beta,\lambda;x)$ must satisfy the first order
optimality condition that,
\begin{align}  \label{Opt-Gamma}
2\lambda g = \ell^{\prime}(\beta^T x + g\sqrt{\delta}\beta^TA(x)^{-1} \beta).
\end{align}
This verifies the first part of the statement of  Lemma \ref{Lemma-Gamma}.
To prove the inequality in (\ref{Lower-bound-g-i}), we proceed by considering the following cases depending on the signs of $\ell^\prime(\beta^Tx)$ and $g.$
If $\ell^{\prime}(\beta^Tx) = 0,$ inequality \eqref{Lower-bound-g-i} is
trivial. Thus, in order to prove \eqref{Lower-bound-g-i}, it suffices to
consider the case where $\ell^{\prime}(\beta^Tx)$ is strictly positive or
strictly negative.
As $\ell^{\prime}(\beta^Tx)\neq 0$, we have $g \neq 0$. Therefore it is
sufficient to establish \eqref{Lower-bound-g-i} by considering cases where $%
\ell^\prime(\beta^Tx),$ $g$ are strictly positive or negative.
	
	\textbf{Case 1} - Suppose that $\ell^{\prime}(\beta^Tx)>0$ and $g > 0.$
	Since the convexity of $\ell(\cdot)$ in Assumption \ref{Assump-Obj-Conv}
	ensures that $\ell^{\prime}(\cdot)$ is non-decreasing, we have
	$2\lambda g = \ell^{\prime}(\beta^T x + g\sqrt{\delta}\beta^TA(x)^{-1}
	\beta)\geq\ell^{\prime}(\beta^T x),$ 
	due to \eqref{Opt-Gamma}; equivalently, $g \geq \ell^{\prime}(\beta^T
	x)/(2\lambda).$ This verifies \eqref{Lower-bound-g-i} when both $%
	\ell^{\prime}(\beta^Tx)$ and $g$ are positive.
	
	\textbf{Case 2} - Suppose that $\ell^{\prime}(\beta^Tx)>0$ and $g < 0.$ Due
	to convexity of $\ell(\cdot)$, and optimality of $g,$
	\begin{align}
	F(g, \beta,\lambda;x) &\geq \sup_{\gamma\geq 0}\left\{ \ell(\beta^Tx) +
	\ell^{\prime}(\beta^Tx) \sqrt{\delta}\beta^TA(x)^{-1}\beta\gamma-\lambda
	\sqrt{\delta} \big(\gamma^2 \beta^TA(x)^{-1}\beta - 1\big) \right\}  \notag
	\\
	&= \ell(\beta^Tx)+ \lambda \sqrt{\delta} + \sqrt{\delta}\beta^TA(x)^{-1}%
	\beta \frac{[\ell^{\prime}(\beta^Tx)]^2}{4\lambda}.  \label{Inter-F-LB}
	\end{align}
	An application of the fundamental theorem of calculus to the terms $%
	\ell(\beta^Tx + \sqrt{\delta}g\beta^TA(x)^{-1}\beta)$ and $g^2$ in the
	definition of $F(g,\beta,\lambda;x)$ (see \eqref{Univar-Opt}) allows us to
	rewrite the left hand side as,
	\begin{align*}
	F(g,\beta,\lambda;x) &= \ell(\beta^Tx) + \lambda\sqrt{\delta} + \sqrt{\delta}%
	\beta^{T}A(x)^{-1}\beta \int_{g}^{0} \Big(2\lambda \gamma -
	\ell^{\prime}(\beta^T x + \gamma\sqrt{\delta}\beta^TA(x)^{-1} \beta)\Big)%
	d\gamma.
	\end{align*}
	For any $\gamma$ in $(g,0),$ we have from the monotonicity of $%
	\ell^\prime(\cdot)$ that $2\lambda \gamma - \ell^{\prime}(\beta^T x + \gamma%
	\sqrt{\delta}\beta^TA(x)^{-1}\beta)$ does not exceed the positive part of $%
	2\lambda \gamma-\ell^{\prime}(\beta^T x + g\sqrt{\delta}\beta^TA(x)^{-1}
	\beta).$ Consequently, it follows from the optimality condition in %
	\eqref{Opt-Gamma} that,
	\begin{align*}
	F(g,\beta,\lambda;x) &\leq \ell(\beta^Tx) + \lambda\sqrt{\delta} + \sqrt{%
		\delta}\beta^{T}A(x)^{-1}\beta \int_{g}^{0} \Big(2\lambda \gamma-2\lambda g%
	\Big)_{+}d\gamma \\
	&= \ell(\beta^Tx)+ \lambda \sqrt{\delta} + \sqrt{\delta}\lambda%
	\beta^TA(x)^{-1}\beta g^2.
	\end{align*}
	Combining this observation with that in \eqref{Inter-F-LB}, we obtain $\vert
	g \vert\geq \ell^{\prime}(\beta^Tx)/(2\lambda).$
	
	When $\ell^{\prime}(\beta^Tx)<0$, \eqref{Lower-bound-g-i} follows by an
	argument symmetric to that of the $\ell^\prime(\beta^Tx) > 0$ cases
	described above. This completes the proof of Lemma \ref{Lemma-Gamma}.
\hfill \Halmos
\endproof
\
\proof{\textbf{Proof of Lemma \protect\ref{Lem-Lbars}.}}  Define the
function $L: B \rightarrow \mathbb{R}_+$ as
$L(\beta) =
E_{P_0}[\ell^\prime(\beta^TX)^2]^{1/2}.$
With $\ell^\prime(\beta^TX) \neq 0$ almost surely, we have that
$L(\beta) > 0$ for any $\beta \in B.$ Moreover, we have that
$L(\beta)$ is continuous in $\beta$ due to the continuity of
$\ell^\prime(\cdot).$ Then the existence of finite
$\bar{L},\underline{L},$  as  in the statement of Lemma \ref{Lem-Lbars},
follows immediately from the fact that
continuous functions attain their extrema over compact sets.\hfill
\Halmos
\endproof
\
\proof{\textbf{Proof of Proposition \protect\ref{prop:hessian-lrob}.}}
With $\varphi(g,\beta,\lambda;x) =  \varphi_g(\beta,\lambda;x)  > 0,$ we have,
\begin{equation*}
\frac{\partial^2 F}{\partial \gamma} (g,\beta,\lambda;x) = -\sqrt{\delta}
\beta^TA(x)^{-1}\beta \varphi(g,\beta,\lambda;x) < 0.
\end{equation*}
Moreover, we have that
$g(\beta,\lambda;x) \in \Gamma^\ast(\beta,\lambda;x)$
satisfies the first order optimality condition that
\begin{align}
\ell^\prime(\beta^Tx + \sqrt{\delta} g(\beta,\lambda;x)
\beta^TA(x)^{-1}\beta) - 2\lambda g(\beta,\lambda;x) = 0.
\label{Eq-First-Order-Opt-Gamma}
\end{align}
Using implicit function theorem, the partial derivatives of
$g(\beta,\lambda;x)$ are given as follows: 
\begin{align}
\frac{\partial g}{\partial \beta}(\beta,\lambda;x)
&= -\frac{\partial^2
F/\partial \beta\partial \gamma (g(\beta,\lambda;x),\beta,\lambda;x)}
{\partial^2 F/\partial \gamma^2(g(\beta,\lambda;x),\beta,\lambda;x)}
= \frac{\ell^{\prime\prime}(\beta^TT_g(x))}{\varphi_g(\beta,\lambda;x)} \bar{T}_g(x)
\label{partial-g-beta} \\
\frac{\partial g }{\partial \lambda}(\beta,\lambda;x)
&= -\frac{\partial^2F/\partial \lambda\partial \gamma(g(\beta,\lambda;x),\beta,\lambda;x)}
{\partial^2 F/\partial \gamma^2(g(\beta,\lambda;x),\beta,\lambda;x)}
= -\frac{2g(\beta,\lambda;x)}{\varphi_g(\beta,\lambda;x)}.
\label{partial-g-lambda}
\end{align}
Following these expressions for the gradient of $g,$ the Hessian of
$\ell_{rob}(\,\cdot\ ;x)$ in the statement of Proposition
\ref{prop:hessian-lrob} follows from the first order derivative
information in Proposition \ref{prop:firstorderinfo} and elementary
rules of differentiation.
Next, to establish (\ref{PDness-2}), we first provide an equivalent
characterization of the relationship
$\nabla^2f_{\delta}(\beta,\lambda;x)-\Lambda(x) B(x)\succeq 0.$ For simplicity,
we re-scale $\Lambda(x)$ and pick a new parameter $m$ such that
$\Lambda(x) = m%
\sqrt{\delta}g^2.$ To avoid clutter in expressions, we write
$\tilde{x}:= T_g(x)$ and $\bar{x} := \bar{T}_g(x)$ throughout this
proof. The matrix
$\nabla^{2}\ell_{rob}(\beta,\lambda;x)-m\sqrt{\delta}%
g^{2}B(x)$ can be written as a block matrix, namely,
	\begin{equation*}
	\nabla^{2}\ell_{rob}(\beta,\lambda;x)-m\sqrt{\delta}g^{2}B(x) =
	\begin{bmatrix}
	\left(2\lambda-m\right)\sqrt{\delta}g^2A(x)^{-1}+\frac{(2\lambda-m)\ell^{%
			\prime\prime}(\beta^T \tilde{x})}{\varphi}\bar{x}\bar{x}^{T} & -2\sqrt{\delta%
	}g^2z \\
	-2\sqrt{\delta}g^2z^{T} & \frac{4\sqrt{\delta} g^2 \beta^TA(x)^{-1}\beta}{%
		\varphi}-m\sqrt{\delta}g^2%
	\end{bmatrix}%
	,
	\end{equation*}
	where $z := A(x)^{-1}\beta + \frac{\beta^T A(x)^{-1}\beta}{\psi}\bar{x}$ and
	$\psi:=g\varphi/\ell^{\prime\prime}(\beta^T\tilde{x}).$ According to Schur
	complement condition, the matrix $\nabla^{2}\ell_{rob}(\beta,\lambda;x)-m\sqrt{\delta}%
	g^{2}B(x)$ is positive definite if and only if $\left(2\lambda-m\right)\sqrt{%
		\delta}g^2A(x)^{-1}+\frac{(2\lambda-m)\ell^{\prime\prime}(\beta^T \tilde{x})%
	}{\varphi}\bar{x}\bar{x}^{T}$ is positive definite and
	\begin{align}  \label{Condition-Lemma-3}
	\frac{4\sqrt{\delta} g^2 \beta^TA(x)^{-1}\beta}{\varphi}-m\sqrt{\delta}g^2
	>4\delta g^4 z^{T}\left( \left(2\lambda-m\right)\sqrt{\delta}g^2A(x)^{-1}+%
	\frac{(2\lambda-m)\ell^{\prime\prime}(\beta^T \tilde{x})}{\varphi}\bar{x}%
	\bar{x}^{T}\right)^{-1}z.
	\end{align}
	Recalling from the assumptions that $m\in(0, 2\lambda)$ and $\ell(\cdot)$ is
	convex, the positive definiteness of $\left(2\lambda-m\right)\sqrt{\delta}%
	g^2A(x)^{-1}+\frac{(2\lambda-m)\ell^{\prime\prime}(\beta^T \tilde{x})}{%
		\varphi}\bar{x}\bar{x}^{T}$ is automatically satisfied. Then, applying
	Sharman-Morrison formula, one can show that
	\begin{align}
	\left(\left(2\lambda-m\right)\sqrt{\delta}g^2A(x)^{-1}+\frac{%
		(2\lambda-m)\ell^{\prime\prime}(\beta^T \tilde{x})}{\varphi}\bar{x}\bar{x}%
	^{T}\right)^{-1} = \frac{1}{(2\lambda-m)\sqrt{\delta}g^2}C,
	\label{eq-Sharman-Morrison}
	\end{align}
	where $C$ is a matrix defined as
	\begin{equation*}
	C := A(x) - \frac{A(x) \bar{x} \bar{x}^T A(x)}{\bar{x}^T A(x) \bar{x} +
		\sqrt{\delta} g \psi}
	\end{equation*}
	Thus, combining equation \eqref{Condition-Lemma-3} and %
	\eqref{eq-Sharman-Morrison}, if $(\beta,\lambda) \in
\mathbb{V}$ and $m\in (0,2\lambda)$, then the matrix
$\nabla^{2}\ell_{rob}(\beta,\lambda;x)-m\sqrt{\delta}%
	g^{2}B(x)$ if and only if
	\begin{align*}
	\left( 2\lambda - m\right) \left(\frac{4\beta^TA(x)^{-1}\beta}{\varphi} - m
	\right) > 4z^T Cz.
	\end{align*}
	Let $a,b$ and $\theta$ be constant defined as
	\begin{align}
	a := 2\lambda, \quad b := \frac{4\beta^T A(x)^{-1} \beta}{\varphi}, \quad
	\theta := 4z^T C z.  \label{abtheta-constants}
	\end{align}
	If $\theta \in (0, ab),$ then we have $m := (ab - \theta)/(a + b)$
	satisfying $m \in (0, a \wedge b)$ and $(a-m)(b-m) > \theta.$ So it follows
	that
	\begin{align}  \label{PD-ness}
	\nabla^{2}\ell_{rob}(\beta,\lambda;x)-m\sqrt{\delta}g^{2}B(x)\succeq 0
	\end{align}
	for any $(\beta,\lambda) \in \mathbb{V}.$
	
	The rest of this proof is devoted to arguing that $\theta \in (0,ab),$ and
	to obtain a simplified lower bound for $m=(ab-\theta )/(a+b).$ We accomplish
	this by claiming that,
	\begin{equation}
	ab-\theta \geq \frac{4(\beta ^{T}\tilde{x})^{2}}{\bar{x}^{T}A(x)\bar{x}+%
		\sqrt{\delta }g^{2}\varphi /\ell ^{\prime \prime }(\beta ^{T}\tilde{x})}.
	\label{eq-ab-theta}
	\end{equation}%
	To show \eqref{eq-ab-theta}, first we derive an alternative expression of $%
	\theta $. It follows from the definition of $z$ and $C$ that
	\begin{equation*}
	\frac{\theta }{4}=\left( A(x)^{-1}\beta +\frac{\beta ^{T}A(x)^{-1}\beta }{%
		\psi }\bar{x}\right) ^{T}\left( A(x)-\frac{A(x)\bar{x}\bar{x}^{T}A(x)}{\bar{x%
		}A(x)\bar{x}+\sqrt{\delta }g\psi }\right) \left( A(x)^{-1}\beta +\frac{\beta
		^{T}A(x)^{-1}\beta }{\psi }\bar{x}\right)
	\end{equation*}%
	On expanding the bracket,
	\begin{align*}
	\frac{\theta }{4}& =\beta ^{T}A(x)^{-1}\beta +\left( \frac{\beta
		^{T}A(x)^{-1}\beta }{\psi }\right) ^{2}\bar{x}^{T}A(x)\bar{x}+2\frac{\beta
		^{T}A(x)^{-1}\beta }{\psi }\beta ^{T}\bar{x} \\
	& \quad -\frac{(\beta ^{T}\bar{x})^{2}+\left( \frac{\beta ^{T}A(x)^{-1}\beta
		}{\psi }\right) ^{2}(\bar{x}A(x)\bar{x})^{2}+2\frac{\beta ^{T}A(x)\beta }{%
			\psi }\bar{x}^{T}A(x)\bar{x}(\beta ^{T}\bar{x})}{\bar{x}A(x)\bar{x}+\sqrt{%
			\delta }g\psi },
	\end{align*}%
	which further implies
	\begin{align*}
	\frac{\theta }{4}& =\beta ^{T}A(x)^{-1}\beta -\frac{(\beta ^{T}\bar{x})^{2}}{%
		\bar{x}^{T}A(x)\bar{x}}+\frac{(\beta ^{T}\bar{x})^{2}}{\bar{x}^{T}A(x)\bar{x}%
	}+\left( \frac{\beta ^{T}A(x)^{-1}\beta }{\psi }\right) ^{2}\bar{x}^{T}A(x)%
	\bar{x}+2\frac{\beta ^{T}A(x)^{-1}\beta }{\psi }\beta ^{T}\bar{x} \\
	& \quad -\frac{(\beta ^{T}\bar{x})^{2}+\left( \frac{\beta ^{T}A(x)^{-1}\beta
		}{\psi }\right) ^{2}(\bar{x}A(x)\bar{x})^{2}+2\frac{\beta ^{T}A(x)\beta }{%
			\psi }\bar{x}^{T}A(x)\bar{x}(\beta ^{T}\bar{x})}{\bar{x}A(x)\bar{x}+\sqrt{%
			\delta }g\psi } \\
	& =\beta ^{T}A(x)^{-1}\beta -\frac{(\beta ^{T}\bar{x})^{2}}{\bar{x}^{T}A(x)%
		\bar{x}}+\frac{\left( \beta ^{T}A(x)^{-1}\beta \frac{\sqrt{\bar{x}^{T}A(x)%
				\bar{x}}}{\psi }+\frac{\beta ^{T}\bar{x}}{\sqrt{\bar{x}^{T}A(x)\bar{x}}}%
		\right) ^{2}(\sqrt{\delta }g\psi )}{\bar{x}A(x)\bar{x}+\sqrt{\delta }g\psi }
	\\
	& =\beta ^{T}A(x)^{-1}\beta -\frac{(\beta ^{T}\bar{x})^{2}}{\bar{x}^{T}A(x)%
		\bar{x}}+\frac{\left( \beta ^{T}A(x)^{-1}\beta \frac{\sqrt{\bar{x}^{T}A(x)%
				\bar{x}}}{\psi }+\frac{\beta ^{T}\bar{x}}{\sqrt{\bar{x}^{T}A(x)\bar{x}}}%
		\right) ^{2}}{1+\frac{\bar{x}^{T}A(x)\bar{x}}{\sqrt{\delta }g\psi }}.
	\end{align*}%
	Then, using above upper bound for $\theta $ and the definition of $a$ and $b$%
	, we obtain,
	\begin{align*}
	\frac{ab-\theta }{4}\left( 1+\frac{\bar{x}^{T}A(x)\bar{x}}{\sqrt{\delta }%
		g\psi }\right) & \geq \left[ \left( \frac{2\lambda }{\varphi }-1\right)
	\beta ^{T}A(x)^{-1}\beta +\frac{(\beta ^{T}\bar{x})^{2}}{\bar{x}^{T}A(x)\bar{%
			x}}\right] \left( 1+\frac{\bar{x}^{T}A(x)\bar{x}}{\sqrt{\delta }g\psi }%
	\right) \\
	& \quad \quad \quad \quad -\left( \frac{\beta ^{T}A(x)^{-1}\beta }{\psi }%
	\sqrt{\bar{x}^{T}A(x)\bar{x}}+\frac{\beta ^{T}\bar{x}}{\sqrt{\bar{x}^{T}A(x)%
			\bar{x}}}\right) ^{2}.
	\end{align*}%
	Since $2\lambda -\varphi =\sqrt{\delta }(\beta ^{T}A(x)^{-1}\beta )\ell
	^{\prime \prime }(\beta ^{T}\tilde{x}),$ $\psi =g\varphi /\ell ^{\prime
		\prime }(\beta ^{T}\tilde{x})$ and $\bar{x}=\tilde{x}+\sqrt{\delta }%
	gA(x)^{-1}\beta ,$ on expanding the squares in the last term, the above
	inequality simplifies to,
	\begin{equation*}
	\frac{ab-\theta }{4}\left( 1+\frac{\bar{x}^{T}A(x)\bar{x}}{\sqrt{\delta }%
		g\psi }\right) \geq \frac{\ell ^{\prime \prime }(\beta ^{T}\tilde{x})}{\sqrt{%
			\delta }g^{2}\varphi }\left( \beta ^{T}\bar{x}-\sqrt{\delta }g\beta
	^{T}A(x)^{-1}\beta \right) ^{2}=\frac{\ell ^{\prime \prime }(\beta ^{T}%
		\tilde{x})}{\sqrt{\delta }g^{2}\varphi }\left( \beta ^{T}\tilde{x}\right)
	^{2}.
	\end{equation*}%
	This establishes \eqref{eq-ab-theta}. Finally, combining \eqref{PD-ness} and %
	\eqref{eq-ab-theta}, we have
	\begin{equation*}
	\nabla ^{2}\ell_{rob}(\beta ,\lambda ;x)-\frac{4\left( \beta ^{T}\tilde{x}\right)
		^{2}\ell ^{\prime \prime }(\beta ^{T}\tilde{x})}{1+\bar{x}^{T}A(x)\bar{x}%
		\ell ^{\prime \prime }(\beta ^{T}\tilde{x})/(\sqrt{\delta }g^{2}\varphi )}%
	\frac{1}{2\lambda \varphi +4\beta ^{T}A(x)^{-1}\beta }B(x)\succeq 0,
	\end{equation*}%
	which is obtained by plugging in the definitions of $a,b$ from %
	\eqref{abtheta-constants}.
\hfill \Halmos \endproof
\proof{\textbf{Proof of Lemma \protect\ref{Lem-Str-conv-Thr}.}}
Since $\lambda > \lambda_{thr}^\prime(\beta),$ there exist
$\varepsilon > 0$ such that
$\lambda \geq (M/2 + \varepsilon)\sqrt{\delta}\beta^TA(x)^{-1}\beta,$
for $P_0-$almost every $x.$ Since $\ell(\cdot)$ is twice
differentiable and $\ell^{\prime\prime}(\cdot) \leq M$ (see Assumption
\ref{Assump-Obj-2diff}), it follows from the definition of $F(\cdot)$
in \eqref{Univar-Opt} that
\begin{align}
\frac{\partial^2 F}{\partial \gamma^2}(\gamma,\beta,\lambda;x)
&= \sqrt{\delta}\beta^TA(x)^{-1}\beta
\left( \ell^{\prime\prime}(\beta^Tx + \sqrt{\delta}\gamma\beta^TA(x)^{-1}%
\beta) \sqrt{\delta}\beta^TA(x)^{-1}\beta - 2\lambda\right)
\label{DDeriv-F}\\
&\leq \sqrt{\delta}\beta^TA(x)^{-1}\beta \left( M \sqrt{\delta}%
\beta^TA(x)^{-1}\beta - 2\lambda\right) \leq
-2\varepsilon\delta\beta^TA(x)^{-1}\beta,  \notag
\end{align}
for $P_0-$almost every $x.$ Thus the map
$\gamma \mapsto F(\gamma,\beta,\lambda;x)$ is strongly concave for
every $\beta \in B, \lambda > \lambda_{thr}^\prime \xspace(\beta)$ and
$P_0-$almost every $x \in \mathbb{R}^d,$ and attains maximum at a
unique point $g(\beta,\lambda;x)$. In such case, the set
$\Gamma^\ast(\beta,\lambda;x) = \{g(\beta,\lambda;x)\}$ is
singleton.
Moreover, for any
$\beta \in B, \lambda > \lambda_{thr}^\prime(\beta),$ we have,
\begin{align*}
\varphi(g(\beta,\lambda;x),\beta,\lambda;x)  = -\frac{\partial^2 F}{\partial \gamma^2}(g(\beta,\lambda;x),\beta,\lambda;x) < 0,
\end{align*}
for $P_0-$almost every $x.$ Thus, every element of
$\{(\beta,\lambda):\beta \in B, \lambda >
\lambda_{thr}^\prime(\beta)\}$ lies in the set $ \mathcal{U}(x).$
\hfill \Halmos \endproof
\
\proof{\textbf{Proof of Proposition
\protect\ref{prop:secondorder-info-smoothcase}.}}  a) The
inclusion that $\mathbb{V} \subseteq \mathbb{W}$ is immediate from
their respective definitions. To verify the second inclusion, see that
for any $\beta \in B,$
\begin{align*}
\lambda_{thr}^\prime \xspace(\beta) \leq 2^{-1}\sqrt{\delta}%
M\rho_{\min}^{-1}\Vert \beta \Vert^2 < 2^{-1}\sqrt{\delta_0}%
M\rho_{\min}^{-1}R_\beta\Vert \beta \Vert \leq 2^{-1}(\underline{L}%
\rho_{\max}^{-1})^{1/2}\Vert \beta \Vert = K_1\Vert \beta \Vert,
\end{align*}
due to Assumption \ref{Assump-c}b. Since we have that
$\lambda_{thr}^\prime(\beta) < K_1\Vert \beta \Vert,$ any
$(\beta,\lambda) \in \mathbb{W}$ is also an element of the set
$\{(\beta,\lambda): \beta \in B, \lambda >
\lambda_{thr}^\prime(\beta)\}.$ The final inclusion in the statement
of Proposition \ref{prop:secondorder-info-smoothcase}a follows from
Lemma \ref{Lem-Str-conv-Thr}.
b) For every $(\beta,\lambda) \in \mathbb{W},$ we have
$\lambda > \lambda_{thr}^\prime(\beta).$ Then it is immediate from
Lemma \ref{Lem-Str-conv-Thr} that $\Gamma^\ast(\beta,\lambda;x)$ is
singleton for $(\beta,\lambda) \in \mathbb{W}$ and $P_0-$almost every
$x.$ As a result, any measurable selection $g(\cdot)$ satisfying
(\ref{meas-sel-twice-diff}) is uniquely specified for almost every
$(\beta,\lambda,x)$ in the subset
$\mathbb{W} \times S_X \subseteq \mathcal{U}.$ Moreover, due to mean
value theorem, the first order optimality condition
\eqref{Eq-First-Order-Opt-Gamma} means that
$g(\beta,\lambda;x) = \ell^\prime(\beta^Tx)/(2\lambda - \sqrt{\delta}\beta^T
A(x)^{-1}\beta \ell^{\prime \prime}(\eta)),$
for some $\eta$ between the real numbers $\beta^Tx$ and
$\beta^T\tilde{x}.$ Since $\ell^{\prime\prime}(\cdot) \leq M$ and
$\delta \leq \delta_0,$ for $(\beta,\lambda)\in\mathbb{W}$ we have that
\[2\lambda - \sqrt{\delta} \beta^TA(x)^{-1}\beta
\ell^{\prime\prime}(\eta) \ \geq\ 2\lambda_{\min}(\beta) -
\sqrt{\delta} \beta^TA(x)^{-1}\beta \ell^{\prime\prime}(\eta) \ \geq
\ \varphi_{\min}\Vert\beta \Vert.\] Also note that if $(\beta,\lambda)\in\mathbb{V}$ we have 
\[
2\lambda - \sqrt{\delta}\beta^T
A(x)^{-1}\beta \ell^{\prime \prime}(\eta) \leq 2\lambda \leq
2K_2\Vert\beta\Vert
\]
Then the conclusion in
Proposition \ref{prop:secondorder-info-smoothcase}b follows.
c) Since
$\vert \ell^\prime(\beta^TX) - \ell^\prime(0) \vert \leq M\Vert \beta
\Vert \Vert X \Vert$ (due to Assumption \ref{Assump-Obj-2diff}),
$E_{P_0}\Vert X \Vert^4$ is finite and
$\Vert \beta \Vert \leq R_\beta$ (see Assumptions
\ref{Assump-Obj-Conv} and \ref{Assump-Compact}), we have from the
bounds in \eqref{gamma-UB-LB} that
$\sup_{(\beta,\lambda) \in \mathbb{V}} E_{P_0}\left[g^4(\beta)\right]
< \infty.$ Therefore, the collection
$\{g^2(\beta,\lambda;X):(\beta,\lambda) \in \mathbb{V}\}$ is
$L_2-$bounded. Then it is immediate from the definitions
$T_g(x) := x+\sqrt{\delta}g(\beta,\lambda;x)A(x)^{-1}\beta,$
$\bar{T}_g(\cdot) := x+2\sqrt{\delta}g(\beta,\lambda;x)A(x)^{-1}\beta$ and
Cauchy-Schwarz inequality that the collections
$\{(T_g(X))^2, (\bar{T}_g(X))^2: (\beta,\lambda) \in \mathbb{V}\}$ are
$L_2-$bounded. Consequently, the collections
$\{\ell(\beta^TT_g(X)), \ell^\prime(\beta^TX_g)^2: (\beta,\lambda) \in
\mathbb{V}\}$ are $L_2-$bounded as well due to the at most quadratic
growth property of $\ell(\cdot)$ (see Assumption
\ref{Assump-Obj-Conv}).
d) Recall from Part a) that $\mathbb{V}$ is strictly contained in
$\mathcal{U}(x),$ for $P_0-$almost every $x.$ With the DRO objective
$\sup_{P:D_c(P,P_n) \leq \delta}E_{P}\left[\ell(\beta^TX)\right] =
E_{P_0}[\ell_{rob}(\beta,\lambda)] =: f_{\delta}(\beta,\lambda)$
defined in terms of the convex loss $\ell(\cdot)$ specified over the
entire real line, the second order partial derivative expressions of
$\ell_{rob}(\cdot)$ in the statement of Proposition \ref{prop:hessian-lrob}
hold throughout the set $\mathbb{V}.$ It follows from the
$L_2-$boundedness just established in Part c) and these partial
derivative expressions that the norms of the individual entries of the
Hessian matrix $\nabla_\theta^2\ell_{rob}(\theta;X)$ are all bounded
in $L_2-$norm over the set $\theta \in \mathbb{V}.$ With this
$L_2-$boundedness of the collection
$\{\nabla_\theta^2\ell_{rob}(\theta;X):\theta \in \mathbb{V}\}$, the
desired exchange of derivative and expectation in
$\nabla_\theta^2f_\delta(\theta) = E_{P_0}[\nabla_\theta^2\ell_{rob}
\xspace(\theta;X)],$ for $\theta \in \mathbb{V},$ follows as a
consequence of dominated convergence theorem.  \hfill \Halmos \endproof

\proof{\textbf{Proof of Theorem \protect\ref{result-noncompact}.}} 

For any map $g:\mathcal{U} \rightarrow\mathbb{R}$ satisfying
(\ref{meas-sel-twice-diff}), consider $\varphi_{g}(\cdot)$ in
(\ref{Str-Conv-Pf-Defns-1}) and define the functions,
\begin{align*}
I_{0}(\beta,\lambda)  & := E \left[ g^{2}(\beta,\lambda;x) \beta^{T}%
A(X)^{-1}\beta\right] , \quad I_{1}(\beta,\lambda;x) := \vert g\beta^{T}%
T_{g}(x) \vert\frac{\varphi_{g}(\beta,\lambda;x)}{2\lambda},\\
I_{2}(\beta,\lambda;x)  & := \frac{\sqrt{\beta^{T}A(x)^{-1}\beta}}{2\lambda},
\qquad\text{ and } \qquad I_{3}(\beta,\lambda;x) := \sqrt{\delta} 2\lambda
g^{2} + \bar{T}_{g}^{T}A(x)\bar{T}_{g}(x) \ell^{\prime\prime}(\beta^{T}%
T_{g}(x)).
\end{align*}

It follows from the definition of $\varphi_{g}(\cdot)$ that $\varphi
_{g}/2\lambda\leq1.$ Then, for any $(\beta,\lambda,x) \in\mathcal{U}$ for
which the Hessian $\nabla^{2}\ell_{rob}(\beta,\lambda;x)$ (computed with
respect to variables $\beta,\lambda$) exists, we have from Proposition
\ref{prop:hessian-lrob} that $\nabla^{2}\ell_{rob}(\beta,\lambda;x) -
\Lambda(\beta,\lambda;x)B(x) \succeq0;$ here, $\Lambda(\beta,\lambda;x),$
defined as in (\ref{PD-coeff}), satisfies,
\begin{align}
\frac{2\lambda}{\sqrt{\delta}}\Lambda(\beta,\lambda;x) \geq\frac{I_{1}%
^{2}(\beta,\lambda;x) \ell^{\prime\prime}(\beta^{T}T_{g}(x))}{I_{3}%
(\beta,\lambda;x) (1 + 4I_{2}^{2}(\beta,\lambda;x))}%
.\label{LB-Lambda-noncompact}%
\end{align}
We also define,
\begin{align*}
\delta_{2} := c_{1}^{2}c_{2}^{2}p\rho_{\min}^{2} \left( 2k_{1}\rho_{\max
}^{1/2} + 4c_{1}\rho_{\min}^{1/2}(1+k_{2})\right) ^{-2}.
\end{align*}

Then, fix any
$\beta\in B$ and $\lambda_{\ast}(\beta) \in\arg\min_{\lambda\geq0} f_{\delta
}(\beta,\lambda).$  It follows from the first-order optimality condition
that,
\begin{align*}
0 \leq\frac{\partial_{+}f_{\delta}}{\partial\lambda}(\beta,\lambda_{\ast
}(\beta)) \leq\sqrt{\delta}\left( 1 - E_{P_{0}}\left[  g^{2}(\beta
,\lambda_{\ast}(\beta)) \beta^{T}A(X)^{-1}\beta\right] \right)
\end{align*}
(see the proof of Lemma \ref{Lemma-Lambda-Bound} in the earlier Subsection
\ref{Sec-Prep-results} for a similar application of the first order optimality
condition). Consequently, for a given $\rho> 0,$ there exists $r_{1} > 0$ such
that
\begin{align}
\left\vert E_{P_{0}}\left[  g^{2}(\tilde{\beta},\lambda;x)\tilde{\beta}%
^{T}A(X)^{-1}\tilde{\beta}\right]  - 1 \right\vert \leq\rho
\label{inter-opt-bound}%
\end{align}
for all $(\tilde{\beta},\lambda) \in\mathcal{N}_{r_{_{1}}}((\beta
,\lambda_{\ast}(\beta))).$ This follows from the continuity properties of
$\ell^{\prime}(\cdot).$ Since $\vert g(\tilde{\beta},\lambda;x)\vert\geq
\ell^{\prime}(\tilde{\beta}^{T}x)/(2\lambda),$ we have from
(\ref{inter-opt-bound}) and Assumption \ref{Assump-nondegeneracy} that
\[
2\lambda\geq\frac{E_{P_{0}}[\ell^{\prime}(\tilde{\beta}^{T}X)^{2}]^{1/2}
\rho_{\max}^{-1/2}\Vert\beta\Vert}{(1+\rho)^{1/2}} \geq c_{1}(p\rho_{\max
}^{-1}/2)^{1/2}\Vert\beta\Vert,
\]
for any choice of $\rho< 1$ and $(\tilde{\beta},\lambda) \in\mathcal{N}%
_{r_{_{1}}}((\beta,\lambda_{\ast}(\beta))).$ Consequently, we have that
\begin{align}
I_{2}(\tilde{\beta},\lambda;x) := \frac{\sqrt{\beta^{T}A(x)^{-1}\beta}%
}{2\lambda} \leq\frac{\rho_{\min}^{-1/2}\Vert\beta\Vert}{c_{1}(p\rho_{\max
}^{-1}/2)^{1/2}\Vert\beta\Vert} = \frac{1}{c_{1}} \frac{(2\rho_{\max})^{1/2}%
}{(p\rho_{\min})^{1/2}},\label{UB-I2}%
\end{align}
for any $(\tilde{\beta},\lambda) \in\mathcal{N}_{r_{_{1}}}((\beta
,\lambda_{\ast}(\beta))).$ Moreover, we have from the definition of
$\varphi_{g}(\cdot)$ that,
\begin{align*}
I_{1}(\tilde{\beta},\lambda;x) \geq\vert g\beta^{T}T_{g}(x) \vert- \frac
{\sqrt{\delta}}{2\lambda} \vert g \vert\beta^{T}A(x)^{-1}\beta\vert\beta
^{T}T_{g}(x) \vert\ell^{\prime\prime}(\beta^{T}T_{g}(x)).
\end{align*}
Since $\vert u \vert\ell^{\prime\prime}(u) \leq k_{1} + k_{2} \vert
\ell^{\prime}(u) \vert,$ for any $u \in\mathbb{R}$ (see the assumption in the
statement of Theorem \ref{result-noncompact}), we have that
\begin{align*}
I_{1}(\tilde{\beta},\lambda;x)  & \geq\vert g\tilde{\beta}^{T}T_{g}(x) \vert-
\frac{\sqrt{\delta}}{2\lambda} \vert g \vert\tilde{\beta}^{T}A(x)^{-1}%
\tilde{\beta} \left(  k_{1} + k_{2} \vert\ell^{\prime}(\tilde{\beta}^{T}%
T_{g}(x))\vert\right) \\
& = \vert g\tilde{\beta}^{T}T_{g}(x) \vert- \frac{\sqrt{\delta}}{2\lambda}
\vert g \vert\tilde{\beta}^{T}A(x)^{-1}\tilde{\beta} \left(  k_{1} + 2 \lambda
k_{2} \vert g \vert\right) \\
& \geq\vert g\tilde{\beta}^{T}x \vert- \sqrt{\delta} g^{2}\tilde{\beta}%
^{T}A(x)^{-1}\tilde{\beta} - \frac{\sqrt{\delta}}{2\lambda} \vert g
\vert\tilde{\beta}^{T}A(x)^{-1}\tilde{\beta} \left(  k_{1} + 2\lambda k_{2}
\vert g \vert\right) ,
\end{align*}
where the equality follows from the first-order optimality condition satisfied
by $g(\cdot),$ and the last inequality is a simple consequence of triangle
inequality applied to $\tilde{\beta}^{T}T_{g}(x) =\tilde{\beta}^{T}x +
\sqrt{\delta} g\tilde{\beta}^{T}A(x)^{-1}\tilde{\beta}.$ As a result,
\begin{align*}
\frac{I_{1}(\tilde{\beta},\lambda;x)}{\sqrt{g^{2}\tilde{\beta}^{T}%
A(x)^{-1}\tilde{\beta}}}  & \geq\frac{\beta^{T}x}{\sqrt{\beta^{T}%
A(x)^{-1}\beta}} - \sqrt{\delta} \left(  k_{1}I_{2}(\tilde{\beta},\lambda;x) +
k_{2} \sqrt{g^{2}\beta^{T}A(x)^{-1}\beta}\right) .
\end{align*}
By applying the upper bound for $I_{2}(\cdot)$ derived in \eqref{UB-I2}, we
arrive at,
\begin{align}
\frac{I_{1}(\tilde{\beta},\lambda;x)}{\sqrt{g^{2}\tilde{\beta}^{T}%
A(x)^{-1}\tilde{\beta}}}  & \geq\frac{\beta^{T}x}{\sqrt{\beta^{T}%
A(x)^{-1}\beta}} - \sqrt{\delta}\left(  \frac{k_{1}}{c_{1}}\frac{(2\rho_{\max
})^{1/2}}{(p\rho_{\min})^{1/2}} + k_{2} \sqrt{g^{2}\beta^{T}A(x)^{-1}\beta
}\right) \label{LB-I1}%
\end{align}

Next, as in the proof of Theorem \ref{Thm-Str-Convexity}, we define the
following subsets of $\mathbb{R}^{d}:$ \newline$A_{1} := \{x: \vert
\ell^{\prime}(\tilde{\beta}^{T}x) \vert\geq c_{1}/2, \vert\tilde{\beta}%
^{T}x\vert\geq c_{2} \Vert\tilde{\beta}\Vert/2 \text{ for all } \tilde{\beta}
\in\mathcal{N}_{r_{_{2}}}(\beta) \},$ $A_{2} := \{x: g^{2}(\tilde{\beta
},\lambda;x)\beta^{T}A(x)^{-1}\beta\leq C_{0} \text{ for all } (\tilde{\beta
},\lambda) \in\mathcal{N}_{r_{_{1}}}(\beta,\lambda_{\ast}(\beta)) \},$ and
$A_{3} := \{x: \Vert x \Vert\leq C_{1}\},$ where the constants $C_{0}%
,C_{1},r_{2},\rho$ are to be chosen imminently. Define $\underline{r} = r_{1}
\wedge r_{2}$ and take the constant $C_{1}$ large enough such that
$P_{0}(A_{3}) \geq p/4,$ where $p$ is specified as in Assumption
\ref{Assump-nondegeneracy}. Due to Markov's inequality and
\eqref{inter-opt-bound}, we also have that $P_{0}(A_{2}) \geq1-(1+\rho
)/C_{0}.$ For any choice of $\rho< 1,$ if we take $C_{0} = 8/p,$ then
$P_{0}(A_{2})\geq3p/4.$ Likewise, due to Assumption \ref{Assump-nondegeneracy}%
, we have $P(A_{1}) \geq p$ for a suitably small $r_{2}.$ If we let $A :=
A_{1} \cap A_{2} \cap A_{3},$ then it follows from union bound that $P(A) \geq
p/2.$

Moreover, we have from \eqref{LB-I1} that
\begin{align}
\frac{I_{1}(\tilde{\beta},\lambda;x)}{\sqrt{g^{2}\tilde{\beta}^{T}%
A(x)^{-1}\tilde{\beta}}} \geq\left(  c_{2} \rho_{\min}^{1/2} - \sqrt{\delta
}\left(  \frac{k_{1}}{c_{1}}\frac{(2\rho_{\max})^{1/2}}{(p\rho_{\min})^{1/2}}+
(1+k_{2})(8/p)^{1/2} \right) \right) ,\label{bd-I1}%
\end{align}
whenever $x \in A$ and $(\beta,\lambda) \in\mathcal{N}_{\underline{r}}%
(\beta,\lambda_{\ast}(\beta)).$ For any fixed $\delta< \delta_{2},$ it follows
from the definition of $\delta_{2}$ that $I_{1}(\tilde{\beta},\lambda;x) > 0$
for all $x \in A,$ and $(\tilde{\beta},\lambda) \in\mathcal{N}_{\underline{r}%
}(\beta,\lambda_{\ast}(\beta)).$ Since $\varphi_{g}(\tilde{\beta},\lambda;x)$
is positive whenever $I_{1}(\tilde{\beta},\lambda;x)$ is positive, we have
(from Proposition \ref{prop:hessian-lrob}) that the Hessian $\nabla^{2}%
\ell_{rob}(\tilde{\beta},\lambda;x)$ exists and it satisfies,
\[
\nabla^{2}\ell_{rob}(\tilde{\beta},\lambda;x) -  \Lambda(\tilde{\beta}%
,\lambda;x)B(x) \succeq0,
\]
for $x \in A$ and $(\tilde{\beta},\lambda) \in\mathcal{N}_{\underline{r}%
}(\beta,\lambda_{\ast}(\beta)).$ Following the same reasoning as in the proof
of Theorem \ref{Thm-Str-Convexity}, one can obtain upper bound $C_{2}$ for
$\Vert\bar{T}_{g}(x)\Vert.$ Moreover, due to (\ref{inter-opt-bound}) and the
property that $2\lambda\vert g(\tilde{\beta},\lambda;x) \vert\geq\vert
\ell^{\prime}(\tilde{\beta}^{T}x)\vert\geq c_{1} $ for $x \in A,$ we have that
$g^{2}(\tilde{\beta},\lambda;x)\tilde{\beta}^{T}A(x)^{-1}\tilde{\beta}$ is
bounded away from zero, for every $x \in A.$ Utilizing these observations and
the bounds for $I_{1},$ $I_{2}$ (see (\ref{bd-I1}) and (\ref{UB-I2})) in the
expression for $\Lambda(\cdot)$ in \eqref{LB-Lambda-noncompact}, we arrive at
the following conclusion: For any $x \in A,$ there exists $\kappa(x) > 0$ such
that
\begin{align*}
\ell_{rob}(\alpha\theta_{1} + (1-\alpha)\theta_{2};x) \leq\ell_{rob}%
(\theta_{1};x) + (1-\alpha)\ell_{rob}(\theta_{2};x) - \frac{\kappa(x)}{2}
\alpha(1-\alpha)\Vert\theta_{1} - \theta_{2} \Vert^{2},
\end{align*}
for $\theta_{1},\theta_{2} \in\mathcal{N}_{\underline{r}}(\beta,\lambda_{\ast
}(\beta)).$ Since $f_{\delta}(\theta) := E_{P_{0}}\left[  \ell_{rob}%
(\beta,\theta;X)\right] ,$ taking expectations on both sides, we arrive at the
conclusion that
\begin{align*}
f_{\delta}(\alpha\theta_{1} + (1-\alpha)\theta_{2};x) \leq f_{\delta}%
(\theta_{1};x) + (1-\alpha)f_{\delta}(\theta_{2};x) - E\left[ \frac{\kappa
(X)}{2}\mathbb{I}(X \in A) \right]  \alpha(1-\alpha)\Vert\theta_{1} -
\theta_{2} \Vert^{2},
\end{align*}
for all $\theta_{1},\theta_{2} \in\mathcal{N}_{\underline{r}}(\beta
,\lambda_{\ast}(\beta)).$ With $P_{0}(A) \geq p/2$ being positive, we have
that the constant $\kappa:= E[\kappa(X)\mathbb{I}(X \in A)]$ is positive as
well. This concludes the proof of Theorem \ref{result-noncompact}.
\Halmos \endproof

\proof{\textbf{Proof of Proposition \protect\ref{Prop-sDelta-subgradients}.}}
	Let $g(\beta,\lambda;X)$ be such that $g(\beta, \lambda;X) \in
	\Gamma^\ast(\beta,\lambda;X)$ and
	\begin{align}
	h(\beta,\lambda;X) =
	\begin{pmatrix}
	L^\prime\tilde{X} \\
	\sqrt{\delta}\left( 1-g^2(\beta,\lambda;X)\beta^TA(X)^{-1}\beta \right)%
	\end{pmatrix}
	\quad\quad P_0-\text{a.s.,}  \label{h-subgrad-proof}
	\end{align}
	where $\tilde{X} := X + \sqrt{\delta}g(\beta,\lambda;X)A(X)^{-1}\beta$ and $L^\prime$ is any arbitrary (measurable) choice from the subgradient interval $\partial \ell(\beta^T \tilde{X}).$ Since $\ell(\cdot)$ is convex, with at most
	quadratic growth (see Assumption \ref{Assump-Obj-Conv}), there exist
	positive constants $C_0,C_1$ such that $\vert L^\prime \vert \leq C_0
	+ C_1\vert \beta^T\tilde{X} \vert.$ As $E\Vert X \Vert^2 < \infty,$ it follows from Lemma \ref{Lem-Finite-Exp}a that $E[g^2(\beta,\lambda;X)], E\Vert \tilde{X}\Vert^2, E\vert \beta^T \tilde{X} \vert$ are finite. Then due to
	Cauchy-Schwartz inequality, we have that $E h(\beta,\lambda;X)$ is
	well-defined. Here we have used that $\beta^TA(x)^{-1}\beta$ is bounded for $%
	P_0-$almost every $x$ (see Assumption \ref{Assump-c}b). Now, since $%
	h(\beta,\lambda;X) \in D(\beta,\lambda;X)  = \partial \ell_{rob} \xspace(\beta,\lambda),$ we have that
	\begin{align*}
	\ell_{rob} \xspace(\beta^\prime,\lambda^\prime;X) \geq \ell_{rob} \xspace%
	(\beta,\lambda;X) + h(\beta,\lambda;X)^T
	\begin{pmatrix}
	\beta^\prime -\beta \\
	\lambda^\prime - \lambda%
	\end{pmatrix}%
	, \quad\quad P_0-\text{a.s.}
	\end{align*}
	Taking expectations on both sides of the above inequality, we obtain $%
	Eh(\beta,\lambda;X) \in \partial f_{\delta}(\beta,\lambda).$
\hfill \Halmos
\endproof
\proof{\textbf{Proof of Lemma \protect\ref{Lem-Grad-Sec-Moment}.}}
	Let $g(\beta,\lambda;X)$ be such that $g(\beta, \lambda;X) \in
	\Gamma^\ast(\beta,\lambda;X)$ and the given subgradient $h(\beta,\lambda;X)$
	is defined as in \eqref{h-subgrad-proof} in terms of $g(\beta,\lambda;X)$
	and $\tilde{X} := X + \sqrt{\delta}g(\beta,\lambda;X)A(X)^{-1}\beta.$ As in
	the proof of Proposition \ref{Prop-sDelta-subgradients}, we have $\vert
	\ell^\prime(u) \vert \leq C_0 + C_1\vert u \vert$ as a consequence of
	convexity, continuous differentiability and at most quadratic growth of $%
	\ell(\cdot).$ Since $E\Vert X \Vert^4 < \infty,$ $\beta^TA(x)^{-1}\beta$ is
	bounded for $P_0-$almost every $x,$ $\lambda > \lambda_{thr} \xspace(\beta)
	+ \eta,$ and $\Vert \beta \Vert \leq R_\beta,$ it follows from Lemma \ref%
	{Lem-Finite-Exp}a that $E[g^4(\beta,\lambda;X)], E\Vert \tilde{X}\Vert^4$
	and $E[\ell^\prime(\beta^T\tilde{X})^4]$ are all uniformly bounded for every
	$(\beta,\lambda) \in U_\eta.$ Then due to Cauchy-Schwartz inequality, we have
	that $\sup_{(\beta,\lambda) \in \mathbb{U}_\eta}E \Vert h(\beta,\lambda;X)
	\Vert^2 < \infty.$
	\hfill \Halmos
\endproof

\proof{\textbf{Proof of Proposition
\protect\ref{Prop-Conv-SGD-Over-All}}.} 
We use $\varepsilon_k$ to denote the estimation error of $
\nabla_{\theta}\ell_{rob} \xspace(\theta_{k-1};X_{k})$ induced by line search. Due to Lemma \ref{lemma-line-search-bias}, there exists a constant $C > 0$ such that the estimation error can be controlled as $\|\varepsilon_k\|\leq C \alpha_k$, if we apply bisection method for at least $\log_2(\alpha_k^{-1}(1+\|X_k\|)^2)$ steps.

We first show that $\|\theta_k-\theta_\ast\|\rightarrow 0$ with probability one.
The update step is
\begin{align}
\theta_{k} = \Pi_{\mathbb{W}} \big(\theta_{k-1} - \alpha_{k}
\nabla_{\theta}\ell_{rob} \xspace(\theta_{k-1};X_{k})- \alpha_{k}\varepsilon_k\big).
\label{eqn-parameter-update-in-proof}
\end{align}
Since the following conditions are satisfied:
\begin{itemize}
    \item[i)] $\sup_{\theta\in
    \mathbb{W}} E\Vert\nabla_{\theta}\ell_{rob} \xspace(\theta;X_k)+\varepsilon_k\Vert^{2} <
    \infty,$ due to Lemma \ref{Lem-Grad-Sec-Moment} and the boundedness of $\varepsilon_k$;
    \item[ii)] $\theta\mapsto \nabla_{\theta}f_\delta(\theta) = E [\nabla_{\theta}\ell_{rob} \xspace(\theta;X)]$ is continuous ;
    \item[iii)] $\sum \alpha^2_k < \infty$ and $\sum \alpha_k\|\varepsilon_k\| < \infty$
\end{itemize}
Applying \cite[Theorem 5.2.1]{kushner2003stochastic}, we conclude that $\|\theta_k-\theta_\ast\|\rightarrow 0$ with probability one.

Now we show that $f_\delta(\theta_k)-f_\ast = O_p(\alpha_k)$. In view of the second order differentiability of $f_\delta$ at $\theta_\ast$, it is sufficient to show that
$\|\theta_k - \theta_\ast\| = O_p(\sqrt{\alpha_k})$, i.e., the sequence $\{\|\theta_k-\theta_\ast\|/\sqrt{\alpha_{k}}\}$ is tight. To this end, we shall slightly modify the proof of \cite[Theorem 10.4.1]{kushner2003stochastic} by allowing a small bias term $\varepsilon_k$ appears in each update step, as shown in \eqref{eqn-parameter-update-in-proof}. As discussed in the proof of \cite[Theorem 10.4.1]{kushner2003stochastic}, since $\|\theta_k-\theta_\ast\| \rightarrow 0$ with probability one, given any small $\nu > 0$, there is an $N_{\nu,\rho}$ such that $\|\theta_k-\theta_\ast\| \leq \rho$ for $k \geq N_{\nu,\rho}$ with
probability larger than $1-\nu$. By shifting the time origin by $N_{\nu,\rho}$, we can suppose that $\|\theta_k-\theta_\ast\| \leq \rho$. If for every $\nu > 0$ the time-shifted sequence is shown to be tight, then the original sequence is tight. Thus for the purposes of the tightness proof, it can be supposed without loss of
generality that $\|\theta_k-\theta_\ast\| \leq \rho$ for all $k$ for the original process, where $\rho > 0$ is arbitrarily small. 

To analyze the convergence rate of $\{\theta_k\}$, we define the Lyapunov function $V(\theta) = \|\theta-\theta_\ast\|^2$. Since the region $\mathbb{W}$ is convex, we have 
\begin{align*}
V(\theta_k)
&\leq \|\theta_{k-1} -\theta_\ast- \alpha_{k}
\nabla_{\theta}\ell_{rob} \xspace(\theta_{k-1};X_{k})- \alpha_{k}\varepsilon_k\|^2
\end{align*}
Notice that $\{\theta_k\}$ is a Markov Chain, by taking conditional expectation on both side we have the following inequalities hold with probability one,
\begin{align*}
E[V(\theta_k)|\theta_{k-1}]
&\leq V(\theta_{k-1}) -2\alpha_{k}
(\theta_{k-1}-\theta_\ast)^{T}\nabla_\theta
f_\delta(\theta_{k-1}) 
\\
&\qquad -\alpha_k^2 E\|\nabla_{\theta}\ell_{rob} \xspace(\theta_{k-1};X_{k})+\varepsilon_k\|^2
-2\alpha_k(\theta_{k-1}-\theta_\ast)^{T}\varepsilon_k\\
&\leq V(\theta_{k-1}) -2\alpha_{k}
(\theta_{k-1}-\theta_\ast)^{T}\nabla_\theta
f_\delta(\theta_{k-1})+ O(\alpha_k^2),
\end{align*}
where the second inequality is due to $\|\theta_k-\theta_\ast\| \leq \rho, \|\varepsilon_k\| \leq C\alpha_k,$ and $\sup_{\theta\in
\mathbb{W}} E\Vert\nabla_{\theta}\ell_{rob} \xspace(\theta;X_k)+\varepsilon_k\Vert^{2} <\infty.$ Since $\rho>0$ can be arbitrarily small, applying Taylor expansion we have
$$
(\theta_{k-1}-\theta_\ast)^{T}\nabla_\theta
f_\delta(\theta_{k-1})
= (\theta_{k-1}-\theta_\ast)^{T}
\nabla^2_\theta
f_\delta(\theta_\ast)(\theta_{k-1}-\theta_\ast)+ o(1)V(\theta_{k-1}).
$$
Thus for any $\rho > 0$ that is smaller than the smallest eigenvalue of $\nabla^2_\theta
f_\delta(\theta_\ast)$, we have
$$
E[V(\theta_k)|\theta_{k-1}]
-V(\theta_{k-1})\leq 
-\rho\alpha_k V(\theta_{k-1})
+O(\alpha_k^2).
$$
Recall that $\alpha_k = \alpha k^{-\tau}$ as required by Assumption \ref{Assump-step-sizes}, if either of the following is true
(i) $\tau = 1$ and $\rho\alpha > 1$; (ii) $\tau \in [1/2,1)$; then we have $E\|\theta_k-\theta_\ast\|^2 = E\|V(\theta_k)\| = O(\alpha_k)$ \cite[Proof of Theorem 10.4.1]{kushner2003stochastic}, which implies that $\|\theta_k-\theta_\ast\| = O_p(\sqrt{\alpha_k})$ and $f_\delta(\theta_k) - f_\ast = O_p(\alpha_k).$
\hfill\Halmos
\proof{\textbf{Proof of Proposition \protect\ref{Prop-Conv-SGD-2}.}} Due to
the characterization of subgradients $\partial f_{\delta}(\beta,\lambda)$ in
Proposition \ref{Prop-sDelta-subgradients}, we have that $\partial
_{+}f_{\delta}/\partial\lambda(\beta,\lambda)\leq\sqrt{\delta}$ for every
$(\beta,\lambda)\in\mathbb{U}.$ Recalling the definitions of $\mathbb{U}%
_{\eta}$ in \eqref{Defn-Ueta} and $\mathbb{U}$ in \eqref{Defn-U}, the above
reasoning leads to concluding that, $f_{\delta}(\beta,\lambda+\varepsilon
)-f_{\delta}(\beta,\lambda)\leq\varepsilon\sqrt{\delta}$ for any
$\varepsilon>0.$ Let $(\beta_{\ast},\lambda_{\ast})\in\inf_{(\beta,\lambda
)\in\mathbb{U}}f_{\delta}(\beta,\lambda).$ Then
\begin{equation}
\inf_{\theta\in\mathbb{U}_{\eta}}f_{\delta}(\theta)-f_{\ast}\leq f_{\delta
}(\beta_{\ast},\lambda_{\ast}+\eta)-f_{\delta}(\beta_{\ast},\lambda_{\ast
})\leq\eta\sqrt{\delta}\label{Inter-Prop-SGD-2}%
\end{equation}
It follows from Lemma \ref{Lem-Grad-Sec-Moment} that $\sup_{k}E_{P_{0}}\Vert
H_{k}\Vert^{2}<G_{\eta}.$ As a consequence, we have from Theorem 2 and the
remark following Theorem 4 in \cite{Shamir:2013:SGD:3042817.3042827} that
$E[f_{\delta}(\theta_{k})]-\inf_{\theta\in\mathbb{U}_{\eta}}f_{\delta}%
(\theta)=O_p(k^{-1/2}\log k)$ and $E[f_{\delta}(\bar{\theta}_{k})]-\inf
_{\theta\in\mathbb{U} _{\eta}}f_{\delta}(\theta)=O_p(k^{-1/2}).$ Combining this
with the observation in \eqref{Inter-Prop-SGD-2}, we obtain that $E[f_{\delta
}(\bar{\theta}_{k})]-f_{\ast}\leq\eta\sqrt{\delta}+O_p(k^{-1/2}).$ As in the
proof of Proposition \ref{Prop-Conv-SGD-1}a, the conclusion in Proposition
\ref{Prop-Conv-SGD-2} follows as a consequence of Markov's inequality.
\hfill\Halmos \endproof

\section{Line search scheme.}
\label{Sec-Line-Search}
Our iterative procedure requires evaluating
\begin{equation*}
\ell _{rob}\xspace(\beta ,\lambda ;x)=\sup_{\gamma \in \mathbb{R}}F(\gamma
,\beta ,\lambda ;x)
\end{equation*}
and obtaining a maximizer $\gamma ^{\ast }=g(\beta ,\lambda ;x)\in \Gamma
^{\ast }(\beta ,\lambda ;x)$. This task involves a one dimensional
optimization problem over $\gamma $. This problem, we claim, can be solved
through a line search. This can be done efficiently on a case-by-case basis
given $\ell (\cdot )$ (as we do in our numerical examples). However, our
goal here is to provide reasonably general conditions which can be used to
efficiently implement a line search procedure to compute $\ell _{rob}\xspace%
(\beta ,\lambda ;x)$.
Unfortunately, however, the function $F(\cdot ,\beta ,\lambda ;x)$ is not
necessarily concave. So, to show that the line search can be implemented
efficiently, we need to use study the definition of $F\left( \cdot \right) $
and introduce assumptions on $\ell (\cdot )$, which we believe are
reasonable.
The general line search scheme is easy to develop for $\lambda$ small or
large enough. Recall that
\begin{enumerate}
	\item When $\lambda<\lambda_{thr}(\beta)$, the dual objective $%
	f_{\delta}(\beta,\lambda) = \infty$, so the line search algorithm will not be
	executed in this case.
	
	\item When $\lambda \geq \lambda _{thr}^{\prime }(\beta )$, the function $%
	F(\cdot ,\beta ,\lambda ;x)$ is concave for $P_{0}-$almost every $x$.
	Consequently, finding $g(\beta ,\lambda ;x)$ is a convex optimization
	problem, and therefore can be solved by bisection method or
	Newton-Raphson method. We will always be in this case if $\delta < \delta_0$.
	
	\item The third case is the most challenging case, as we shall explain. It requires imposing smoothness assumptions on our loss function.
\end{enumerate}
\begin{lemma}
\label{lemma-line-search-bias}
Suppose Assumptions \ref{Assump-c} - \ref{Assump-Compact} are
satisfied then $\delta < \delta_0$. In turn, this implies that we are in case 2. above \emph{(}i.e. $\lambda \geq \lambda _{thr}^{\prime }(\beta )$\emph{)} and a unique maximizer
with $|g(\beta,\lambda,x)|\leq |\ell^\prime(\beta^Tx)|/(\varphi_{\min}\|\beta\|)$ exist for $P_0$-almost surely $x$ due to Proposition \ref{prop:secondorder-info-smoothcase} b). Then, there exist constants $C_1$ and $C_2$, such that for every $(\beta,\lambda)\in\mathbb{W}$ and $x\in\mathbb{R}^d$, it is required at most $\log_2(C_1\varepsilon^{-1}\|\beta\|^{-1}(1+\|x\|))$ steps for binary search method to solve for a solution $\gamma$ such that $|\gamma-\gamma^\ast|\leq \varepsilon$. In turn, the binary search procedure generates estimates for $\partial \ell_{rob}/\partial \beta$ and $\partial \ell_{rob}/\partial \lambda$ with $\varepsilon$ accuracy and $\log_2(C_2\varepsilon^{-1}(1+\|x\|)^2)$ steps. (The explicit construction of the derivative estimates is summarized in Appendix \ref{sec-algorithm}.)  
\end{lemma}
\proof{\textbf{Proof.}}
When solving the one dimensional optimization problem, $\max_{\gamma} F(\gamma,\beta,\lambda;x)$, we consider a scaled problem by setting $\tilde\gamma = \|\beta\|\gamma$ and it suffices to consider a scaled problem $\max_{\tilde\gamma} F(\tilde\gamma,\beta,\lambda;x)$, for $|\tilde\gamma|\leq \varphi_{\min}^{-1}|\ell^\prime(\beta^Tx)|.$ Note that there exist some constant $C$, independent of $\beta,\lambda, x$, such that the bound $\varphi_{\min}^{-1}|\ell^\prime(\beta^Tx)|\leq \varphi_{\min}^{-1}(|\ell^\prime(0)| + M\|\beta\|\|x\|)\leq C(1+\|x\|)$
	. Notice that the function $F(\cdot,\beta,\lambda;x)$ is concave by Lemma \ref{Lem-Str-conv-Thr}, so using binary search method to solve for the optimal $\tilde\gamma$ up to an $\varepsilon$ error requires at most $\log_2(C\varepsilon^{-1}(1+\|x\|))$ steps. Finally, consider the inverse scaling $\gamma = \|\beta\|^{-1}\tilde{\gamma}$, the error of optimal $\tilde{\gamma}$ need to be bounded by $ \|\beta\|\varepsilon$ in order to get $\gamma$ with error bounded by $\varepsilon$. Thus an $\varepsilon$-accuracy solution for $\gamma$ requires at most
	$\log_2(C\varepsilon^{-1}\|\beta\|^{-1}(1+\|x\|))$ steps.
	
	Now we consider the error of $\partial \ell_{rob}/\partial \beta$ and $\partial \ell_{rob}/\partial \lambda$ induced by the error of line search. Recall from Proposition \eqref{prop:firstorderinfo} that
	$\partial \ell_{rob}/\partial \beta = \ell'(\beta^Tx + \sqrt{\delta}g\beta^T A(x)^{-1} \beta)(x + \sqrt{\delta}g A(x)^{-1} \beta)$ and $\partial \ell_{rob}/\partial \lambda = -\sqrt{\delta }\left( g^2 \beta^TA(x)^{-1}\beta - 1\right)$. For every $(\beta,\lambda)\in\mathbb{W}$ and $x\in\mathbb{R}^d$, there exist some constant $L$ uniform in $\beta,\lambda,x$, such that $g\mapsto \partial \ell_{rob}/\partial \theta = (\partial \ell_{rob}/\partial \beta,\partial \ell_{rob}/\partial \lambda)$ for $g\leq C\|\beta\|^{-1}(1+\|x\|)$ is Lipschitz continuous with Lipschitz constant $L\|\beta\|(1+\|x\|)$. Consequently, in order to get an $\varepsilon$-accuracy evaluation for $\partial \ell_{rob}/\partial \theta$, we need to solve for an $\varepsilon L^{-1}\|\beta\|^{-1}(1+\|x\|)^{-1}$-accuracy solution for $g$. Consequently the bisection method is required to run for at most $\log_2(CL\varepsilon^{-1}(1+\|x\|)^2)$ steps.
\hfill \Halmos
\endproof
It then remains to discuss case 3. Namely, to develop an algorithm to compute $g(\beta ,\lambda ;x)$
when $\lambda \in \lbrack \lambda _{thr}(\beta ),\lambda _{thr}^{\prime
}(\beta ))$, which, requires a more delicate analysis.
The following example shows that the function $F(\cdot ,\beta ,\lambda ;x)$
can have infinitely many local optima.
\begin{example}
	\label{Eg-Inf-Opt} Suppose that $\beta \neq \mathbf{0}$, $P_{0}(\cdot
	)=\delta _{\{\mathbf{0}\}}(\cdot )$ and $\ell (u)=u^{2}-\cos u$. It then
	follows that $\kappa =1$ and $\lambda _{thr}(\beta )=\sqrt{\delta }\beta
	^{T}A(\mathbf{0})^{-1}\beta $. Thus, $F(\gamma ,\beta ,\lambda _{thr}(\beta
	);\mathbf{0})=-\cos \left( \sqrt{\delta }\beta ^{T}A(\mathbf{0})^{-1}\beta
	\gamma \right) $, which has infinitely many local optima.
\end{example}
So, to solve the global nonconvex optimization problem, it is necessary to
reduce the feasible region of optimization problem to a compact interval. To
this end, we consider the scaled line search problem $\max_{\bar{\gamma} \in
	\mathbb{R}}F(\bar{\gamma}\beta ^{T}A(x)^{-1}\beta,\beta ,\lambda ;x)$,
instead of considering the original line search problem $\max_{\gamma \in
	\mathbb{R}}F(\gamma ,\beta ,\lambda ;x)$. In the following Lemma, we show
that when $(\beta,\lambda)\in\mathbb{U}_{\eta}$, it suffices to consider the
scaled line search problem with a compact feasible region.
\begin{lemma}
	\label{Lem-line-compact} Recall the definition of $\mathbb{U}_\eta$ from %
	\eqref{Defn-Ueta} and suppose that Assumption \ref{Assump-c}-\ref%
	{Assump-Obj-2diff} hold and $\eta>0$. Then there exist a random
	variable $R$ with $E_{P_0}[R^2]<\infty$, such that
	\begin{equation*}
	|g\beta^TA(X)^{-1}\beta|\leq R
	\end{equation*}
	for any $(\beta,\lambda)\in \mathbb{U}_{\eta}$ and $g\in
	\Gamma^{\ast}(\beta,\lambda;X)$.
\end{lemma}
\proof{Proof.}
	The fact that $(\beta,\lambda)\in\mathbb{U}_\eta$ implies that $\lambda\geq
	\lambda_{thr} \xspace(\beta)+ \eta$. Then, according to the Assumptions we
	have $\beta^TA(X)^{-1}\beta\leq \rho_{\min}^{-1}R_\beta^2.$ Thus, letting $%
	\varepsilon = \eta\delta^{-1/2}\rho_{\min}R^{-2}_{\beta}$, we have $%
	\lambda\geq(\kappa + \varepsilon)\sqrt{\delta}\beta^TA(x)^{-1}\beta$, and
	thus the result of Lemma \ref{Lem-Finite-Exp}a can be applied. As a result,
	there exist a constant $C_1$ such that
	\begin{equation*}
	\vert g \beta^TA(X)^{-1}\beta\vert \leq 1 + C_1 \varepsilon^{-1}(1 + \vert
	\beta^TX \vert)=:R
	\end{equation*}
	and the squared integrability of $R$ is easy to verified.
\hfill \Halmos
\endproof
With the help of Lemma \ref{Lem-line-compact}, we know it suffices to
consider the scaled line search problem $\max_{\bar{\gamma} \in [-R,R]}F(%
\bar{\gamma}\beta ^{T}A(X)^{-1}\beta,\beta ,\lambda ;X)$ with a bounded
feasible region $[-R,R]$, where the length of interval $2R$ is squared
integrable, controlling the average complexity of the line search.
Next, we need to rule out the pathological case that the stationary points
of $\gamma\mapsto F(\gamma\beta^TA(x)^{-1}\beta ,\beta ,\lambda ;x)$ in $%
[-R,R]$ contain infinitely many connected components. To this ends, we
further impose an assumption that $\ell (\cdot )$ is \emph{piecewise real
	analytic} in any compact set $K$.
A function $f$ is \emph{real analytic} on an open set $D$ if for any $%
x_{0}\in D$ one can write $f(x)=\sum_{n=0}^{\infty }a_{n}\left(
x-x_{0}\right) ^{n}$, in which the coefficients $a_{n}$ are real numbers and
the series is convergent to $f(x)$ for $x$ in a neighborhood of $x_{0}$.
A function $f$ is \emph{piecewise real analytic} in a compact set $K$ if
there exist $n\in \mathbb{N}$ and closed intervals $D_{1},\ldots ,D_{n}$,
such that $K\subset \bigcup_{i=1}^{n}D_{i}$, and for each $D_{i}$, the
restriction of $f$ on $D_{i}$ has a real analytic extension. In other words,
for each set $D_{i}$, there exists an open set $D_{i}\subset \tilde{D}_{i}$
and a real analytic function $g_{i}$ on $\tilde{D}_{i}$, such that $%
f(x)=g_{i}(x)$ for all $x\in D_{i}$.
\begin{lemma}
	Suppose that $f$ is piecewise real analytic in compact set $K$, then the
	stationary points of $f$ in $K$ are contained in only finitely many
	connected components.
\end{lemma}
\proof{Proof.}
	If a connected component of stationary points is not a discrete point, then
	it must contains an open interval that disjoint with the remaining connected
	components. Thus, as the set $K$ is compact, the total number of
	non-singleton connected components is finite.
	
	It remains to prove the number of discrete stationary points of $f$ is
	finite. To this end, it suffices to prove $g_{i}$ has finite discrete
	stationary points in $D_{i}$. We claim that there does not exist an
	accumulation point of discrete stationary points of $g_{i}$ in set $D_{i}$.
	Otherwise, we can find a sequence of discrete stationary points $%
	\{x_{n}:n\geq 1\}$, and $x_{n}$ converge to a point $x\in D_{i}$. Consider
	the Taylor series of the function $g_{i}$ around $x$, if the Taylor series
	is zero except the constant term. Then by the real analytic property, the
	function is a constant in a neighborhood around $x$, violating the
	assumption that all the $x_{n}$ are discrete stationary points. If the
	Taylor series has non-zero higher order terms, then there exist a
	neighbourhood of $x$ such that $x$ is the only stationary point in that
	neighborhood, violating the assumption that $x_{n}$ converge to $x$. So the
	discrete stationary points of $g_{i}$ does not have an accumulation point in
	$D_{i}$. As a result, we can find an open cover of $D_{i}$ such that each
	open set in the open cover contains at most one discrete stationary point of
	$g_{i}$. Since $D_{i}$ is also compact, $g_{i}$ has finite discrete
	stationary points in $D_{i}$. The result follows.
\hfill \Halmos
\endproof
Note that it is important for the series to be absolutely and uniformly
convergent; smoothness alone does not imply the existence of finitely many
stationary points on a compact interval, as the next example shows.
\begin{remark}
	Even if a function is in $C^{\infty}(\mathbb{R})$, it may have infinitely many isolated
	local optima on a compact set. Consider
	\begin{equation*}
	f(x):=%
	\begin{cases}
	\cos \left( -(1-x^{2})^{-1}\right) \exp \left( -(1-x^{2})^{-1}\right) & %
	\mbox{if }-1<x<1, \\
	0 & \mbox{otherwise.}%
	\end{cases}%
	\end{equation*}
\end{remark}
Now we discuss the line search scheme and its complexity. If the loss
function $\ell (\cdot )$ is piecewise real analytic, the function $\gamma\mapsto F(\gamma
,\beta ,\lambda ;x)$ is also piecewise real analytic. In addition, using the
result of Lemma \ref{Lem-line-compact}, the optimization problem $%
\max_{\gamma \in \mathbb{R}}F(\gamma ,\beta ,\lambda ;x)$ is equivalent to
the problem $\max_{\bar{\gamma}\in \lbrack -R,R]}F(\bar{\gamma}\beta
^{T}A(x)^{-1}\beta ,\beta ,\lambda ;x)$, a one dimensional optimization
problem with compact feasible region. We denote the closed intervals
partitioning $[-R,R]$ by $D_{1},\ldots ,D_{n}$. Thus, $F(\bar{\gamma}\beta
^{T}A(x)^{-1}\beta ,\beta ,\lambda ;x)$ has finite local optimal points in
compact interval $[-R,R]$, which are either stationary points in the
interior of a interval, or a hinge point connecting two adjacent intervals.
One possible approach for computing stationary points of a real analytic
function is to consider the holomorphic extension of the function and then
apply Cauchy's theorem (see, for example, \cite%
{dellnitz2002locating,delves1967numerical}). This approach is guaranteed to locate all of the stationary points. However, the use of Cauchy's theorem requires the evaluation of certain integrals in smooth trajectories. The evaluation of these trajectories can be done with high precision integration rules which take advantage of the analytic properties of the integrands, evaluating $o\left(
\varepsilon ^{-\delta }\right) $ (for any $\delta >0$) points in the
integrand to achieve a $\varepsilon $ relative error, for example, applying
Newton integration rules.
The complexity of finding all the stationary points of function $\bar{\gamma}%
\mapsto F(\bar{\gamma}\beta ^{T}A(x)^{-1}\beta ,\beta ,\lambda ;x)$ is
proportional to $2R$, the length of the searching interval. Therefore, the
total complexity of the line search scheme is $O_{p}(\varepsilon ^{-\delta
}) $, for any $\delta >0$, uniformly for all $(\beta ,\lambda )\in \mathbb{U}%
_{\eta }$. This complexity includes the evaluation of the global maxima by
comparing the value of $F(\gamma ,\beta ,\lambda ;x)$ at all local optimal
points.

Another approach, instead of using Cauchy's theorem, applying Newton's method repeatedly. Because we have established that there are finitely many roots if the loss is piecewise real analytic, by restarting Newton's method from randomly chosen initial conditions we will be able to locate, with an exponential decaying error rate in the number of retrials, the global optimum. While this algorithms is easy to implement, its analysis is of independent interest and too long to include in this paper. So, we will discuss it in future work. 

\section{Analytical solution to projection $\Pi_{\mathbb{W}}$.}
\label{Sec-Analytic-Projection}
The algorithms described in Sections \ref{Sec-First-SGD} and \ref{Sec-Mod-SGD} requires to project the variables $\theta = (\beta, \lambda)$ to the set $\mathbb{W}.$ In this section we provide an analytical solution to the projection.
Recall that the definition of $\mathbb{W}$ as
\begin{align*}
	\mathbb{W} = \left\{(\beta,\lambda)\in B\times\mathbb{R}\mid K_1\|\beta\|\leq \lambda\leq K_{2}R_{\beta} \right\}.
\end{align*}
Suppose that $\theta = (\beta, \lambda) \in B\times \mathbb{R}$. The projection $\Pi_{\mathbb{W}}(\theta) = (\beta',\lambda')$ is determined by the convex optimization problem
\begin{align*}
	\min\quad& \|(\beta,\lambda) - (\beta',\lambda')\|^2\\
	\mbox{s.t.}\quad & K_1\|\beta'\|\leq \lambda'\leq K_{2}R_{\beta} ,\\
	& (\beta',\lambda')\in B\times\mathbb{R}.
\end{align*}
Thus, analytical solution of the projection $\Pi_{\mathbb{W}}(\theta)$ is given by
\begin{align*}
	\Pi_{\mathbb{W}}(\theta) = \begin{cases}
	(\beta,\lambda)& \mbox{if }(\beta,\lambda)\in\mathbb{W},\\
	(\beta,K_{2}R_{\beta})& \mbox{if } \|\beta\|\leq K_{2}R_{\beta}/K_1, \lambda > K_{2}R_{\beta},\\
	(\mathbf{0},0)& \mbox{if } \lambda < -\|\beta\|/K_1,\\
	\left(\frac{\|\beta\|+\lambda}{1+K_1^2}\frac{\beta}{\|\beta\|}, \frac{K_1\|\beta\|+K_1^2\lambda}{1+K_1^2}\right)& \mbox{if }-\|\beta\|/K_1\leq \lambda <\min\{K_1\|\beta\|, K_{2}R_{\beta}(1+K_1^{-2})-\|\beta\|/K_1\},\\
	\left(\frac{K_{2}R_{\beta}}{K_1}\frac{\beta}{\|\beta\|}, K_{2}R_{\beta}\right), &\mbox{otherwise}.
	\end{cases}
\end{align*}
The optimality of above solution can be verified using the KKT condition.

\section{Tables specifying useful constants and an illustration of computation of some of these constants in examples}\label{Sec-table}
Tables \ref{tab:constants-given} - \ref{tab:constants-algorithmic framework} below present a compilation of useful constants used in the main text. A demonstration of how some of the relevant constants can be computed is presented in the subsequent sections.

\begin{table}[ht!]
\caption{Constants which are specified as part of the framework}
\begin{center}
\begin{tabular}{ |c|l| }
\hline
constant &\hspace{80pt} description\\
\hline
$\delta$ & the radius of the Wasserstein ball specified in the DRO formulation \eqref{DRO_B0}\\
$\rho_{\max}$ & the largest possible eigenvalue of the matrices $\{A(x): x \in \mathbb{R}^d\}$\\
$\rho_{\min}$ & the smallest possible eigenvalue of the matrices $\{A(x): x \in \mathbb{R}^d\}$\\
$\kappa$  &   quadatic growth rate of $\ell(\cdot)$ characterized by $\kappa := \inf\{s \geq 0: \sup_{u \in \mathbb{R}} \left(\ell(u)-su^2 \right) < \infty\}$\\
$M$ & the maximum possible value for the second derivative $\ell^{\prime\prime}(\cdot),$ whenever it exists\\
$R_\beta$ & equals $\sup_{\beta \in B} \Vert \beta \Vert$ if the set $B$ is taken to be bounded\\
$c_1,c_2,p$& positive constants which ensure $P_0\left( \vert  \ell^\prime(\beta^TX) \vert > c_1, \vert \beta^T X\vert > c_2\Vert \beta \Vert \right) \geq p$\\
$k_1,k_2$ &  positive constants $k_1,k_2$ satisfying $\vert u \vert \ell^{\prime\prime}(u)  \leq k_1 + k_2\vert \ell^\prime(u)\vert$ required by Theorem \ref{result-noncompact}
\\
$K$ & the number of piecewise components in the loss $\ell(u) = \max_{i=1,\ldots,K} \ell_i(u)$\\
\hline
\end{tabular}
\end{center}
\label{tab:constants-given}
\end{table}
\begin{table}[h!]
\caption{Some useful constants which are specified in the analysis}
\begin{center}
\begin{tabular}{c c}
\begin{tabular}{ |c|l| }
\hline
constant & value specified in the  analysis\\
\hline
$\delta_0$ & $\rho_{\min}^2 \underline{L}R_\beta^{-2}M^{-2}\rho_{\max}^{-1}$\\
$\delta_1$ & $\min\{ \delta_0/4, c_1^2c_2^2p^2\rho_{\min}^2\rho_{\max}^{-1}\underline{L}\overline{L}^{-2}/256\}$\\
$\delta_2$ & $c_1^2c_2^2p\rho_{\min}^2\left( 2k_1\rho_{\max}^{1/2}
+ 4c_1\rho_{\min}^{1/2}(1+k_2)\right)^{-2}$\\
$\kappa_{0}$ & $2^{-1}\underline{L}\rho_{\max}^{-1}$ \\
$\kappa_{1}$ & $pC\rho_{\max}^{-1}/2$\\
$K_1$ & $\underline{L}^{1/2}\rho_{\max}^{-1/2}/2$\\
\hline
\end{tabular}  &
\begin{tabular}{ |c|l| }
\hline
constant & value specified in the  analysis\\
\hline
$K_2$ & $\sqrt{\delta}MR_\beta \rho_{\min}^{-1} + \rho_{\min}^{-1/2}\overline{L}$\\
$\lambda_{thr}(\beta)$ & $\sqrt{\delta}\kappa P_0-\text{ess-sup}_x \beta^TA(x)^{-1}\beta$\\
$\lambda^\prime_{thr}(\beta)$ & $\sqrt{\delta}M P_0-\text{ess-sup}_x \beta^TA(x)^{-1}\beta$\\
$\overline{L}$ &  $\max_{\beta \in B} E_{P_0}\left[ \ell^\prime(\beta^TX)^2\right]^{1/2}$\\
$\underline{L}$ & $\min_{\beta \in B} E_{P_0}\left[ \ell^\prime(\beta^TX)^2\right]^{1/2}$\\
$\varphi_{\min}$&  $ \underline{L}^{-1/2}\rho_{\max}^{-1/2} - \sqrt{\delta}R_\beta M\rho_{\min}^{-1}$\\
\hline
\end{tabular}
\end{tabular}
\end{center}
\label{tab:constants-derived}
\end{table}
\begin{table}[h!]
\caption{Constants involved in the SGD algorithm}
\begin{center}
\begin{tabular}{ |c|l| }
\hline
constant & \hspace{80pt} description\\
\hline
$\{\alpha_k: k \geq 1\}$ & step-size sequence satisfying $\alpha_k = \alpha k^{-\tau}$ for some $\alpha >0, \  \tau \in [1/2,1)$ \\
$\eta$ & specifies the set $\mathbb{U}_\eta$ (see \eqref{Defn-Ueta}) onto which iterates are projected\\
$\xi$ & polynomial averaging constant in \eqref{SGD-1}\\
\hline
\end{tabular}
\end{center}
\label{tab:constants-algorithmic framework}
\end{table}
\subsection{Logistic regression} The logistic loss function is given by
\begin{equation*}
	\ell(u;y) = \log\left(1+\exp(-yu)\right)
\end{equation*}
Theorem \ref{Thm-Str-Convexity} and/or \ref{result-noncompact} are applied to analyze the locally strong convexity of $f_{\delta}$; Proposition \ref{Prop-Conv-SGD-1} and \ref{Prop-Conv-SGD-Over-All} guarantees the efficacy of the proposed algorithm.
The constants appearing in the related assumptions can be chosen as:
\begin{itemize}
	\item $\rho_{\min}, \rho_{\max}$: Determined by the selection of $A(x)$. If $A(x)$ is an identity matrix, then $\rho_{\min}= \rho_{\max} = 1$.
	\item $\kappa$: Since the loss function is asymptotically linear, $\kappa  = 0.$
	\item $k_1,k_2$:
    $k_1 = k_2 = 1$, because $\sup_{u\in\mathbb{R}} |u|\ell''(u; \pm1)\leq 1$
	\item $M$: $M = 1/4$, because $\ell''(u; \pm1) = (1/4) \cosh(u/2)^{-2} \leq1/4$, 
	\item $c_1,c_2,p$: Depends on the distribution of $X$. The constants with desired properties exist if and only if $\mathrm{rank}\;\{Y_{i}\cdot X_{i}\}_{i=1}^{n}=n$, which would happen
    almost surely if the data generating distribution of $X$ has a density.
\end{itemize}

\subsection{Linear regression} The squared loss function is given by
\begin{equation*}
\ell(u;y) = \left(u-y\right)^2
\end{equation*}
Theorem \ref{Thm-Str-Convexity} and/or \ref{result-noncompact} are applied to analyze the locally strong convexity of $f_{\delta}$; Proposition \ref{Prop-Conv-SGD-1} and \ref{Prop-Conv-SGD-Over-All} guarantees the efficacy of the proposed algorithm.
The constants appearing in the related assumptions can be chosen as:
\begin{itemize}
	\item $\rho_{\min}, \rho_{\max}$: Determined by the selection of $A(x)$. If $A(x)$ is an identity matrix, then $\rho_{\min}= \rho_{\max} = 1$.
	\item $\kappa$: Since the loss function is quadratic, $\kappa  = 1.$
	\item $k_1,k_2$: $k_1 = \max_{i = 1,\ldots,n} |Y_i|$ and $k_2 = 1$.
	\item $M$: $M = 1$ because of $\ell''(u; y) = 1$, 
	\item $c_1,c_2,p$: Depends on the distribution of $X$. The constants with desired properties exist almost surely if the data generating distribution of $X$ has a density.
\end{itemize}

\subsection{Support vector machines} The hinge loss function is given by
\begin{equation*}
	\ell(u;y) = \max\left( 0,1-yu\right)
\end{equation*}
Proposition \ref{Prop-sDelta-subgradients} and \ref{Prop-Conv-SGD-2} provide theoretical foundation for hinge loss function, in which Assumption \ref{Assump-c} and \ref{Assump-Obj-Conv} are imposed. The related constants can be chosen as:
\begin{itemize}
    \item $\rho_{\min}, \rho_{\max}$: Determined by the selection of $A(x)$. If $A(x)$ is an identity matrix, then $\rho_{\min}= \rho_{\max} = 1$.
	\item $\kappa$: Since the loss function is asymptotically linear, $\kappa  = 0.$
\end{itemize}

\section{Algorithm}\label{sec-algorithm}
\quad
\begin{algorithm}[H]
	\caption{Stochastic Gradient Descent for the case $\delta < \delta_0$}
	\begin{algorithmic}
		\State \textbf{input: } Initial parameter $\bar{\theta}_{0} = \theta_0 = (\beta_0, \lambda_0)\in \mathbb{W}$, step-size sequence $(\alpha_k)_{k\geq 1}$, total number of iterations $N$.
		\For{$k = 1,2,\ldots, N$}
			\State Generate an independent sample $X_k$ from the distribution $P_0$.
			\State Set $n_k\geq \tau\log_2(k)-\log_2(\alpha) + 2\log_2(1+\|X_{k}\|)$ as total cuts for bisection method.
			\State Set $I_k =\big[ -|\ell^\prime(\beta_{k-1}^TX_k)|/(\varphi_{\min}\|\beta_{k-1}\|),	|\ell^\prime(\beta_{k-1}^TX_k)|/(\varphi_{\min}\|\beta_{k-1}\|)\big]$ as the initial interval.
			\State Solve $\gamma_k = \arg\max_{\gamma\in I_k} F(\gamma,\theta_{k-1};X_k)$ using bisection method for $n_k$ steps.
			\State Compute $\tilde{X}_k = X_k + \gamma_k \sqrt{\delta}A(X_k)^{-1}\beta_{k-1}$.
			\State Compute $\nabla_\theta \ell_{rob}(\theta_{k-1}; X_k)$ using the closed-form expression
			$$\frac{\partial \ell_{rob} \xspace}{\partial \beta}%
			(\theta_{k-1};X_k) =\ell^{\prime}\left(\beta_{k-1}^T\tilde{X}_k \right)\tilde{X}_k \quad
			\text{and}\quad \frac{\partial \ell_{rob} \xspace}{\partial \lambda}%
			(\theta_{k-1};X_k) = -\sqrt{\delta }\left( \gamma_k^2 \beta_{k-1}^T A(X_k)^{-1}\beta_{k-1} -
			1\right). $$
			\State Update the parameter by $\theta_{k} := \Pi_{\mathbb{W}} \big(\theta_{k-1} - \alpha_k \nabla_\theta
			\ell_{rob} \xspace(\theta_{k-1};X_k)\big).$
			\State Update the trajectory average $\bar{\theta}_{k} = \left(\frac{k-1}{k}\right)  \bar{\theta}_{k-1} + \frac{1}{k} \theta_{k}.$
		\EndFor
		\State \textbf{output:} $\bar{\theta}_{N}$ and 
		$\theta_{N}$.
	\end{algorithmic}
\end{algorithm}
\end{APPENDICES}


\end{document}